\documentclass{amsart}

\usepackage{amsmath, mathtools}
\usepackage{graphicx, enumerate}
\usepackage{subcaption} 
\usepackage{upgreek} 

\usepackage[hidelinks]{hyperref} 
\allowdisplaybreaks 

\usepackage{tcolorbox}
\tcbset{sharp corners=all, boxrule = 1pt, colback = white, coltext = black!60, colframe = black!40}

\usepackage{tikz}
\usetikzlibrary{calc,matrix}
\newcommand{\epsi}{\varepsilon}

\newcommand{\Q}{\mathbb{Q}}
\newcommand{\Z}{\mathbb{Z}}
\newcommand{\R}{\mathbb{R}}
\newcommand{\Real}{\mathbb{R}}

\newcommand{\Ocal}{\mathcal{O}}
\newcommand{\Lcal}{\mathcal{L}}
\newcommand{\Bcal}{\mathcal{B}}
\newcommand{\Hcal}{\mathcal{H}}

\newcommand{\Lcaldis}{\Lcal_{\text{dis}}}
\newcommand{\Hcaldis}{\Hcal_{\text{dis}}}
\newcommand{\HcalMP}{\Hcal_{\text{mp}}}
\newcommand{\Acal}{\mathcal{A}}

\newcommand{\tp}{\top} 

\newcommand*{\itD}{\varDelta} 
\newcommand*{\dxi}{{\itD\xi}} 

\newcommand{\Htotal}{H_\infty}
\newcommand{\lbf}{\pmb{\ell}}
\newcommand{\linfty}{ { \lbf^{\infty} } }
\newcommand{\ltwo}{ { \lbf^{2} } }
\newcommand{\lsp}[1]{{\lbf^{#1}}}
\newcommand{\lone}{ { \lbf^{1} } }
\newcommand{\lstar}{ { \lbf^{*} } }
\newcommand{\hone}{ \mathbf{h}^{1} }
\newcommand{\Ed}{\mathbf{E}_{\dxi}}
\newcommand{\Ea}{\Ed^\textnormal{aux}}
\newcommand{\Ek}{\Ed^\textnormal{ker}}
\newcommand{\Vd}{\mathbf{V}_{\!\!\dxi}}
\newcommand{\E}{\mathbf{E}}
\newcommand{\V}{\mathbf{V}}
\newcommand{\Ltwo}{\mathbf{L}^{2}}
\newcommand{\Lone}{\mathbf{L}^{1}}
\newcommand{\Linfty}{\mathbf{L}^{\infty}}
\newcommand{\Hone}{\mathbf{H}^1}

\newcommand{\Gcal}{\mathcal{G}}
\newcommand{\Ucal}{\mathcal{U}}
\newcommand{\Kcal}{\mathcal{K}}
\DeclareMathOperator*{\id}{Id}
\newcommand{\Rcal}{\mathcal{R}}
\newcommand{\Qcal}{\mathcal{Q}}

\newcommand{\sumZ}[2]{\sum_{#1\in\Z}#2}
\newcommand{\sumZxi}[2]{\dxi \sum_{#1\in\Z}#2}
\newcommand{\intR}[2]{\int_{\R}#2\,d#1}

\newcommand{\ip}[2]{\left<#1,#2\right>}
\newcommand{\ipl}[2]{\left<#1,#2\right>_{\ltwo}}

\newcommand{\norm}[1]{\left\Vert#1\right\Vert}
\newcommand{\norms}[1]{\left\|#1\right\|}
\newcommand{\abs}[1]{\left|#1\right|}

\newcommand{\normlstar}[1]{\left\|#1\right\|_{\lstar}}

\newcommand{\normlinfty}[1]{\left\|#1\right\|_{\linfty}}
\newcommand{\normslinfty}[1]{\|#1\|_{\linfty}}

\newcommand{\normltwo}[1]{\left\|#1\right\|_{\ltwo}}
\newcommand{\normsltwo}[1]{\|#1\|_{\ltwo}}

\newcommand{\normlone}[1]{\left\|#1\right\|_{\lone}}
\newcommand{\normslone}[1]{\|#1\|_{\lone}}

\newcommand{\normhone}[1]{\left\|#1\right\|_{\hone}}

\newcommand{\norml}[2]{\left\|#2\right\|_{\lsp{#1}}}

\newcommand{\normE}[1]{\left\|#1\right\|_{\E}}
\newcommand{\normV}[1]{\left\|#1\right\|_{\V}}

\newcommand{\normEd}[1]{\left\|#1\right\|_{\Ed}}
\newcommand{\normEa}[1]{\left\|#1\right\|_{\Ea}}
\newcommand{\normEk}[1]{\left\|#1\right\|_{\Ek}}
\newcommand{\normVd}[1]{\left\|#1\right\|_{\Vd}}

\newcommand{\normLone}[1]{\left\|#1\right\|_{\Lone}}
\newcommand{\normLtwo}[1]{\left\|#1\right\|_{\Ltwo}}
\newcommand{\normLinfty}[1]{\left\|#1\right\|_{\Linfty}}
\newcommand{\normHone}[1]{\left\|#1\right\|_{\Hone}}

\newcommand{\diff}[2]{\frac{d #2}{d #1}}
\newcommand{\pdiff}[2]{\frac{\partial #2}{\partial#1}}
\newcommand{\vdiff}[2]{\frac{\delta #2}{\delta#1}}
\newcommand{\fracpar}[2]{\frac{\partial #1}{\partial #2}}
\newcommand{\fracpart}{\fracpar{}{t}}

\newcommand*{\A}[1]{\mathrm{A}[{#1}]}
\newcommand*{\B}[1]{\mathrm{B}[{#1}]}
\newcommand*{\D}{\mathrm{D}}

\DeclareMathOperator{\spr}{spr} 

\DeclareMathOperator{\sgn}{sgn}

\newcommand{\Ekin}{\ensuremath{E^\text{kin}}}
\newcommand{\Epot}{\ensuremath{E^\text{pot}}}
\newcommand{\Ekindis}{\ensuremath{E^\text{kin}_{\dxi}}}
\newcommand{\Epotdis}{\ensuremath{E^\text{pot}_{\dxi}}}

\newtheorem{thm}{Theorem}[section]

\newtheorem{cor}[thm]{Corollary}
\newtheorem{lem}[thm]{Lemma}

\theoremstyle{remark}
\newtheorem{rem}[thm]{Remark}

\theoremstyle{definition}
\newtheorem{dfn}[thm]{Definition}

\newtheorem{prop}[thm]{Proposition}

\newcommand{\ds}{\displaystyle}

\numberwithin{equation}{section} 

\newcommand{\arxiv}[1]{\href{http://arxiv.org/abs/#1}{arXiv:#1}} 

\begin{document}

\title[Variational discretization of a 2CH system]{A semi-discrete scheme derived from variational principles for global
  conservative solutions of a Camassa--Holm system}

\author[S.\ T.\ Galtung]{Sondre Tesdal Galtung}
\address[S.\ T.\ Galtung]{Department of Mathematical Sciences, NTNU -- Norwegian University of Science and Technology, 7491 Trondheim, Norway}
\email{sondre.galtung@ntnu.no}

\author[X.\ Raynaud]{Xavier Raynaud} 
\address[X.\ Raynaud]{Department of Mathematical Sciences, NTNU -- Norwegian University of Science and Technology, 7491 Trondheim and SINTEF applied mathematics and cybernetics, Oslo, Norway}
\email{xavier.raynaud@ntnu.no}

\keywords{Camassa--Holm equation, two-component Camassa--Holm system, calculus of variations, Lagrangian coordinates, energy preserving discretizations, discrete Green's functions, discrete Sturm--Liouville operators}

\begin{abstract}
  We define a kinetic and a potential energy such that the principle of stationary action from Lagrangian mechanics yields
  a Camassa--Holm system (2CH) as the governing equations. After discretizing these energies, we use the same variational
  principle to derive a semi-discrete system of equations as an approximation of the 2CH system. The discretizaton is only
  available in Lagrangian coordinates and requires the inversion of a discrete Sturm--Liouville operator with time-varying
  coefficients. We show the existence of fundamental solutions for this operator at initial time with appropriate decay.
  By propagating the fundamental solutions in time, we define an equivalent semi-discrete system
  for which we prove that there exists unique global solutions. Finally, we show how the solutions of the semi-discrete system can
  be used to construct a sequence of functions converging to the conservative solution of the 2CH system.
\end{abstract}

\maketitle

\section{Introduction} \label{sec:intro}
The Camassa--Holm (CH) equation
\begin{equation}
  u_{t}-u_{txx} + 3uu_{x} -2u_{x}u_{xx} - uu_{xxx} = 0,
  \label{eq:CH}
\end{equation}
is first known to have appeared in \cite{Fokas1981}, although written in an alternative form, as a special case in a hierarchy of completely
integrable partial differential equations. The equation gained prominence after it was derived in \cite{Camassa1993} as a limiting
case in the shallow water regime of the Green--Naghdi equations from hydrodynamics, see also
\cite{constantin2009hydrodynamical}. Since then, the CH equation has been widely studied due to its rich mathematical structure: It
is for instance bi-Hamiltonian, admits a Lax pair and is completely integrable. The solutions may develop singularities in finite
time even for smooth initial data, see, e.g., \cite{constantin1998wave,constantin2000blow}.

The so-called Camassa--Holm system (2CH)
\begin{subequations}
  \label{eq:2CH}
  \begin{align}
    \label{eq:2CHa}
    u_{t} - u_{txx} + 3uu_{x} - 2u_{x}u_{xx} -uu_{xxx} + \rho\rho_{x} &= 0\\
    \label{eq:2CHb}
    \rho_{t} + (\rho u)_{x} &= 0
  \end{align}
\end{subequations}
was first introduced in \cite{Olver1996}.

This is not the only two-component generalization which has been proposed for the CH equation.
For instance, in \cite{chen2006two,falqui2005camassa} the authors showed how similar systems are related to the AKNS hierarchy.
However, we will here only consider \eqref{eq:2CH}, which similarly to \eqref{eq:CH} can be derived as a model
for water waves. Indeed, the system was derived in \cite{escher2016two} from the Euler
equations in the case of constant vorticity, while different derivation based on the Green--Naghdi equations can
be found in \cite{ConsIvan2008}. The 2CH system shares many properties with the CH equation: The equation is bi-hamiltonian
\cite{Olver1996}, admits a Lax pair and is integrable \cite{ConsIvan2008}.  Results on the well-posedness, blow-up and global existence of solutions to
\eqref{eq:2CH} are provided in \cite{escher2007well, gui2010global, guan2015well,escher2011geometry}.

Both the CH equation and the 2CH system are geodesic equations, see \cite{cons:03,cons:01b,constantin2002geometric, escher2011geometry}. The CH
equation is a geodesic on the group of diffeomorphisms for the right-invariant norm
\begin{equation}
  \label{eq:defenergyH1}
  \Ekin = \frac12\normHone{u}^2 = \frac12\int_\Real(u^2 + u_x^2)\,dx.
\end{equation}
To clarify this statement, we introduce the notation $y:\Real^+\times\Real \to \Real$ for a path in the group of diffeomorphisms,
meaning that $y(t, \xi)$ denotes the path of a particle initially at $\xi$, and the Eulerian velocity is given by $u(t, x) =
y_t(t, y^{-1}(x))$. The geodesic equation is then obtained as an extremal solution for the action functional
\begin{equation*}
  \Acal(y) = \int_{t_0}^{t_1} \Ekin(t)\,dt = \frac12\int_{t_0}^{t_1}\int_\Real\left(y_t^2y_\xi  + \frac{y_{t\xi}^2}{y_\xi}\right)\,d\xi dt.
\end{equation*}
The momentum map, as defined in \cite{arnold1999topological}, is given by the Helmholtz transform $m(u) = u - u_{xx}$ in Eulerian
coordinates. Then we may write the energy as $\Ekin = \frac12\int_\Real m(u)u\,dx$. For the 2CH system in
\cite{escher2011geometry}, the diffeomorphism group is replaced with a semi-direct product which accounts for the variable
$\rho$. Then the 2CH system is a geodesic for the right-invariant norm $ \frac12\normHone{u}^2 +
\frac12\normLtwo{\rho}^2$. However, we will not follow this approach here, but rather use the fact that
  \eqref{eq:2CH} can be derived as the governing equation for a different variational problem, where the action functional
  includes a potential energy term and the variation is performed on the group of diffeomorphisms only. This point of view enables
  us to derive a discretization which mimics the variational derivation of the continuous case.  In this approach, we consider the
  variable $\rho$ as a density entering the action functional through a potential term
\begin{equation}
  \label{eq:defEpot}
  \Epot = \frac12\int_\Real(\rho - \rho_\infty)^2\,dx,
\end{equation}
where $\rho_{\infty} \ge 0$ is the asymptotic value of $\rho$. The mass conservation equation $\rho_t + (\rho u)_x = 0$ is not a
result of the variational derivation, but is instead a given constraint of the problem.  Equation \eqref{eq:defEpot} can be
interpreted as an elastic energy: It increases whenever the system deviates from the rest configuration given by
$\rho \equiv \rho_\infty$.  In the beginning of Section \ref{sec:disc}, we present the derivation of the 2CH as the critical point
for the variation of $\Ekin - \Epot$. This approach follows the classical framework, see \cite{arnold89}, and the potential term
$\Epot$ depending on the density is added in the same way as in \cite{lin2005hydrodynamics}, see also
\cite{hyon2010energetic,giga2017variational,xu2014energetic} for applications to more complex fluids. In Lagrangian variables, the
mass conservation equation simplifies to the expression
\begin{equation}
  \label{eq:masscons}
  \pdiff{t}{}(\rho(t,y) y_\xi) = 0.
\end{equation}
To derive a discrete approximation of the 2CH system, we propose to follow the same steps of the variational derivation in the
continuous case. First, we discretize the path functions $y(t,\xi)$ by piecewise linear functions, $y_i(t) = y(t, \xi_i)$ for
$\xi_i = i\dxi$, $i\in\Z$ and $\dxi > 0$. Then, we approximate the Lagrangian using these discretized variables. Finally, we
obtain the governing equation for the discretized system from the principle of stationary action, as in the continuous case. In
our opinion, the advantage of using this variational approach as basis for our discretization is that we need only take variations
with respect to a single discrete variable, rather than two. This reduction is achieved by the use of the identity
\eqref{eq:masscons}. Note that the group structure of the diffeomorphisms is not carried over to the discrete setting, as the
composition rule is not defined at the discrete level. In practice, this means that that our discretized equation will not have a
purely Eulerian form and should be solved in Lagrangian variables. We retain two symmetries though, the time and space translation
invariance. As a result, we have conservation of discrete counterparts of the integrals $\int_\Real u^2+u_x^2\,dx$ and
$\int_\Real u\,dx$, see Section \ref{sec:disc}.

We rewrite the 2CH system \eqref{eq:2CH} in Lagrangian variables following \cite{Grunert2012}. We first apply the inverse of the
Helmholtz operator $\id - \partial_{xx}$ to obtain
\begin{equation}
  \label{eq:cheul}
  u_t + uu_x = -P_x,\quad
  P - P_{xx} = u^2 + \frac12 u_x^2 + \frac12\rho^2.
\end{equation}
We rewrite the second equation above as a system of first-order equations,
\begin{equation}
  \label{eq:helmeul}
  \begin{bmatrix}
    -\partial_x & 1 \\
    1 & -\partial_x
  \end{bmatrix}
  \circ
  \begin{bmatrix}
    P\\Q
  \end{bmatrix}
  =
  \begin{bmatrix}
    0\\f
  \end{bmatrix},
\end{equation}
for $Q = P_x$ and $f = u^2 + \frac12 u_x^2 + \frac12\rho^2$. In Lagrangian variables we have $\bar P(\xi) = P(y(\xi))$, and the
system \eqref{eq:helmeul} becomes
\begin{equation}
  \label{eq:helmoLag}
  \begin{bmatrix}
    -\partial_\xi & y_\xi\\
    y_\xi & -\partial_\xi
  \end{bmatrix}
  \circ
  \begin{bmatrix}
    \bar P \\ \bar Q
  \end{bmatrix}
  =
  \begin{bmatrix}
    0\\\bar f
  \end{bmatrix},
\end{equation}
for $\bar f = f\circ y$. In \eqref{eq:helmoLag}, the operator denoted by $y_\xi$ corresponds to pointwise multiplication by
$y_\xi$. The matrix operator corresponds to the momentum map in Lagrangian coordinates and must be inverted to solve the
system. In contrast to its Eulerian counterpart in \eqref{eq:helmeul}, the operator evolves in time. This significantly complicates the analysis,
especially in the discrete case. In Section \ref{sec:analysis}, we introduce the operators $\Gcal$ and $\Kcal$ which define the fundamental
solutions of the momentum operator,
\begin{equation} \label{eq:defGKop}
  \begin{bmatrix}
    -\partial_\xi & y_\xi\\
    y_\xi & -\partial_\xi
  \end{bmatrix}
  \circ
  \begin{bmatrix}
    \Kcal & \Gcal\\
    \Gcal & \Kcal
  \end{bmatrix}
  =
  \begin{bmatrix}
    \delta & 0\\
    0 & \delta
  \end{bmatrix}.
\end{equation}
Note that the operator becomes singular when $y_\xi$ vanishes. In the discrete case, the momentum operator and its fundamental
solution are given by
\begin{equation}
  \label{eq:disdefgkgk}
  \begin{bmatrix}
    -\D_- & \D_+y\\
    \D_+y & -\D_+
  \end{bmatrix}
  \circ
  \begin{bmatrix}
    \gamma & k\\
    g & \kappa
  \end{bmatrix}
  =
  \begin{bmatrix}
    \delta & 0\\
    0 & \delta
  \end{bmatrix},
\end{equation}
where $\D_\pm$ denotes forward and backward difference operators, see Section \ref{sec:disc}.  This is a form of \emph{Jacobi
  difference equation}, cf.\ \cite{JacobiOperator}.  To establish solutions of \eqref{eq:disdefgkgk}, we shall invoke results from
\cite{Friedland2006, Pituk2002} which generalize the Poincar\'{e}--Perron theory on difference equations.
Section \ref{sec:aux} is completely devoted to this analysis.

The CH equation and 2CH system can blow up in finite time, even for smooth initial data. The blow-up scenario for CH has been
described in \cite{constantin1998global,constantin1998wave,constantin2000global} and consists of a singularity where $\lim_{t\to
  t_c}u_x(t, x_c) = -\infty$ for some critical time $t_c$ and location $x_c$. However, since the $\Hone$-norm of the solution is
preserved, the solution remains continuous. In fact, the solution can be prolongated in two consistent ways: Conservative
solutions will recover the total energy after the singularity, while dissipative solutions remove the energy that has been trapped
in the singularity, see
\cite{bressan2007global,holden2007global,Grunert2012,grunholray12,bressan2007dissglobal,holden2009dissipative,grunert2014global}.
If $\rho>0$ initially, no blow-up occurs and the 2CH system preserves the regularity of the initial data, see
\cite{Grunert2012}. We can interpret this property as a regularization effect of the elastic energy: The particles cannot
accumulate at a given location because of a repulsive elastic force. The peakon-antipeakon collision is a good illustration of the dynamics of the blow-up. We present this scenario in Figures \ref{fig:characteristics}
and \ref{fig:mpeakon}. In the peakon-antipeakon solution, which corresponds to $\rho_0 \equiv 0$, we observe the breakdown of the
solution and the concentration of the energy distribution into a singular measure. At collision time, $u^2 + u_x^2 = 0$ and the
energy reduces to a pure singular Dirac measure, which naturally cannot be plotted. For the same $u_0$, but $\rho_0 \equiv 1$, the potential energy
prevents the peaks from colliding, which is clear from the plot of the characteristics in Figure \ref{fig:characteristics}. The
potential energy grows as the characteristics converge and results in an apparent force which diverts them.

The global conservative solutions of the 2CH system are based on the following conservation law for the
energy,
\begin{equation}
  \label{eq:eulconsenerg}
  (\tfrac12(u^2 + u_x^2 + (\rho - \rho_\infty)^2))_t + (u\tfrac12(u^2 + u_x^2 + (\rho - \rho_\infty)^2))_x = -(uR)_x,
\end{equation}
where $R = P - \frac12u^2 - \frac12 \rho_\infty^2$ for $P$ in \eqref{eq:cheul}. This equation enables us to compute the evolution of the cumulative energy defined from the energy
distribution as
\begin{equation*}
  H(t, \xi) = \frac12 \int_{-\infty}^{y(t,\xi)}(u^2 + u_x^2 + (\rho - \rho_\infty)^2)\,dx
\end{equation*}
in Lagrangian coordinates, for which we obtain $\frac{dH}{dt} = -(uR) \circ y$. This evolution equation is essential to keep track of the energy
when the solution breaks down. To handle the blow-up of the solution, we need also to have a framework which allows the flow map
$\xi \mapsto y(t,\xi)$ to become singular, that is where $y_\xi$ can vanish and the momentum operator in Lagrangian coordinates become
ill-posed. In \cite{Grunert2012}, explicit expressions for $P$ and $Q$ are given. Here, we adopt a different approach where we
propagate the fundamental solutions $\Kcal$ and $\Gcal$ from \eqref{eq:defGKop} in time. Introducing $U = u \circ y$,
the equivalent system for \eqref{eq:2CH} is given by
\begin{subequations}
  \label{eq:contchgoveq}
\begin{equation} 
  \label{eq:contchgoveq1}
  y_{t} = U,\quad U_{t} = -Q,\quad  H_{t} = -UR,
\end{equation}
with the evolution equations for $\Kcal$ and $\Gcal$ given by
\begin{equation}    
  \fracpart\Gcal = [\Ucal, \Kcal],\quad \fracpart\Kcal = [\Ucal, \Gcal].
\end{equation}
Here $[\Ucal, \Kcal]$ denotes the commutator of $U$ and $\Kcal$, see Section \ref{sec:analysis}. In the case where $\rho_{\infty} = 0$,
$R$ and $Q$ in \eqref{eq:contchgoveq1} are given by
\begin{equation}
  \begin{bmatrix}
    R\\Q
  \end{bmatrix}
  =
  \begin{bmatrix}
    \Kcal & \Gcal\\
    \Gcal & \Kcal
  \end{bmatrix}
  \circ
  \begin{bmatrix}
    \tfrac12  U^2\\
    H
  \end{bmatrix}_\xi.
\end{equation}
\end{subequations}
The derivation of \eqref{eq:contchgoveq} can be carried over to the discrete system, and this is
done in Section \ref{sec:analysis}.

The short-time existence for the solution of the semi-discrete system relies on Lipschitz estimates. At this stage, one of
the main ingredients in the proofs is the Young-type estimate for discrete operators presented in Proposition
\ref{prp:young}. For the global existence, we have to adapt the argument of the continuous case and complement it with
\textit{a priori} estimates of the fundamental solutions $(g,k,\gamma,\kappa)$. These estimates follow from monotonicity
properties of these operators, see Lemma \ref{lem:fundsolprops}. Section \ref{sec:existence} is devoted to establishing
existence and uniqueness for global solutions of the discrete 2CH system. In Section \ref{sec:convergence}, we explain how the
solution of the semi-discrete system can be used to construct a sequence of functions that converge to the solution of the
2CH system \eqref{eq:2CH}. Finally, in Section \ref{sec:init_data} we present how to construct appropriate initial data for the semi-discrete system in order
to achieve the convergence in Section \ref{sec:convergence}.

\begin{figure}
  \centering
  \begin{subfigure}{0.495\textwidth}
    \includegraphics[width=\textwidth]{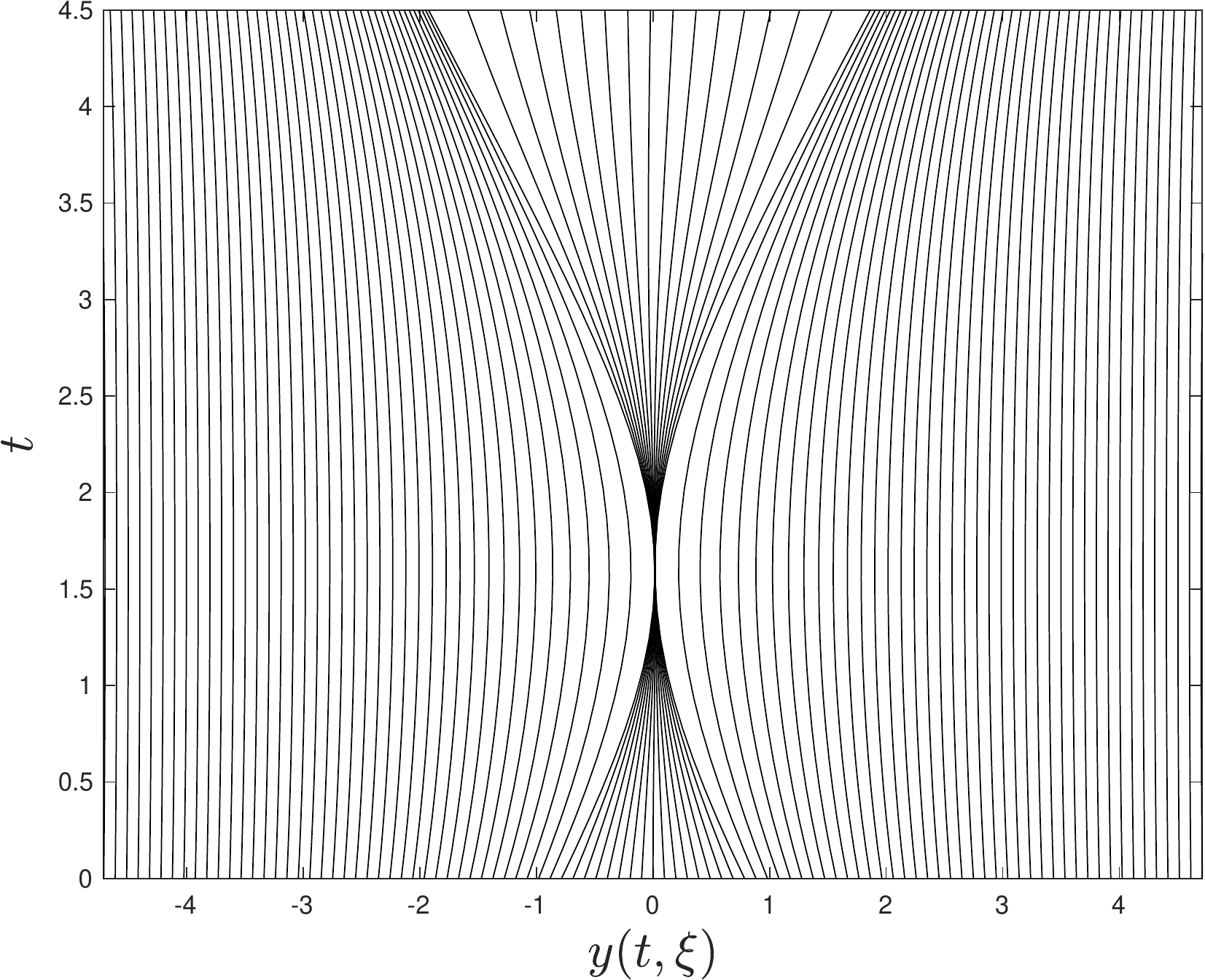}
    \subcaption{$\rho_0 \equiv 0$}
  \end{subfigure}
  \begin{subfigure}{0.495\textwidth}
    \includegraphics[width=\textwidth]{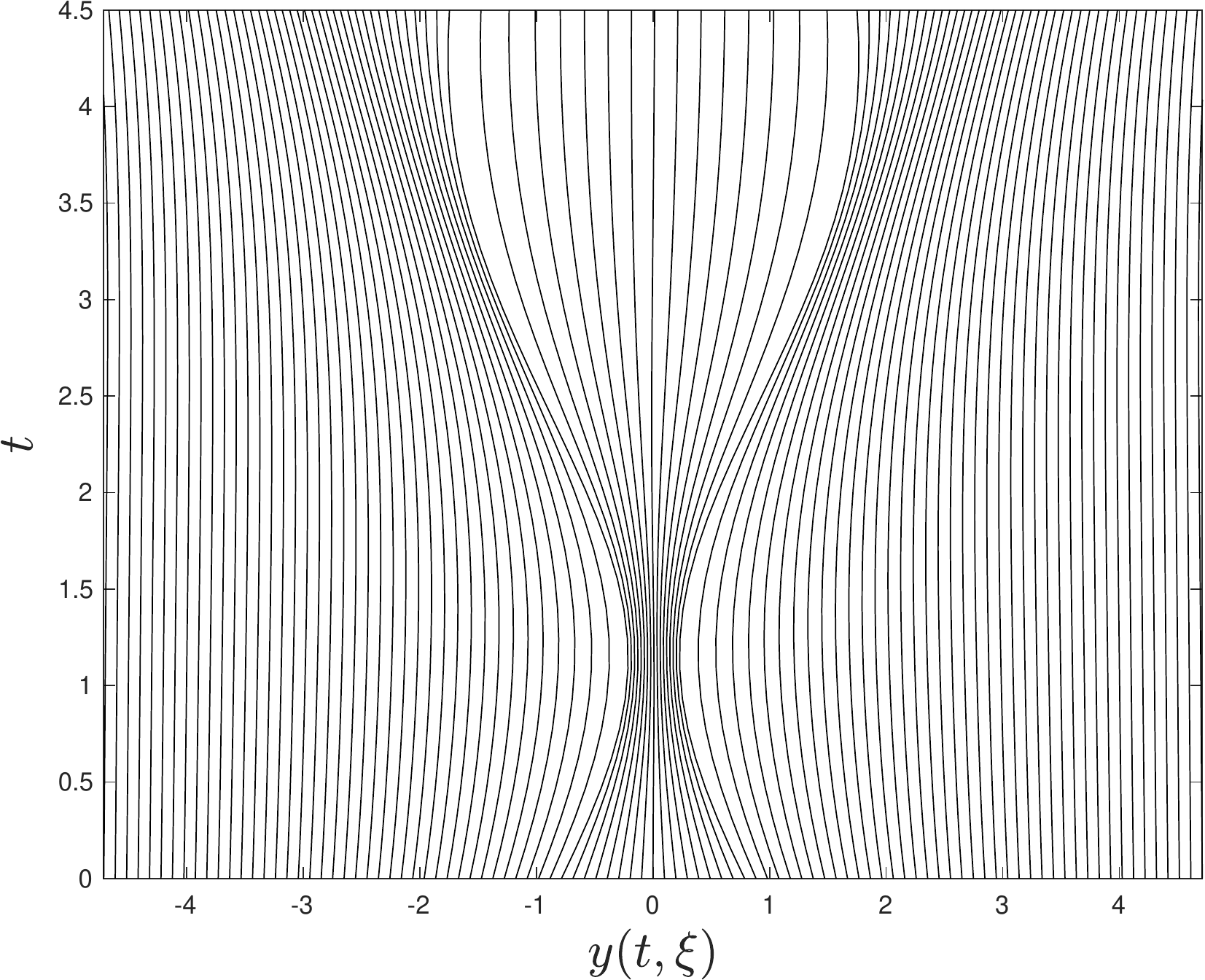}
    \subcaption{$\rho_0 \equiv 1$}
  \end{subfigure}
  \caption{Plot of the characteristics for peakon-antipeakon initial data $u_0$ with $\rho_0$ equal to 0 and 1. We observe the
    regularizing effect of $\rho_0>0$ which prevents the characteristics from colliding.}
  \label{fig:characteristics}
\end{figure}

\begin{figure}
  \centering
  \begin{subfigure}{0.495\textwidth}
    \includegraphics[width=\textwidth]{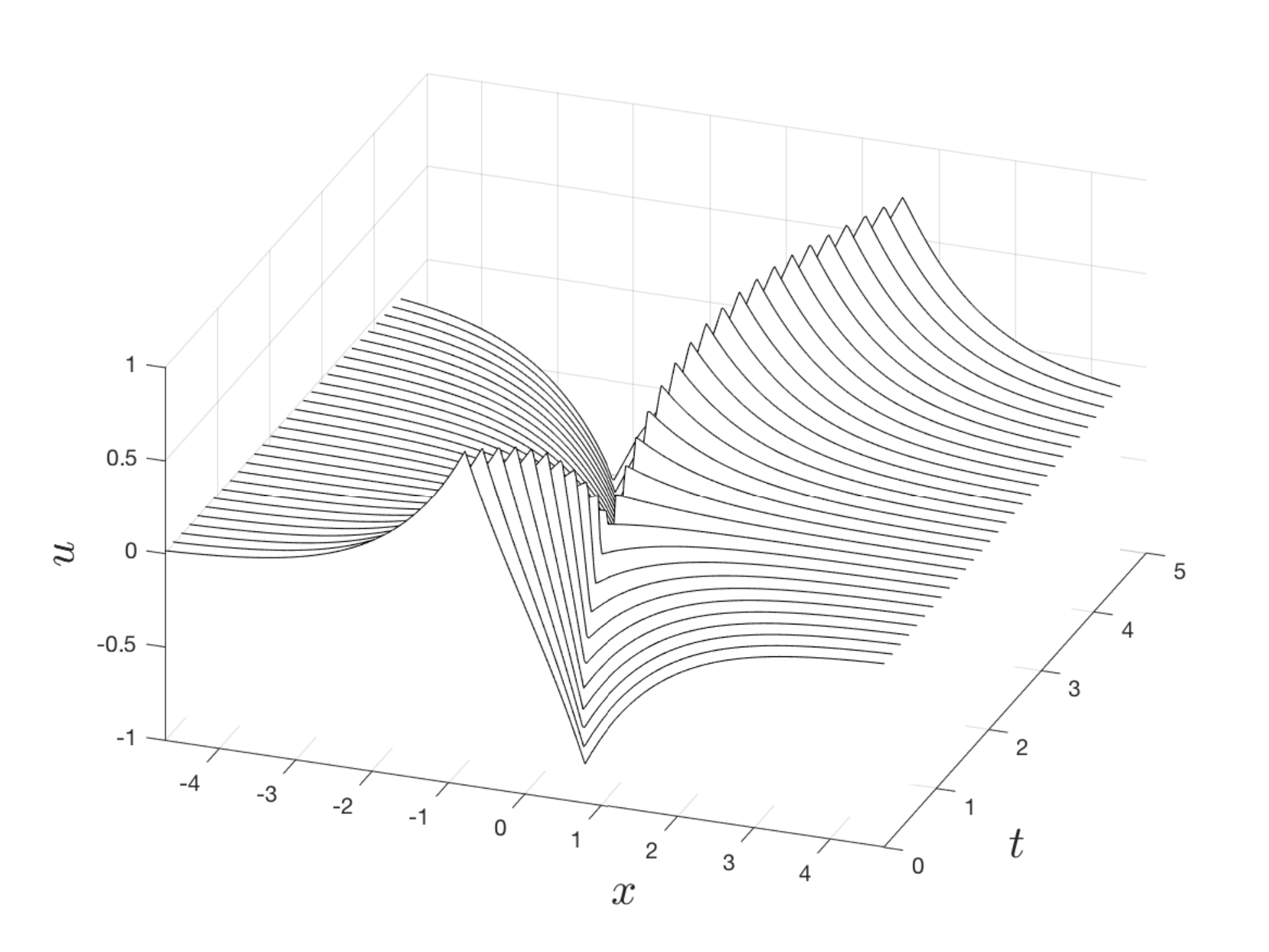}
    \subcaption{$u(t, x)$ for $\rho_0 \equiv 0$}
  \end{subfigure}
  \begin{subfigure}{0.495\textwidth}
    \includegraphics[width=\textwidth]{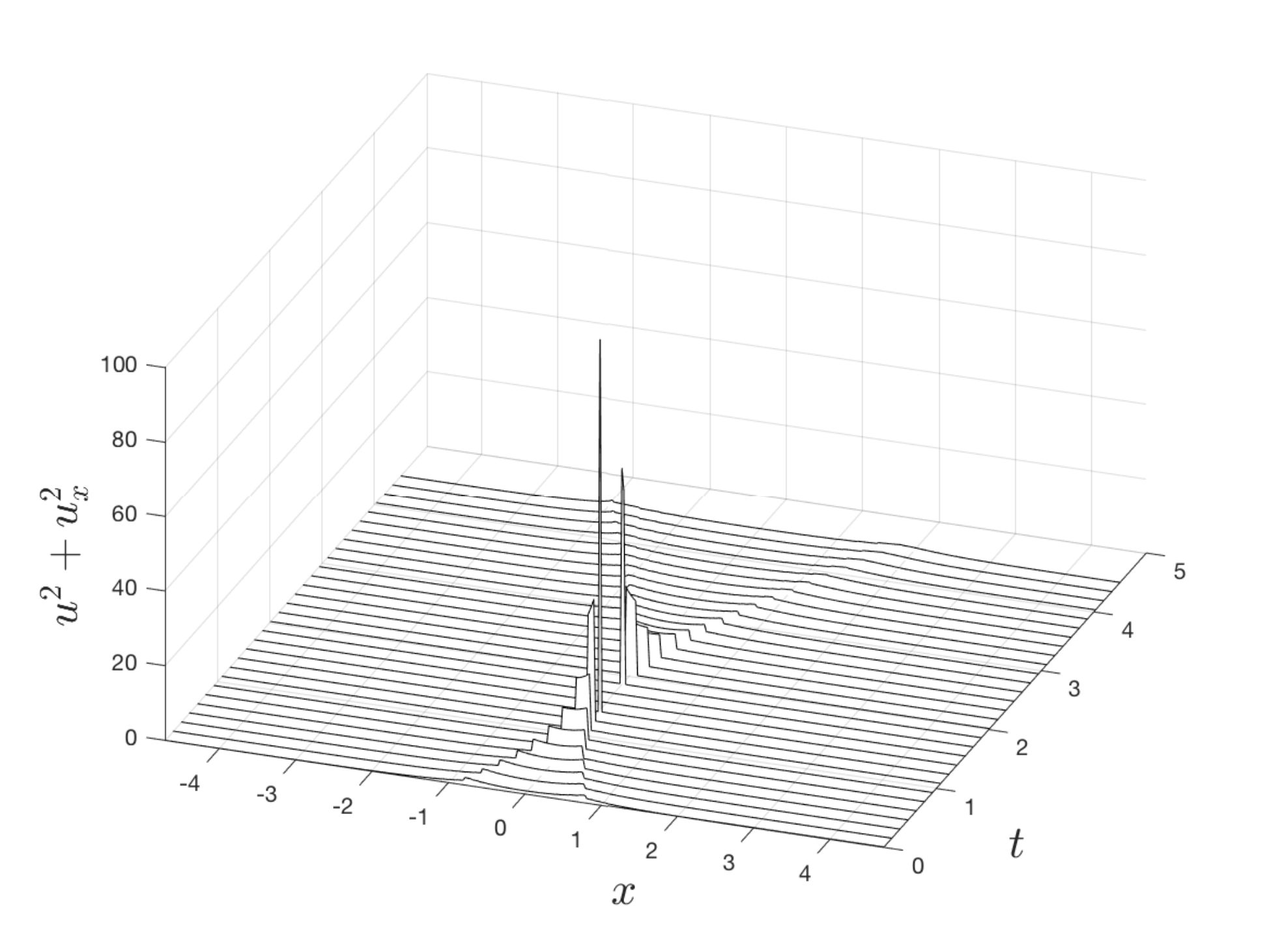}
    \subcaption{$u^2(t, x) + u_x^2(t, x)$ for $\rho_0 \equiv 0$}
  \end{subfigure}
  \begin{subfigure}{0.495\textwidth}
    \includegraphics[width=\textwidth]{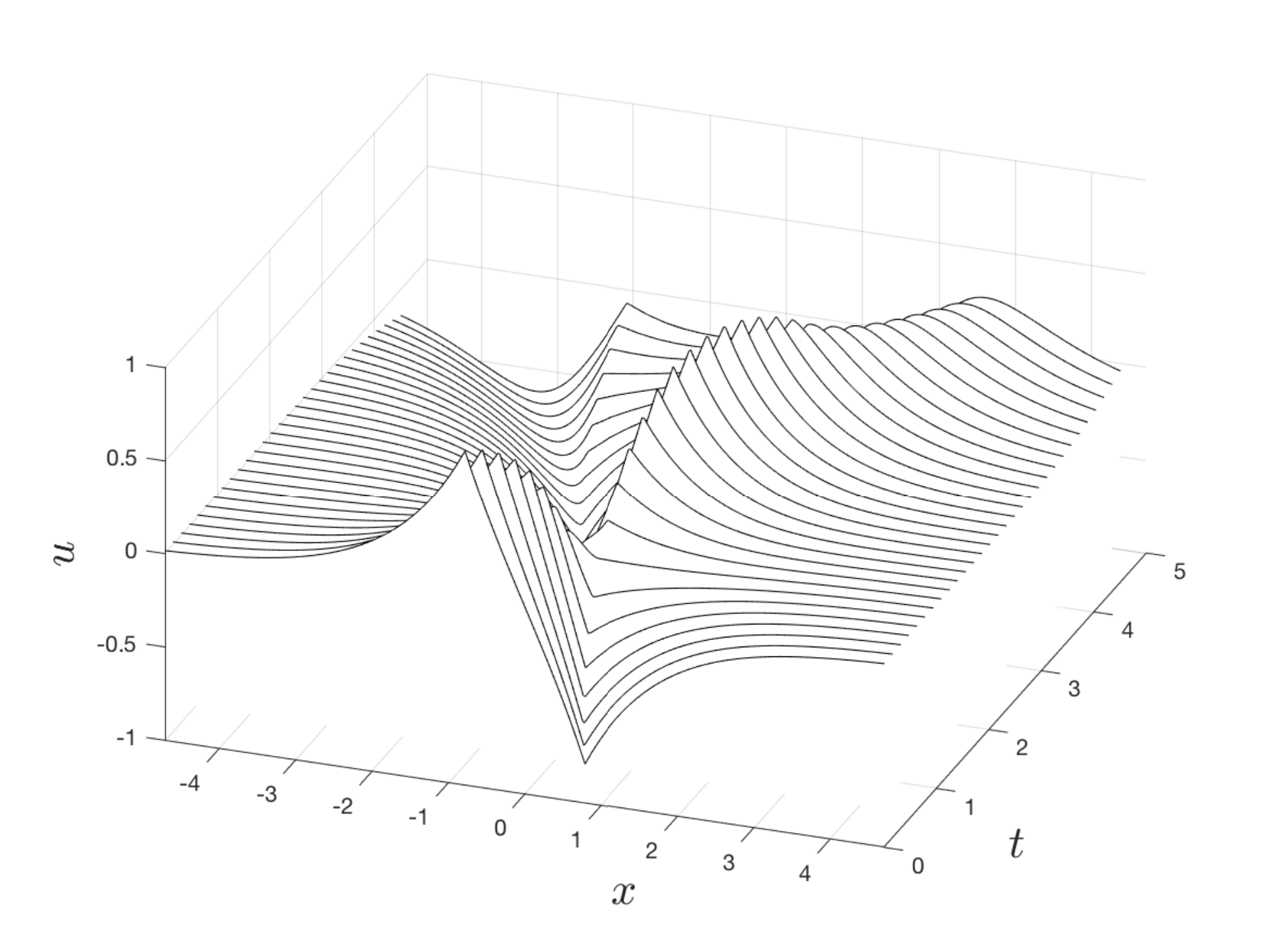}
    \subcaption{$u(t, x)$ for $\rho_0 \equiv 1$}
  \end{subfigure}
  \begin{subfigure}{0.495\textwidth}
    \includegraphics[width=\textwidth]{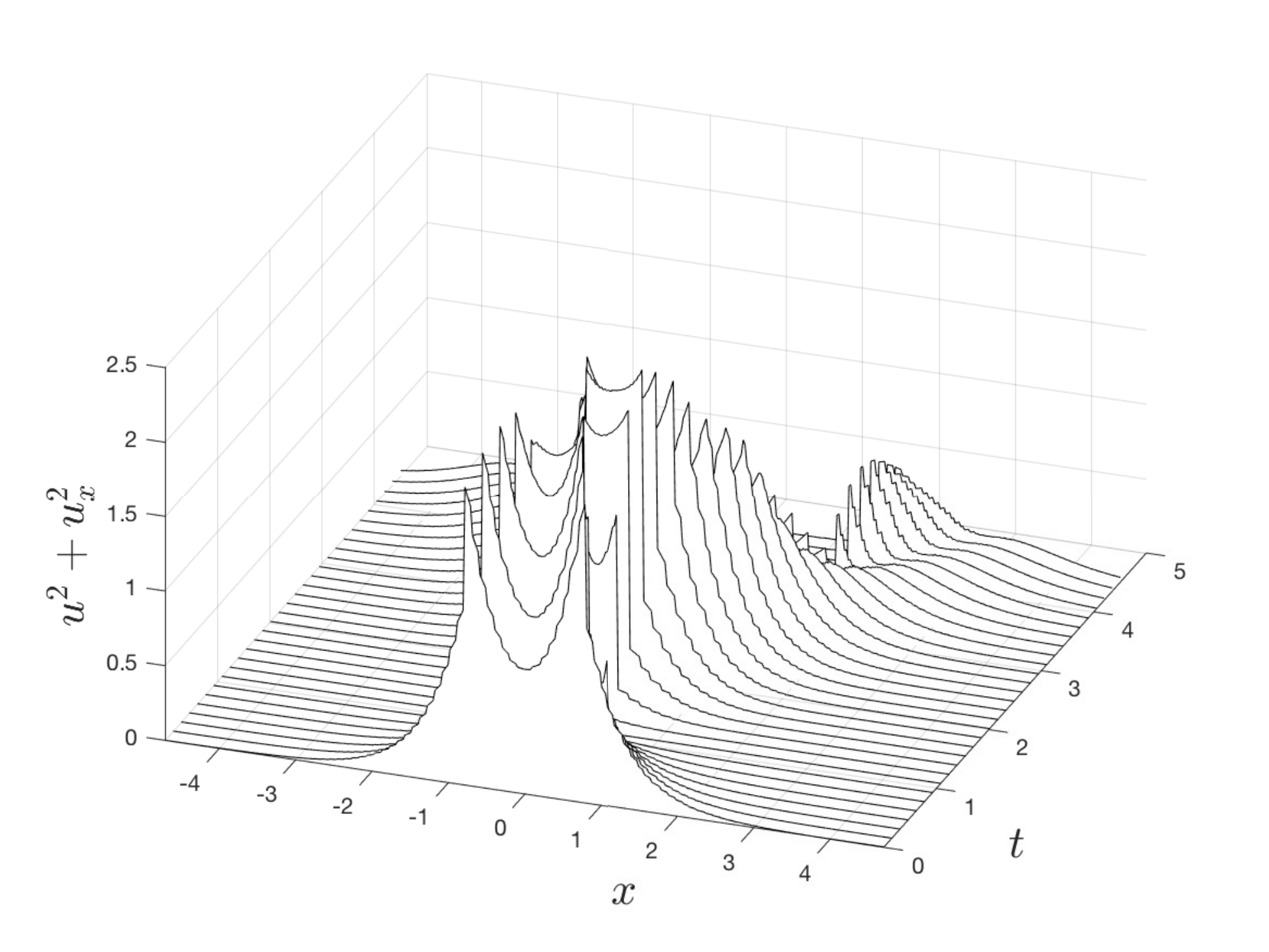}
    \subcaption{$u^2(t, x) + u_x^2(t, x)$ for $\rho_0 \equiv 1$}
  \end{subfigure}
  \begin{subfigure}{0.495\textwidth}
    \includegraphics[width=\textwidth]{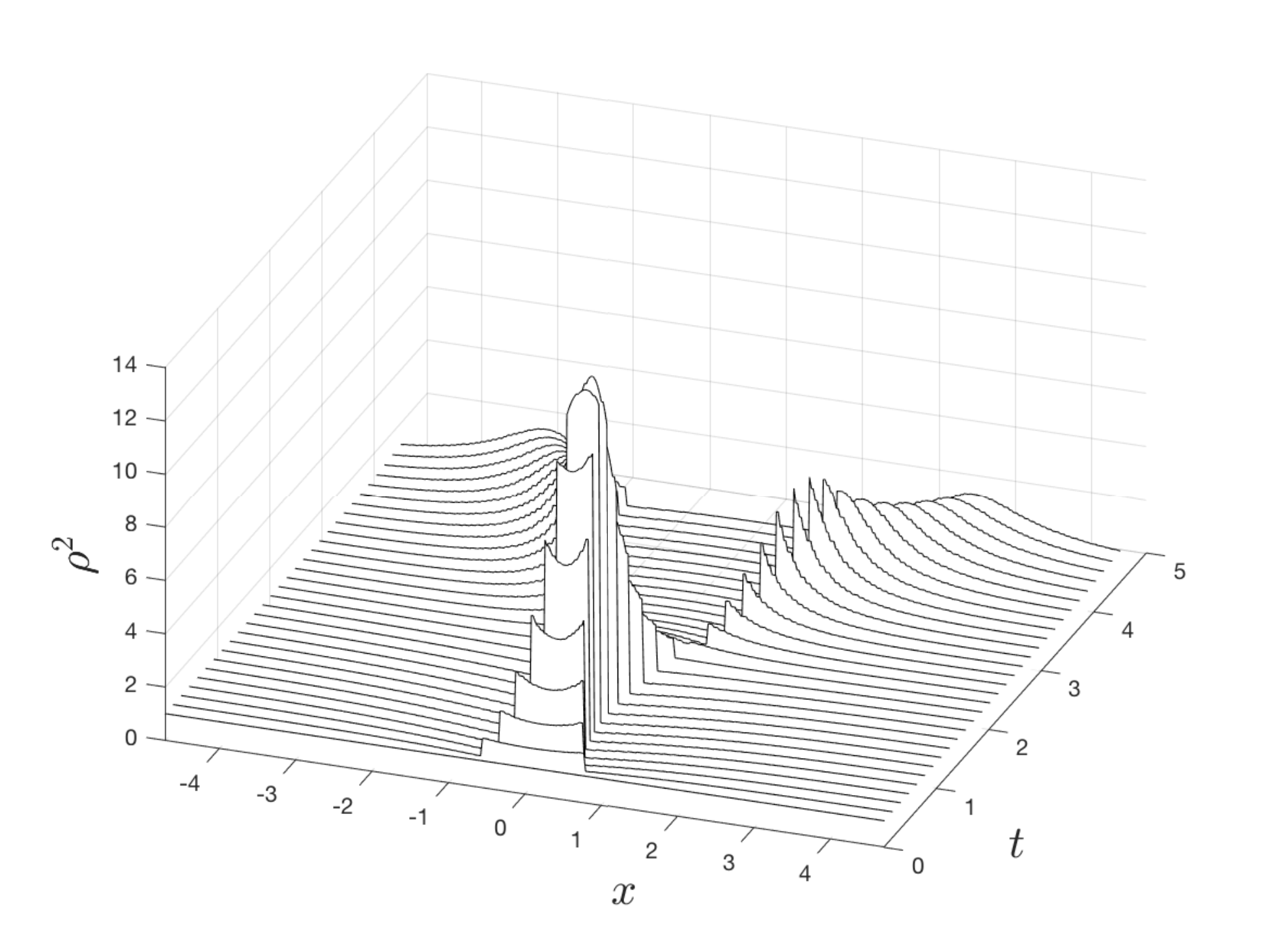}
    \subcaption{$(\rho(t,x) - \rho_\infty)^2$ for $\rho_0 \equiv 1$}
  \end{subfigure}
  \caption{Solutions for peakon-antipeakon initial data. For $\rho_0 \equiv 0$ we plot $u$ in (a) and $u^2 + u_x^2$ in (b). We observe the blow-up of $u_x$ at $t_c\approx 1.5$ and the
    concentration of energy. For the same initial $u_0$, but $\rho_0 \equiv 1$, we plot the corresponding solution in (c)--(e) and observe that $u_x$ does not blow up. In (e) we plot the distribution of the potential energy given by $(\rho(t,x) - \rho_\infty)^2$, and observe how it grows when the peaks get closer to each other.}
  \label{fig:mpeakon}
\end{figure}

\section{Derivation of the semi-discrete CH system using a variational approach}
\label{sec:disc}

As a motivation for our discretization, we will here outline how the system \eqref{eq:2CH} can be derived from
 a variational problem involving a potential term, as indicated in the previous section.
In a standard way, see,
e.g., \cite{arnold89}, the Lagrangian $\Lcal$ is defined as the difference of a kinetic and potential energy
\begin{equation}
  \label{eq:deflagrangian}
  \Lcal = \Ekin - \Epot,
\end{equation}
where the energies are given by \eqref{eq:defenergyH1} and \eqref{eq:defEpot}. The governing equation is then derived by
the least action principle, also called principle of stationary action, on the group of diffeomorphisms.

A first step in this direction is to introduce the particle path, denoted by $y(t,\xi)$.
Then, we rewrite $\Ekin$ and $\Epot$ in Lagrangian variables.  For the kinetic energy, we obtain
\begin{equation}
  \label{eq:defEkincontlag}
  \Ekin(t) = \frac{1}{2} \intR{\xi}{\left( y_{t}^2 y_{\xi} + \frac{y_{t\xi}^2}{y_{\xi}}\right)(t,\xi)}.
\end{equation}
From the principle of mass conservation for a control volume, the density satisfies
\begin{equation}
  \label{eq:defrho}
  \rho(t,y(t,\xi))y_\xi(t,\xi) = \rho(0,y(0,\xi)) y_\xi(0,\xi).
\end{equation}
The definition of $\rho$ given by \eqref{eq:defrho} is equivalent to the conservation law \eqref{eq:2CHb}. We can check this
statement directly:
\begin{equation*}
  \pdiff{t}{}(\rho(t,y) y_\xi) = (\rho_t(t, y) + \rho_x(t, y) u(t, y) + \rho(t, y) u_x(t, y))y_\xi = 0. 
\end{equation*}
We introduce the Lagrangian density $r$ defined as $r(t,\xi) = \rho(t,y(t,\xi))y_\xi(t,\xi)$, and by requiring it to be preserved in time, we impose the definition of the density $\rho$ in the system.
The identity \eqref{eq:defrho} allows us to reduce the number of variables in our Lagrangian.
Indeed, we rewrite the potential energy \eqref{eq:defEpot} in terms of the particle path $y$ only and obtain
\begin{equation}
  \label{eq:defEpotcontlag}
  \Epot(t) = \frac12 \intR{\xi}{\left(\rho_{0}(y(0, \xi))\frac{y_{\xi}(0,\xi)}{y_{\xi}(t,\xi)}-\rho_{\infty}\right)^2 y_{\xi} (t,\xi)}.
\end{equation}
Combining \eqref{eq:deflagrangian} with \eqref{eq:defEkincontlag} and \eqref{eq:defEpotcontlag} one can derive
\begin{equation}\label{eq:varderiv}
  \vdiff{y_t}{\Lcal(y,y_t)} = y_\xi y_t -\left(\frac{y_{t\xi}}{y_\xi}\right)_\xi, \qquad \vdiff{y}{\Lcal(y,y_t)} = -y_t y_{t\xi} + \frac12 \left(\frac{y_{t\xi}^2}{y_\xi^2} - \frac{\rho_0^2}{y_\xi^2}\right)_\xi,
\end{equation}
which must satisfy the Euler--Lagrange equation
\begin{equation}\label{eq:EL}
  \pdiff{t}{} \vdiff{y_t}{\Lcal(y,y_t)} = \vdiff{y}{\Lcal(y,y_t)}.
\end{equation}
Now, from the relations $(\rho \circ y)y_{\xi} = (\rho_{0}\circ y_0) (y_0)_\xi$, $y_{t} = u \circ y$, $y_{t\xi} = (u_x \circ y)y_{\xi}$, and $y_{tt} = (u_{t} + uu_{x}) \circ y$ which implies $y_{tt\xi} = ((u_{t} + uu_{x})_{x} \circ y ) y_{\xi}$, we can write
\begin{align*}
    \pdiff{t}{} \vdiff{y_{t}}{\Lcal(y,y_{t})} &= (y_{t} y_{\xi})_{t} - \left( \frac{y_{t\xi}}{y_{\xi}} \right)_{t\xi} = y_{tt}y_{\xi} + y_{t} y_{t\xi} - \left(\frac{y_{tt\xi}}{y_{\xi}} - \frac{y_{t\xi}^2}{y_{\xi}^2}\right)_{\xi} \\
    &= ((u_{t} + 2uu_{x}) \circ y)y_{\xi} - (((u_{t} + uu_{x})_{x} - u_{x}^2) \circ y)_{\xi} \\
    &= ((u_{t} - u_{txx} + 2uu_{x} - u_{x}u_{xx} - uu_{xxx}) \circ y)y_{\xi}
\end{align*}
and
\begin{align*}
  \vdiff{y}{\Lcal(y,y_{t})} &= ((-uu_{x} + u_{x}u_{xx}) \circ y)y_{\xi} -\frac{1}{2} \left( \left(\rho \circ y\right)^{2} \right)_{\xi} \\
  &= ((-uu_{x} + u_{x}u_{xx}) \circ y)y_{\xi}  - (\rho \rho_{x} \circ y)y_{\xi}.
\end{align*}
Inserting the above identities in the Euler--Lagrange equation \eqref{eq:EL} we get
\begin{equation*}
  ((u_{t} - u_{txx} + 3uu_{x} - 2u_{x}u_{xx} - uu_{xxx} + \rho \rho_{x}) \circ y)y_{\xi} = 0
\end{equation*}
which is exactly \eqref{eq:2CHa} when $y_{\xi} \neq 0$.

For the remainder of this section we give details of how we discretize the variational derivation
outlined above.
Let us start by discretizing the kinetic and potential energies given by \eqref{eq:defEkincontlag} and \eqref{eq:defEpotcontlag}, respectively.
First we divide the line into a uniform grid by defining $\xi_{j} = j \dxi$
for some discretization step $\dxi > 0$ and $j \in \Z$. We approximate $y(t, \xi_j)$ with $y_j(t)$ for $j\in \Z$, and the spatial
derivatives $y_\xi(t, \xi_j)$ with the finite difference $\D_+y_j$. The finite difference operators $\D_+$ and $\D_-$ are defined
as
\begin{equation}
  \label{eq:diff}
  \D_\pm y_{j} \coloneqq \pm\frac{y_{j\pm 1}-y_{j}}{\dxi},
\end{equation}
and they satisfy the discrete product rule
\begin{equation}
\label{eq:FD_prod}
\D_{\pm}(v_{j} w_{j}) = (\D_{\pm} v_{j}) w_{j \pm 1} + v_{j} (\D_{\pm} w_{j}).
\end{equation}
When we later encounter operators in the form of grid functions with two indices, such as $g_{i,j}$ for $i,j \in \Z$, we
will indicate partial differences by including the index in the difference operator, for instance $\D_{j+}g_{i,j} =
(g_{i,j+1}-g_{i,j})/\dxi$. We use the standard notation $\lbf^{p}$ and $\linfty$ for the Banach spaces with norms
\begin{equation}
  \label{eq:lpnorms}
  \|a\|_{\lbf^{p}} \coloneqq \left( \sumZxi{j} |a_{j}|^p \right)^{\frac{1}{p}} \quad\text{ and }\quad \|a\|_{\linfty} \coloneqq \sup_{j \in \Z}{|a_{j}|},
\end{equation}
with $1 \le p < \infty$.

Turning back to the energy functionals, we discretize the kinetic energy \eqref{eq:defEkincontlag} using finite differences and set
\begin{equation}
  \label{eq:defEkindis}
  \Ekindis \coloneqq \frac12 \sumZxi{j}{\left((\dot{y}_{j})^2 (\D_+ y_{j}) + \frac{(\D_+ \dot{y}_{j})^2}{\D_+ y_{j}}\right)}.
\end{equation}
The Lagrangian velocity is as usual defined as $U_i = \dot y_i$ and, using this notation, \eqref{eq:defEkindis} becomes
\begin{equation*}
  \Ekindis = \frac12\sumZxi{j}{\left(U_j^2\D_+y_j + \frac{(\D_+ U_j)^2}{\D_+y_j}\right)}.
\end{equation*}
The discrete counterpart of \eqref{eq:defEpotcontlag} is similarly defined as
\begin{equation}
  \label{eq:defEpotdis}  \Epotdis \coloneqq \frac{1}{2} \sumZxi{j}{\left(\frac{\D_+ y_{0,j}}{\D_+ y_{j}}\rho_{0,j} - \rho_{\infty} \right)^2 (\D_+ y_{j})},
\end{equation}
where $y_{0,j} = y_0(\xi_j)$ and $\rho_{0,j} \coloneqq \rho_{0}(y_{0,j})$. Now we define the Lagrangian as the difference between the kinetic and potential
energy,
\begin{equation*}
  \Lcaldis = \Ekindis - \Epotdis.
\end{equation*}
Now we compute the Fr\'{e}chet derivatives of $\Lcaldis$ with respect to $y$ and $\dot{y}$. This
derivative is given in $\ltwo$, the space of square summable sequence using the duality pairing defined by the scalar
product,
\begin{equation*}
\ipl{v}{w} \coloneqq \sumZxi{j}{v_{j} w_{j}}, \quad v, w \in \ltwo.
\end{equation*}
Formally, we have
\begin{align*}
  \begin{aligned}
    \delta \Ekindis &=  \sumZxi{j}{\left( U_{j} (\delta U)_{j} (\D_+ y_{j}) + \frac{\D_+ U_{j}}{\D_+ y_{j}} \D_+ (\delta U)_{j} \right)} \\
    &\quad+ \frac12 \sumZxi{j}{\left( (U_{j})^{2} \D_+ (\delta y)_{j} - \left(\frac{\D_+ U_{j}}{\D_+ y_{j}}\right)^{2} \D_+ (\delta y)_{j} \right)} \\
    &= \sumZxi{j}{\left( U_{j} (\D_+ y_{j}) - \D_- \left( \frac{\D_+ U_{j}}{\D_+ y_{j}} \right) \right) (\delta U)_{j} } \\
    &\quad- \sumZxi{j}{\frac{1}{2} \D_- \left( (U_{j})^{2} - \left(\frac{\D_+ U_{j}}{\D_+ y_{j}}\right)^{2} \right) (\delta y)_{j}},
  \end{aligned}
\end{align*}
where in the final identity we have used the summation by parts formula
\begin{equation}\label{eq:sum_by_parts}
\dxi\sum_{j=m}^{n}(\D_+a_j)b_j + \dxi \sum_{j=m}^{n}a_j(\D_-b_j) = a_{n+1}b_n - a_m b_{m-1}.
\end{equation}
This leads to the Frechet derivatives
\begin{equation*}
  \left( \vdiff{y}{\Ekindis} \right)_j = -\frac{1}{2} \D_- \left((U_{j})^2 - \left(\frac{\D_+ U_{j}}{\D_+ y_{j}}\right)^2 \right)
\end{equation*}
and
\begin{equation}
  \label{eq:vdifflkin}
  \left( \vdiff{U}{\Ekindis} \right)_j = U_{j} (\D_+ y_{j}) - \D_- \left( \frac{\D_+ U_{j}}{\D_+ y_{j}} \right) = \left( \vdiff{U}{\Lcaldis} \right)_j,
\end{equation}
where the rightmost equality in \eqref{eq:vdifflkin} is a consequence of $\Epotdis$ being independent of $U$.
For the potential term we find
\begin{align*}
  \delta \Epotdis &= \frac{\dxi}{2} \sumZ{j}{}\Big( -2 \left( \frac{\D_+ y_{0,j}}{\D_+ y_{j}}\rho_{0,j} - \rho_{\infty} \right) \frac{\D_+ y_{0,j}}{\D_+ y_{j}}\rho_{0,j} \D_+(\delta y)_{j}\\
  &\quad + \left( \frac{\D_+ y_{0,j}}{\D_+ y_{j}}\rho_{0,j} - \rho_{\infty} \right)^{2} \D_+(\delta y)_{j} \Big) \\
                  &= \sumZxi{j}{\frac{1}{2} \D_- \left( \left(\frac{\D_+ y_{0,j}}{\D_+ y_{j}}\rho_{0,j}\right)^{2} - \rho_{\infty}^{2} \right)\delta y_{j}},
\end{align*}

which gives the Frechet derivative
\begin{equation*}
  \left(\vdiff{y}{\Epotdis}\right)_{j} = \frac{1}{2}\D_- \left( \left(\frac{\D_+ y_{0,j}}{\D_+ y_{j}}\rho_{0,j}\right)^{2} - \rho_{\infty}^{2} \right).
\end{equation*}
The Euler--Lagrange equation is then
\begin{equation}
  \vdiff{y}{\Lcaldis} - \diff{t}{} \vdiff{U}{\Lcaldis} = 0,
  \label{eq:dEL}
\end{equation}
see, e.g., \cite{arnold89}. From \eqref{eq:dEL} we then have
\begin{align*}
  \begin{aligned}
    \diff{t}{} \left(\vdiff{U}{\Lcaldis}\right)_{j} &= \diff{t}{}\left(U_{j} (\D_+ y_{j}) - \D_- \left( \frac{\D_+ U_{j}}{\D_+ y_{j}} \right)\right) \\
    &= \dot{U}_{j} (\D_+ y_{j}) - \D_- \left( \frac{\D_+ \dot{U}_{j}}{\D_+ y_{j}} \right) + U_{j} (\D_+ U_{j}) + \D_-\left( \left(\frac{\D_+ U_{j}}{\D_+ y_{j}}\right)^{2} \right),
  \end{aligned}
\end{align*}
which leads to the following system of governing equations
\begin{subequations}\label{eq:semidisc_sys}
  \begin{equation}
    \dot{y}_j = U_j
  \end{equation}
  and
  \begin{multline}
    (\D_+ y_{j}) \dot{U}_{j} - \D_- \left(\frac{\D_+ \dot{U}_{j}}{\D_+ y_{j}}\right) \\ = -U_{j} (\D_+ U_{j}) - \frac{1}{2}
    \D_-\left( (U_{j})^2 + \left(\frac{\D_+ U_{j}}{\D_+ y_{j}}\right)^2 + \left(\frac{\D_+ y_{0,j}}{\D_+ y_{j}}\rho_{0,j}\right)^2 \right),
  \end{multline}
\end{subequations}
for $j \in \Z$. Note that we have omitted $\rho_{\infty}^2$ on the right hand side in \eqref{eq:semidisc_sys} as $\D_-$ maps
constants to zero.

We can use the Legendre transform to define the Hamiltonian
\begin{equation}
  \label{eq:ham_dis}
  \Hcal_{\text{dis}} = \ipl{\vdiff{U}{\Lcaldis}}{U} - \Lcaldis.
\end{equation}
Writing out the above Hamiltonian explicitly we have
\begin{equation}
  \label{eq:HamiltonianDis}
  \Hcaldis = \frac12 \sumZxi{j}{\left( (U_{j})^2 + \left(\frac{\D_+ U_{j}}{\D_+ y_{j}}\right)^2 + \left(\frac{\D_+ y_{0,j}}{\D_+ y_{j}}\rho_{0,j} - \rho_{\infty} \right)^2 \right) (\D_+ y_{j})}.
\end{equation}
We observe that the Lagrangian $\Lcaldis$ does not depend explicitly on time. Then it is a classical result of mechanics, which follows from Noether's
theorem, that $\Hcaldis$ is time-invariant,
\begin{equation*}
  \diff{t}{\Hcaldis} = 0.
\end{equation*}
The Lagrangian $\Lcaldis$ is also invariant with respect to translation so that an other time invariant can be obtained. We denote by
$\psi:\ltwo\times\R\to\ltwo$ the transformation given by the uniform translation $(\psi(y, \epsi))_j = y_j + \epsi$. For
simplicity, we write $y^\epsi(t) = \psi(y(t), \epsi)$. From the definition of $\psi$, we have
\begin{equation*}
  \dot y^\epsi(t) = \dot y(t)\quad\text{ and } \quad \D_+y^\epsi(t) = \D_+y(t).
\end{equation*}
Hence, the Lagrangian $\Lcaldis$ is invariant with respect to the transformation $\psi$. Then Noether's theorem gives us that the
quantity $\ipl{\vdiff{U}{\Lcaldis}}{\vdiff{\epsi}{y^\epsi}}$ is preserved by the flow. In our case,
$\left(\vdiff{\epsi}{y^\epsi}\right)_j = 1$ and $\left(\vdiff{U}{\Lcaldis}\right)_j = U_{j} (\D_+ y_{j}) - \D_- \left( \frac{\D_+
    U_{j}}{\D_+ y_{j}} \right)$, see \eqref{eq:vdifflkin}. Thus, we obtain that the quantity
\begin{equation}\label{eq:I}
  I = \sumZxi{j}{\left( U_{j} (\D_+ y_{j}) - \D_- \left( \frac{\D_+ U_{j}}{\D_+ y_{j}} \right) \right) } = \sumZxi{j}{U_{j} (\D_+ y_{j})},
\end{equation}
is preserved. Note that $I$ corresponds to a discretization of
\begin{equation*}
  \int_\R (u - u_{xx})\,dx = \int_\R u \,dx
\end{equation*}
in Eulerian coordinates, which is preserved by the 2CH system.

Let us return to \eqref{eq:semidisc_sys}, and in particular to the left-hand side which contains $\dot U_j$, but not in an explicit
form. For a given sequence $a = \{a_j\}_{j\in\Z} \in \linfty$ and an arbitrary sequence $w = \{w_j\}_{j\in\Z} \in \ltwo$, we
define the operator $\A{a} : \ltwo \to \ltwo$ by
\begin{equation}
  (\A{a}w)_{j} \coloneqq a_j w_j - \D_- \left(\frac{\D_+ w_{j}}{a_{j}}\right), \quad j \in \Z.
  \label{eq:defAop}
\end{equation}
When $a = \D_+y$, \eqref{eq:defAop} corresponds to the momentum operator $m$ in Lagrangian coordinates,
and takes the form of a discrete Sturm--Liouville operator.
This operator is symmetric and positive definite for sequences $a$ such that $a_{j} > 0$, as we can see from
\begin{equation*}
  \sumZxi{j}{v_{j}(\A{a}w)_{j}} = \sumZxi{j}{\left(a_{j}w_{j}v_{j} + \frac{1}{a_{j}} (\D_+ w_{j}) (\D_+ v_{j}) \right)},
\end{equation*}
where we once more have used \eqref{eq:sum_by_parts}.
When $\A{\D_+ y}$ is positive definite, it is invertible and we may
formally write \eqref{eq:semidisc_sys} as a system of first order ordinary
differential equations,
\begin{align}
  \begin{aligned}
    \dot{y}_{j} &= U_{j}, \\
    \dot{U}_{j} &= - \A{\D_+y}^{-1} \left( U_j (\D_+ U_j) + \frac{1}{2} \D_-\left( (U_j)^2 + \left(\frac{\D_+ U_j}{\D_+ y_j}\right)^2 + \left(\frac{\D_+ y_{0,j}}{\D_+ y_{j}}\rho_{0,j}\right)^2 \right) \right).
  \end{aligned}
  \label{eq:semidisc_sys_formal}
\end{align}
When solving the above system, we obtain approximations of the fluid velocity and density in Lagrangian variables, $U_{j}(t)
\approx u(t,y(t,\xi_{j}))$ and $\rho_{0,j} / (\D_+ y_{j}(t)) \approx \rho(t,y(t,\xi_{j}))$.

We conclude this section with some comments on the Hamiltonian form of the equations. Hamiltonian equations in generalized position
and momentum variables follow from the Lagrangian approach in classical mechanics, see, e.g., \cite{arnold89}. The generalized
momentum is defined as $p=\vdiff{U}{\Lcaldis}(y, U)$. When we express the Hamiltonian $\Hcaldis$ given in
\eqref{eq:HamiltonianDis} in term of $y$ and $p$, the Hamiltonian equations are then given as usual by
\begin{equation}
  \label{eq:hamsys}
  \dot{y} = \vdiff{p}{\Hcaldis}, \quad \dot{p} = -\vdiff{y}{\Hcaldis}.
\end{equation}
From \eqref{eq:I}, we get that the momentum is $p_j = (\A{\D_+ y}U)_j$. Hence, $U_j = (\A{\D_+ y}^{-1}p)_j$, and the Hamiltonian
\eqref{eq:HamiltonianDis} is
\begin{equation*}
  \Hcaldis = \frac12 \sumZxi{j}{p_j (\A{\D_+ y}^{-1}p)_j} + \Epotdis.
\end{equation*}
If we introduce the fundamental solution $g_{i,j}$ of the operator $\A{\D_+y}$, see Section \ref{sec:aux}, the Hamiltonian
can be rewritten as
\begin{equation*}
  \Hcaldis = \frac12 \sumZxi{j}{p_j \sumZxi{i}{g_{i,j} p_i}} = \frac12 \sumZ{i,j}{(\dxi p_i)(\dxi p_j)g_{i,j} } + \Epotdis.
\end{equation*}
In the case $\rho = 0$ (that is $\Epotdis =0$), we recognized the similarity of this expression with
\begin{equation*}
  \HcalMP = \frac12 \sum_{i,j=1}^{N}p_i p_j e^{-|y_i-y_j|}
\end{equation*}
given in \cite{Camassa1993}. The Hamiltonian $\HcalMP$ defines the multipeakon solutions, which can be seen as another form of
discretization for the CH equation, see \cite{NumMP_HolRay2008} for the global conservative case. Then, the two discretization appear as the
results of two different choices of discretization for the inverse momentum operator: $g_{i,j}$ in the case of this paper
and $\hat g_{i,j} = e^{-|y_i-y_j|}$ in \cite{Camassa1993}.
We note that a numerical study of discretizations of the periodic CH equation considering both multipeakons and the variational method presented in this paper can be found in \cite{vardisc_num}.

\section{Construction of the fundamental solutions of the discrete momentum operator}
\label{sec:aux}

In this section we construct a Green's function, or fundamental solution, for the operator defined in \eqref{eq:defAop}.  Note that
when $a = \D_+ y$ coincides with the constant sequence $\mathbf{1} = \{1\}_{j\in\Z}$ we have from \eqref{eq:defAop} that $\A{\mathbf{1}} = \id - \D_{-}\D_{+}$, which corresponds to the
operator used in the difference schemes studied in \cite{Coclite2008, Holden2006}.  As the coefficients are
constant, the authors are able to find an explicit Green's function $g$ which can be written as
\begin{equation}
  g_{j} = \frac{1}{\sqrt{4+\dxi^2}} \left(1+\frac{\dxi^2}{2} + \frac{\dxi}{2}\sqrt{4+\dxi^2}\right)^{-|j|}
  \label{eq:greens_eul}
\end{equation}
and fulfills $(\id-\D_{-}\D_{+})g = \delta_{0}$. Here $\delta_{0} = \{\delta_{0,j}\}_{j\in\Z}$ for the Kronecker delta
$\delta_{i,j}$, equal to one when the indices coincide and zero otherwise. In our case, the coefficients appearing in the
definition of $\A{\D_+ y}$ are varying with the grid index $j$, which significantly complicates the construction of the Green's
function.

Let us consider the operator $\A{a}$ from \eqref{eq:defAop} and the equation $(\A{a}g)_j = f_j$. We want to prove that there
exists a solution which decreases exponentially as $j \to \pm\infty$.  To this end, we want to find a Green's function for the
operator $\A{a}$, and the first step is to realize that the homogeneous operator equation $(\A{a}g)_j = 0$ can be written as
\begin{equation*}
  \frac{\D_+ g_j}{a_j} = \dxi a_j g_j + \frac{\D_+ g_{j-1}}{a_{j-1}}.
\end{equation*}
This can again be restated as a \emph{Jacobi difference equation}, see \cite[Eq.\ (1.19)]{JacobiOperator},
\begin{equation*}
  -\frac{1}{a_j} g_{j+1} + \left( \frac{1}{a_j} + \frac{1}{a_{j-1}} + a_j (\dxi)^2 \right) g_j - \frac{1}{a_{j-1}}g_{j-1} = 0,
\end{equation*}
or equivalently in matrix form
\begin{equation}
  \begin{bmatrix}
    g_{j} \\ g_{j+1}
  \end{bmatrix} = \begin{bmatrix}
    0 & 1 \\ -\frac{a_j}{a_{j-1}} & 1 + \frac{a_j}{a_{j-1}} + (a_j \dxi)^2
  \end{bmatrix} \begin{bmatrix}
    g_{j-1} \\ g_{j}
  \end{bmatrix} \eqqcolon \tilde{A}_{j} \begin{bmatrix}
    g_{j-1} \\ g_{j}
  \end{bmatrix}.
  \label{eq:nonsym}
\end{equation}
Observe that $\tilde{A}_j$ is not symmetric and always contains positive, negative and zero entries under the assumption $a_j >
0$.  Moreover, $\tilde{A}_j$ is ill-defined when $a_{j-1} = 0$, which corresponds to the occurrence of a singularity in the system.
We want to allow for this in our discretization in order to obtain solutions globally in time.
If we go back to the first restatement of the operator equation and introduce the variable
\begin{equation}
  \gamma_j \coloneqq \frac{\D_+g_j}{a_j} = \frac{g_{j+1}-g_j}{a_j \dxi},
  \label{eq:defgamma}
\end{equation}
we get the following characterization of the homogeneous problem
\begin{equation}
  \begin{bmatrix}
    -\D_{+} & a_j \\ a_j & -\D_{-}
  \end{bmatrix} \begin{bmatrix}
    g_j \\ \gamma_{j}
  \end{bmatrix} = \begin{bmatrix}
    0 \\ 0
  \end{bmatrix},
  \label{eq:hom_g}
\end{equation}
or equivalently
\begin{equation}
  \begin{bmatrix}
    g_{j+1} \\ \gamma_j
  \end{bmatrix} = \begin{bmatrix}
    1 + (a_j \dxi)^2 & a_j \dxi \\ a_j \dxi & 1
  \end{bmatrix} \begin{bmatrix}
    g_j \\ \gamma_{j-1}
  \end{bmatrix} \eqqcolon A_j \begin{bmatrix}
    g_j \\ \gamma_{j-1}
  \end{bmatrix}.
  \label{eq:forw_g}
\end{equation}
Here $A_j$ is a symmetric matrix with positive entries whenever $a_j > 0$, and it reduces to the identity matrix when $a_j = 0$.
We will use \eqref{eq:forw_g} rather than \eqref{eq:nonsym} to construct our Green's function, and it will also significantly simplify
the analysis of the asymptotic behavior of the solutions.
\begin{lem}[Properties of matrix $A_j$] \label{lem:propAj}
  Consider $A_j$ from \eqref{eq:forw_g} and assume $a_j = 1 + \D_+b_j \ge 0$ where $\D_+b \in \ltwo$.
  Then $\det{A_j} = 1$ and there exist $M_b > m_b > 0$ depending on $\normltwo{\D_+b}$ and $\dxi$ such that the eigenvalues $\lambda^\pm_j$ of $A_j$ satisfy
  \begin{equation}
    m_b \le \lambda^-_j < 1 < \lambda^+_j \le M_b
    \label{eq:eigAjbnd}
  \end{equation}
  uniformly with respect to $j$ when $a_j > 0$.
  Moreover there is the obvious identity $\lambda^\pm_j = 1$ when $a_j = 0$.
  Asymptotically we have $\lim_{j \to \pm\infty}A_j = A$, where $A$ is given by $A_j$ after setting $a_j = 1$, and the eigenvalues $\lambda^\pm$ of $A$ satisfy
  \begin{equation}
    m \le \lambda^- < 1 < \lambda^+ \le M
    \label{eq:eigAbnd}
  \end{equation}
  for $M > m > 0$ depending only on $\dxi$.
  Moreover, as the eigenvalues are strictly positive it follows that the spectral radius of $A_j$, $\spr(A_j) \coloneqq \max\{|\lambda_j^+|, |\lambda_j^-|\}$ satisfies $\norm{A_j} = \spr(A_j) = \lambda^+_j$, $\norm{A} = \spr(A) = \lambda^+$, and both matrices can be diagonalized: $A_j = R_j \Lambda_j R_j^\tp$, $A = R \Lambda R^\tp$.
\end{lem}
\begin{proof}[Proof of Lemma \ref{lem:propAj}]
  To see that $\det{A_j} = 1$ one can compute it directly, or see it from the eigenvalues
  \begin{align}
    \begin{aligned}
      \lambda^\pm_j &\coloneqq 1 + \frac{(a_j \dxi)^2}{2} \pm \frac{a_j \dxi}{2}\sqrt{4 + (a_j \dxi)^2} \\
      &= \frac{1}{4}\left( \sqrt{4 + (a_j\dxi)^2} \pm a_j \dxi \right)^2,
    \end{aligned}
        \label{eq:eigAj}
  \end{align}
  which shows that $A_j$ is invertible irrespective of the value of $a_j$.
  As $A_j$ is real and symmetric, it can be diagonalized with orthonormal eigenvectors $r^\pm_j$ as follows
  \begin{equation}
    A_j = R_j \Lambda_j R_j^\tp, \quad \Lambda_j = \begin{bmatrix} \lambda^-_j & 0 \\ 0 & \lambda^+_j \end{bmatrix}, \quad R_j = \begin{bmatrix}  \ds\frac{1}{\sqrt{1 + \lambda^+_j}} & \ds\frac{1}{\sqrt{1+ \lambda^-_j}} \\ \ds-\frac{1}{\sqrt{1 + \lambda^-_j}} & \ds\frac{1}{\sqrt{1 + \lambda^+_j}} \end{bmatrix}.
    \label{eq:diagAj}
  \end{equation}
  Since $\D_+b \in \ltwo$, for any $j \in \Z$ we have the bound
  \begin{equation*}
  \sqrt{\dxi}\abs{\D_+b_j} = \left(\dxi \abs{\D_+b_j}^2 \right)^{1/2} \le \normltwo{\D_+b}
  \end{equation*}
  which leads to
  \begin{equation*}
    0 \le a_j\dxi \le \dxi +
    \sqrt{\dxi}\normltwo{\D_+b} \eqqcolon K_b,
  \end{equation*}
  meaning $a_j$ is bounded from above and below.
  Then it follows that
  \begin{equation*}
    0 < \left(\frac{\sqrt{4 + K_b^2}-K_b}{2}\right)^2 \le \lambda^-_j \le  1 \le  \lambda^+_j \le \left(\frac{\sqrt{4 + K_b^2}+K_b}{2}\right)^2 < (1+K_b)^2,
  \end{equation*}
  corresponding to \eqref{eq:eigAjbnd}.  Furthermore, since $\D_+b_j\in\ltwo$, we have $\lim_{j\to\pm\infty} a_j \dxi = \dxi$. We denote
  by $A$, $\Lambda$, $R$, and $\lambda^\pm$ the matrices and eigenvalues given by $A_j$, $\Lambda_j$, $R_j$, and $\lambda_j^\pm$
  after replacing $a_j$ by 1.  From the preceding limit, \eqref{eq:eigAj} and \eqref{eq:diagAj} we obtain
  \begin{equation}
    \lim\limits_{j\to\pm\infty} (A_j, \Lambda_j, R_j) = (A,\Lambda, R).
    \label{eq:matlim}
  \end{equation}
  Bounds for $\lambda^\pm$ are obtained similarly to the bounds for $\lambda^\pm_j$.
  As $A_j, A$ are symmetric and hence normal, their norms coincide with the spectral radius $\spr(\cdot)$ which here coincides with the largest eigenvalue.
\end{proof}
Note that \eqref{eq:forw_g} corresponds to a transition from $(g_j, \gamma_{j-1})$ to $(g_{j+1}, \gamma_j)$, so that $A_j$ can be
considered as a transfer matrix between these two states. Thus, solving the homogeneous operator equation $(\A{a}g)_j = 0$ bears
clear resemblance to propagating a discrete dynamical system, and this is also the idea employed in the analysis of Jacobi
difference equations in \cite[Eq.\ (1.28)]{JacobiOperator}.  However, in making the change of variable to obtain \eqref{eq:forw_g}
we lose the symmetry of the difference equation, and so the results in \cite{JacobiOperator} are no longer directly applicable.
On the other hand, our system can be regarded as a more general Poincar\'{e} difference system, and our idea is then to apply the
results \cite[Thm.\ 1.1]{Friedland2006} and \cite[Thm.\ 1]{Pituk2002} to the matrix product
\begin{equation}
  \Phi_{k,j} \coloneqq \begin{cases}
    A_{k-1}\dots A_j, & k > j, \\
    I, & k = j, \\
    (A_k)^{-1} \dots (A_{j-1})^{-1}, & k < j
  \end{cases},
  \label{eq:transfer}
\end{equation}
which is the transition matrix from $(g_j,\gamma_{j-1})$ to $(g_k, \gamma_{k-1})$.
Note that in the lemma below, the norms can be taken to be the standard Euclidean norm, but one could use any vector norm.

\begin{lem}[Existence of exponentially decaying solutions]\label{lem:Lyapunov}
  Consider the matrix equation
  \begin{equation}\label{eq:diff_sys}
    v_n = (\Phi_{n,0}) v_0, \qquad v_n = \begin{bmatrix}
      g_n \\ \gamma_{n-1}
    \end{bmatrix},
  \end{equation}
  coming from \eqref{eq:forw_g} with $\Phi_{n,0}$ as defined in \eqref{eq:transfer}.
  Then there exist initial vectors $v_0 = v_0^\pm$ such that the corresponding solutions $v^\pm_n$ satisfy
  \begin{equation}\label{eq:decay}
    \lim\limits_{n\to\mp\infty}\sqrt[n]{\norm{v^\pm_n}} = \lambda^-.
  \end{equation}
  That is, there exist solutions $v_n$ with exponential decay in either direction, owing to the Lyapunov exponent $\lambda^- < 1$.
  Moreover, the initial vectors are unique up to a constant factor.
\end{lem}
\begin{proof}[Proof of Lemma \ref{lem:Lyapunov}]
  We begin with the case of increasing $n$, and we want to apply \cite[Thm.\ 1.2]{Friedland2006} which states that for sequences of positive matrices $\{A_n\}$ satisfying $\lim_{n\to+\infty}A_n = A$ for some positive matrix $A$ we have
  \begin{equation}\label{eq:prodlim}
    \lim\limits_{n\to+\infty}\frac{A_{n} A_{n-1} \dots A_1 A_0}{\norm{A_n A_{n-1} \dots A_1 A_0}} = v w^\tp
  \end{equation}
  for some vectors $v$ and $w$ with positive entries such that $A v = \spr(A) v$.
  As mentioned in \cite[Rem.\ 4]{Borcea2011}, there is in general no easy way of determining the vector $w$ explicitly.
  
  We recall that our $A_n$ has positive entries, unless $a_n = 0$ in which case we have $A_n = I$.
  Because of \eqref{eq:matlim}, there can only be finitely many $n \ge 0$ for which $A_n$ reduces to the identity.
  If we instead consider the sequence of positive matrices consisting of our $\{A_n\}$ where we have omitted the finitely many identity matrices, they clearly still satisfy \eqref{eq:matlim} and so \eqref{eq:prodlim} holds with $\spr(A) = \lambda^+$ and $v = r^+$ from \eqref{eq:eigAj} and \eqref{eq:diagAj}.
  However, as the matrices we omitted were identities, it is clear that the limit in \eqref{eq:prodlim} for both sequences coincide.
  Hence, \cite[Thm.\ 1.1]{Friedland2006} holds for our nonnegative sequence as well.
  
  Now, as $A_n \ge I$ entrywise it follows that the entries of $\Phi_{n,0}$ are nondecreasing for $n \ge 0$, which means that $\norms{\Phi_{n,0}}$ is also nondecreasing for such $n$.
  Therefore, by \eqref{eq:prodlim} we have that any initial vector $v_0$ such that $w^\tp v_0 \neq 0$ leads to a solution $v_n$ with nondecreasing norm, and which then by \cite[Thm.\ 1]{Pituk2002} must satisfy
  \begin{equation}\label{eq:Lyapunov}
    \varrho = \lim_{n\to+\infty} \sqrt[n]{\norm{v_n}}
  \end{equation}
  with $\varrho = \lambda^+ > 1$, i.e., an asymptotically exponentially increasing solution.
  Indeed, the nondecreasing norm rules out the possibility of $v_n = 0$ for $n$ large enough.
  It follows that \eqref{eq:Lyapunov} holds for $\varrho$ equal to either $\lambda^+$ or $\lambda^-$, but if it were $\lambda^- < 1$, then $\norm{v_n}$ could not be nondecreasing.
  However, choosing instead a nonzero $v_0$ such that $w^\tp v_0 = 0$, we obtain an asymptotically exponentially decreasing solution $v_n$ satisfying \eqref{eq:Lyapunov} with $\varrho = \lambda^- < 1$.
  This follows by once more excluding the scenario of $v_n = 0$ for large enough $n$, since $v_0$ is nonzero and each $A_n$ has full rank.
  Then the only remaining possibility is $v_n$ satisfying \eqref{eq:Lyapunov} with $\varrho = \lambda^-$.
  An obvious choice of $v_0$ given $w = [w_1 \enspace w_2 ]^\tp$ is then $v_0 = [w_2 \enspace {-w_1}]^\tp$.
  
  For decreasing $n$, we will be able to reuse the arguments from above.
  From \eqref{eq:transfer} we find that $\Phi_{n,0}$ is a product of inverses of $A_n$ for $n < 0$, and by \eqref{eq:forw_g} we have
  \begin{equation*}
    \begin{bmatrix}
      g_j \\
      \gamma_{j-1}
    \end{bmatrix}
    = (A_j)^{-1}
    \begin{bmatrix}
      g_{j+1} \\
      \gamma_j
    \end{bmatrix} = \begin{bmatrix} 1 & -a_j \dxi \\ -a_j \dxi & 1 + (a_j \dxi)^2 \end{bmatrix} \begin{bmatrix}
      g_{j+1} \\
      \gamma_j
    \end{bmatrix}.
  \end{equation*}
  Since $(A_n)^{-1}$ contains entries of opposite sign, it would appear that we may not be able to use our previous argument. However, a change of variables will do the trick for us.
  First recall \eqref{eq:defgamma} which shows that $\gamma_{j}$ corresponds to a rescaled forward difference for $g_j$, hence its sign indicates whether $g$ is increasing or decreasing at index $j$.
  For an increasing solution in the direction of increasing $n$ it is then necessary for $g_n$ and $\gamma_{n-1}$ to share the same sign as $n \to +\infty$.
  On the other hand, for an increasing solution in the direction of decreasing $n$, the forward difference for $\gamma_{n-1}$ should have the opposite sign of $g_n$ as $n \to -\infty$.
  Therefore, a change of variables allows us to rewrite the previous equations as
  \begin{equation}\label{eq:backw_g}
    \begin{bmatrix}
      g_{j} \\ -\gamma_{j-1}
    \end{bmatrix} = \begin{bmatrix}
      1 & a_j \dxi \\ a_j \dxi & 1 + (a_j \dxi)^2
    \end{bmatrix} \begin{bmatrix}
      g_{j+1} \\ -\gamma_{j}
    \end{bmatrix} \eqqcolon B_j \begin{bmatrix}
      g_{j+1} \\ -\gamma_{j}
    \end{bmatrix},
  \end{equation}
  and
  \begin{equation*}
    \begin{bmatrix}
      g_{n} \\ -\gamma_{n-1}
    \end{bmatrix} = B_n \dots B_{-1} \begin{bmatrix}
      g_{0} \\ -\gamma_{-1}
    \end{bmatrix}, \quad n < 0
  \end{equation*}
  and for this system we may use the positive matrix technique from before.
  The eigenvalues of $B_j$ in \eqref{eq:backw_g} are the same as those of $A_j$, but they switch positions in the corresponding eigenvectors $\tilde{r}^\pm_j$ compared to $r^\pm_j$ of $A_j$:
  \begin{equation*}
    \tilde{r}^\pm_j =
    \begin{bmatrix}
      \ds \frac{1}{\sqrt{1+\lambda^\pm_j}} \\
      \ds\pm\frac{1}{\sqrt{1+\lambda^\mp_j}}
    \end{bmatrix},
    \qquad r^\pm_j =
    \begin{bmatrix}
      \ds \frac{1}{\sqrt{1+\lambda^\mp_j}} \\
      \ds\pm\frac{1}{\sqrt{1+\lambda^\pm_j}}
    \end{bmatrix}.
  \end{equation*}
  The same argument as in the case of increasing $n$ then proves the existence of $v_0$ giving exponentially decreasing solutions as $n \to -\infty$.
  
  The uniqueness follows from the uniqueness of limits in \eqref{eq:prodlim}, which for a given eigenvector $v$ of $A$ means that $w$ is unique up to a constant factor.
  But then again, since we are in $\R^2$, the vector orthogonal to $w$ is unique up to a constant factor.
\end{proof}

\begin{rem}[Signs of the initial vectors]\label{rem:vec_sgn}
  Here we underline that the form of $\Phi_{n,0}$ implies that the entries of $v^\pm_0$ in Lemma \ref{lem:Lyapunov} must be nonzero, with opposite signs for $v^-_0$ and same sign for $v^+_0$.
  Indeed, by \eqref{eq:diff_sys} and \eqref{eq:decay} we have
  \begin{equation*}
    \lim\limits_{n \to+\infty}\norm{(\Phi_{n,0})v_0^-} = 0.
  \end{equation*}
  Let us then assume $v_0^- \neq 0$ with nonnegative entries of the same sign, namely $v_0^- \ge 0 \: (v_0^- \le 0)$ understood entrywise.
  From the definition \eqref{eq:transfer} and $A_n \ge I$, it is clear that $(\Phi_{n,0})v_0^- \ge v_0^- \: ((\Phi_{n,0}) v_0^- \le v_0^- )$ for $n \ge 0$, and so it is impossible for the norm to tend to zero.
  Hence, the entries of $v_0^-$ must be nonzero and of opposite sign.
  For $n \to -\infty$, we can use \eqref{eq:backw_g} and the same argument to arrive at the same conclusion for $[g_0 \enspace {-\gamma_{-1}}]^\tp$, implying that $v^+_0 = [g_0 \enspace \gamma_{-1}]^\tp$ has nonzero entries of equal sign.
\end{rem}

\begin{thm}[Existence of a discrete Green's function]
  \label{thm:green_g}
  Assume $\{a_j\}_{j\in\Z}$ to be a nonnegative sequence such that $a_j = 1 + \D_+b_j$ with $\D_+b
  \in \ltwo$. Then, for any given index $i$, there exists a unique sequence
  $g_i = \{g_{i,j}\}_{j\in\Z}$ such that
  \begin{equation}
    (\A{a}g_i)_j = \frac{\delta_{i,j}}{\dxi}.
    \label{eq:fml_g}
  \end{equation}
\end{thm}

\begin{proof}
  Our strategy follows the standard approach for constructing Green's functions: We first construct solutions of the homogeneous
  version of \eqref{eq:fml_g} with exponential decay, and then we combine them in order to obtain a delta
  function at a given point. We start by constructing $g_{0,j}$ centered at $i=0$.
  
  Choosing $v^\pm_0$ from Lemma \ref{lem:Lyapunov} we set
  \begin{equation}\label{eq:initvecs}
    \begin{bmatrix}
      g^-_0 \\ \gamma^-_{-1}
    \end{bmatrix} \coloneqq v_0^-, \qquad
    \begin{bmatrix}
      g^+_0 \\ \gamma^+_{-1}
    \end{bmatrix} \coloneqq  v_0^+,
  \end{equation}
  and define the sequences
  \begin{equation}\label{eq:ggamma_pm}
    \begin{bmatrix}
      g^-_n \\
      \gamma^-_{n-1}
    \end{bmatrix}
    \coloneqq
    \Phi_{n,0} \begin{bmatrix}
      g^-_0 \\ \gamma^-_{-1}
    \end{bmatrix},
    \quad
    \begin{bmatrix}
      g^+_n \\
      \gamma^+_{n-1}
    \end{bmatrix} \coloneqq \Phi_{n,0} \begin{bmatrix}
      g^+_0 \\ \gamma^+_{-1}
    \end{bmatrix}, \quad n \in \Z,
  \end{equation}
  where by construction $g^\pm, \gamma^\pm$ have exponential decay for $n \to \mp\infty$.
  Then, applying the operator $\A{a}$ to $g^\pm$ we find
  \begin{equation*}
    (\A{a}g^\pm)_j = a_j g^\pm_j -\D_-\gamma^\pm_j = 0, \qquad j \in \Z
  \end{equation*}
  by construction of $g^\pm$ and $\gamma^\pm$.
  Let us then define
  \begin{equation}
    \label{eq:g_proto}
    g_{0,j} \coloneqq C \begin{cases}
      g^-_j g^+_0, & j \ge 0, \\ g^+_j g^-_0, & j < 0,
    \end{cases} \quad
    \gamma_{0,j} \coloneqq C \begin{cases}
      \gamma^-_j g^+_0, & j \ge 0, \\ \gamma^+_j g^-_0, & j < 0
    \end{cases}
  \end{equation}
  for some hitherto unspecified constant $C$, and observe from the homogeneous equation that $a_j g_{0,j} - \D_-\gamma_{0,j} = 0$
  for $j \neq 0$.  Moreover, we have $\D_+g_{0,j} = a_j \gamma_{0,j}$ for all $j$ by construction.  Now we would like to show that
  the constant $C$ can be chosen to obtain $\A{a}g_{0,0} = 1/\dxi$.
  From \eqref{eq:sum_by_parts}, we get
  \begin{equation}
    \dxi \sum_{j=m}^{n}g^+_j(\A{a}g^-)_j - \dxi \sum_{j=m}^{n}(\A{a}g^+)_j g^-_j = W_n(g^-,g^+) - W_{m-1}(g^-,g^+),
    \label{eq:GreenFormula}
  \end{equation}
  where we in the spirit of \cite[Eq.\ (1.21)]{JacobiOperator} have defined a discrete Wronskian
  \begin{equation}
    W_n(g^-,g^+) \coloneqq g^-_{n+1} \gamma^+_n - g^+_{n+1} \gamma^-_n = g^-_{n} \gamma^+_n - g^+_{n} \gamma^-_n,
    \label{eq:wron}
  \end{equation}
  and the last equality follows from the identity $g^\pm_{n+1} = g^\pm_n + \dxi a_n \gamma^\pm_n$.
  Since the left-hand side of \eqref{eq:GreenFormula} vanishes by definition of $g^\pm$, we have $W_n(g^-,g^+) = W_{m-1}(g^-,g^+)$ for any $n, m \in \Z$.
  That is, the Wronskian $W_n(g^-,g^+)$ is a constant $W(g^-,g^+)$ for the constructed sequences $g^+$ and $g^-$.
  
  Next, we want to show that the Wronskian is nonzero.
  Considering
  \begin{equation*}
    W(g^-,g^+) = W_{-1}(g^-,g^+) = g^-_{0} \gamma^+_{-1} - g^+_{0}\gamma^-_{-1} = g^+_{0} \gamma^-_{-1} + g^-_{0} (-\gamma^+_{-1})
  \end{equation*}
  and the definition \eqref{eq:initvecs}, we use the sign properties stated in Remark \ref{rem:vec_sgn}  to conclude that the two terms in the final sum are always nonzero and of the same sign, implying $W(g^-,g^+) \neq 0$.
  Finally, we will determine the constant $C$ by considering the backward difference
  \begin{align*}
    \D_{j-}\gamma_{0,0} &= C \frac{\gamma^-_0 g^+_0 - \gamma^+_{-1}g^-_0}{\dxi} = C \frac{\gamma^-_0 g^+_0 - \gamma^-_{-1} g^+_0 + \gamma^-_{-1} g^+_0 - \gamma^+_{-1} g^-_0 }{\dxi} \\
                        &= C g^+_0 a_0 g^-_0 - C \frac{W_{-1}(g^-,g^+)}{\dxi} = a_0 g_{0,0} - C \frac{W(g^-,g^+)}{\dxi},
  \end{align*}
  which leads to
  \begin{equation*}
    (\A{a}g_0)_0 = a_0 g_{0,0} - \D_-\gamma_{0,0} = C \frac{W(g^-,g^+)}{\dxi}.
  \end{equation*}
  Consequently, setting $C^{-1} = W(g^-,g^+)$ in \eqref{eq:g_proto} gives the desired Green's function.
  
  Note that there is nothing special about the index $i=0$ where we centered the Green's function.
  We can simply use the sequences \eqref{eq:ggamma_pm} from before and define
  \begin{equation}
    \label{eq:g}
    g_{i,j} = \frac{1}{W(g^-,g^+)} \begin{cases}
      g^+_j g^-_i, & j \ge i, \\ g^-_j g^+_i, & j < i,
    \end{cases} \qquad
    \gamma_{i,j} = \frac{1}{W(g^-,g^+)} \begin{cases}
      \gamma^+_j g^-_i, & j \ge i, \\ \gamma^-_j g^+_i, & j < i
    \end{cases}
  \end{equation}
  to obtain a Green's function $g_{i,j}$ centered at an arbitrary $i$.
  
  The uniqueness of $g_{i,j}$ follows from the vectors $v_0^\pm$ in Lemma \ref{lem:Lyapunov} being uniquely defined up to constant factors.
  Indeed, when constructing the Green's function in \eqref{eq:g} these factors disappear since we are dividing by the Wronskian $W(g^-,g^+)$, and so we have no degrees of freedom left in our construction of $g_{i,j}$, hence it is unique.
\end{proof}

Note that $\A{a}$ is not the only way to discretize the operator
\begin{equation*}
  a(\xi) \id - \pdiff{\xi}{} \frac{1}{a(\xi)} \pdiff{\xi}{}
\end{equation*}
with first order differences, we may also consider
\begin{equation}
  \label{eq:defBop}
  (\B{a}k)_j \coloneqq a_j k_j - \D_{+}\left(\frac{\D_{-}k_j}{a_j}\right).
\end{equation}
In fact, we will need the Green's function for this operator as well to close our upcoming system of differential equations.
Fortunately, the existence of Green's function for \eqref{eq:defBop} follows from the considerations already made in Theorem
\ref{thm:green_g}.
\begin{cor}\label{cor:green_k}
  Under the same assumptions on $\{a_j\}_{\j \in \Z}$ as in Theorem \ref{thm:green_g}, for any given index $i$ there exists a unique sequence $k_i = \{k_{i,j}\}_{\jmath\in\Z}$ such that
  \begin{equation}
    (\B{a}k_i)_j = \frac{\delta_{i,j}}{\dxi}.
    \label{eq:fml_k}
  \end{equation}
\end{cor}
\begin{proof}[Proof of Corollary \ref{cor:green_k}]
  Manipulating the homogeneous version of \eqref{eq:fml_k} we find it to be equivalent to
  \begin{equation*}
    \frac{\D_-k_{j+1}}{a_{j+1}} = \dxi a_j k_j + \frac{\D_-k_{j}}{a_{j}}.
  \end{equation*}
  Introducing
  \begin{equation}
    \kappa_j = \frac{\D_-k_j}{a_j} = \frac{k_{j}-k_{j-1}}{a_j \dxi},
    \label{eq:defkappa}
  \end{equation}
  the previous equation can be written as
  \begin{equation*}
    \begin{bmatrix}
      \kappa_{j+1} \\
      k_{j}
    \end{bmatrix}
    =
    \begin{bmatrix}
      1 + (a_j \dxi)^2 & a_{j} \dxi \\
      a_{j} \dxi & 1
    \end{bmatrix}
    \begin{bmatrix}
      \kappa_{j} \\
      k_{j-1}
    \end{bmatrix}
    = A_j
    \begin{bmatrix}
      \kappa_{j} \\
      k_{j-1}
    \end{bmatrix},
  \end{equation*}
  where we recognize the matrix $A_j$ from \eqref{eq:forw_g}.
  Going backward we find
  \begin{equation*}
    \begin{bmatrix}
      \kappa_{j} \\
      k_{j-1}
    \end{bmatrix} = \begin{bmatrix}
      1 & -a_{j} \dxi \\
      -a_{j} \dxi & 1 + (a_j \dxi)^2
    \end{bmatrix} \begin{bmatrix}
      \kappa_{j+1} \\
      k_{j}
    \end{bmatrix},
  \end{equation*}
  or equivalently
  \begin{equation*}
    \begin{bmatrix}
      -\kappa_{j} \\
      k_{j-1}
    \end{bmatrix} = \begin{bmatrix}
      1 & a_{j} \dxi \\
      a_{j} \dxi & 1 + (a_j \dxi)^2
    \end{bmatrix} \begin{bmatrix}
      -\kappa_{j+1} \\
      k_{j}
    \end{bmatrix} = B_{j} \begin{bmatrix}
      -\kappa_{j+1} \\
      k_{j}
    \end{bmatrix}
  \end{equation*}
  with $B_{j}$ from \eqref{eq:backw_g}.
  Hence, we get the solution for free from \ref{thm:green_g}.
  Indeed, choosing
  \begin{equation*}
    \begin{bmatrix}
      \kappa^-_{n} \\ k^-_{n-1}
    \end{bmatrix} = \begin{bmatrix}
      g^-_{n} \\ \gamma^-_{n-1}
    \end{bmatrix}, \qquad \begin{bmatrix}
      -\kappa^+_{n} \\ k^+_{n-1}
    \end{bmatrix} = \begin{bmatrix}
      g^+_{n} \\ -\gamma^+_{n-1}
    \end{bmatrix}
  \end{equation*}
  we know that these sequences have the correct decay at infinity.
  Defining
  \begin{align}
    \label{eq:k}
    \begin{aligned}
      k_{i,j} &= \frac{1}{W(g^-,g^+)} \begin{cases}
        k^-_j k^+_i, & j > i, \\ k^+_j k^-_i, & j \le i,
      \end{cases} = \frac{-1}{W(g^-,g^+)} \begin{cases}
        \gamma^-_j \gamma^+_i, & j > i, \\ \gamma^+_j \gamma^-_i, & j \le i,
      \end{cases} \\
      \kappa_{i,j} &= \frac{1}{W(g^-,g^+)} \begin{cases}
        \kappa^-_j k^+_i, & j > i, \\ \kappa^+_j k^-_i, & j \le i,
      \end{cases} = \frac{-1}{W(g^-,g^+)} \begin{cases}
        g^-_j \gamma^+_i, & j > i, \\ g^+_j \gamma^-_i, & j \le i,
      \end{cases}
    \end{aligned}
  \end{align}
  it follows from \eqref{eq:hom_g} that $(\B{a}k_i)_j = a_j k_{i,j} - \D_{j+}\kappa_{i,j} = 0$ for $j \neq i$.
  Moreover, by the constancy of \eqref{eq:wron} we find $(\B{a}k_i)_i = 1/\dxi$ in the same way as for $(\A{a} g_i)_i$.
\end{proof}

\begin{rem} \label{rem:greens}
  Note that we may observe directly from \eqref{eq:g} and \eqref{eq:k} that $g_{i,j} = g_{j,i}$, $k_{i,j} = k_{j,i}$, and $\kappa_{i,j} = -\gamma_{j,i}$.
  Moreover, the eigenvalues
  \begin{equation*}
    \lambda^{\pm} = \frac{1}{2}\left(2 + \dxi^2 \pm \dxi\sqrt{4 + \dxi^2}\right)
  \end{equation*}
  are exactly those found in \eqref{eq:greens_eul} for the operator $\id-\D_{-}\D_{+}$.
  In fact, for $a_j \equiv 1$ the sequences $g$ and $k$ coincide since $\D_-\D_+ = \D_+\D_-$, and their explicit expression \eqref{eq:greens_eul} can be recovered from the columns of $\Lambda^n R^{-1}$ in the diagonalization $A^n = R \Lambda^n R^{-1}$.
\end{rem}

Observe that by \eqref{eq:defgamma} and \eqref{eq:defkappa} we can rewrite \eqref{eq:fml_g}
and \eqref{eq:fml_k} in the compact form
\begin{equation}
  \label{eq:fundrelggkk}
  \begin{bmatrix}
    -\D_{j-} & a_j \\
    a_j & -\D_{j+}
  \end{bmatrix}
  \begin{bmatrix}
    \gamma_{i,j}&k_{i,j}      \\
    g_{i,j} &\kappa_{i,j} 
  \end{bmatrix}
  = \frac{1}{\dxi}
  \begin{bmatrix}
    \delta_{i,j}& 0\\
    0 & \delta_{i,j}
  \end{bmatrix}.
\end{equation}

\begin{lem}[Sign properties of the discrete Green's functions]\label{lem:gk_sgn}
  Assume $a_j \ge 0$ for $j \in \Z$, and let $g$, $\gamma$, $k$, and $\kappa$ be solutions of \eqref{eq:fundrelggkk} which decay to zero for $\abs{j-i} \to +\infty$.
  Then the following sign properties hold,
  \begin{enumerate}[(i)]
  \item
    $g_{i,j} > 0$ and $k_{i,j} > 0$ for $j \in \Z$, \label{lem:gk_sgn:pos}
  \item
    $\sgn(\gamma_{i,j}) = \sgn(i-j-1/2)$ and $\sgn(\kappa_{i,j}) =
    \sgn(i-j+1/2)$. \label{lem:gk_sgn:sgn}
  \end{enumerate}
  In particular, this leads to the monotonicity properties
  \begin{equation}\label{eq:gk_mono}
    \max_{j\in\Z}g_{i,j} = g_{i,i}, \quad \lim_{|j-i|\to+\infty}g_{i,j} \searrow 0, \quad \max_{j\in\Z}k_{i,j} = k_{i,i}, \quad \lim_{|j-i|\to+\infty}k_{i,j} \searrow 0,
  \end{equation}
  where the arrows denote monotone decrease.
\end{lem}
In Figure \ref{fig:gk} we have included a sketch of $g_{i,n}$, $\gamma_{i,n}$, $k_{i,n}$, and $\kappa_{i,n}$ for $\dxi = 0.2$, $i = 0,4$ and $a_n = a(n\dxi)$ given by
\begin{equation}\label{eq:a_ex}
  a(\xi) = \begin{cases}
    2, & -1 < \xi \le 0.5, \\ 0, & 0.5 < \xi \le 1, \\ 4, & 1 < \xi \le 1.5 \\ 1, & \text{otherwise}.
  \end{cases}
\end{equation}
We say sketch, as they have been computed on a finite grid $n \in \{-20,\dots,20\}$ with boundary conditions $\gamma_{i,-21} = g_{i,21} = 0$ and $k_{i,-21} = \kappa_{i,21} = 0$, and consequently we find that neither of $g_{i,-20}$, $\gamma_{i,20}$, $\kappa_{i,-20}$ or $k_{i,20}$ are exactly zero.
However, the exponential decay makes them very small and the qualitative behavior indicated in Lemma \ref{lem:gk_sgn} is still the same.
Note how $a(\xi)$ being zero on the interval $(0.5,1]$ leads to constant kernel values in that neighborhood, even at the peaks for the kernels centered at $\xi_4 = 0.8$.

\begin{figure}
  \begin{subfigure}{0.45\textwidth}
    \includegraphics[width=\textwidth]{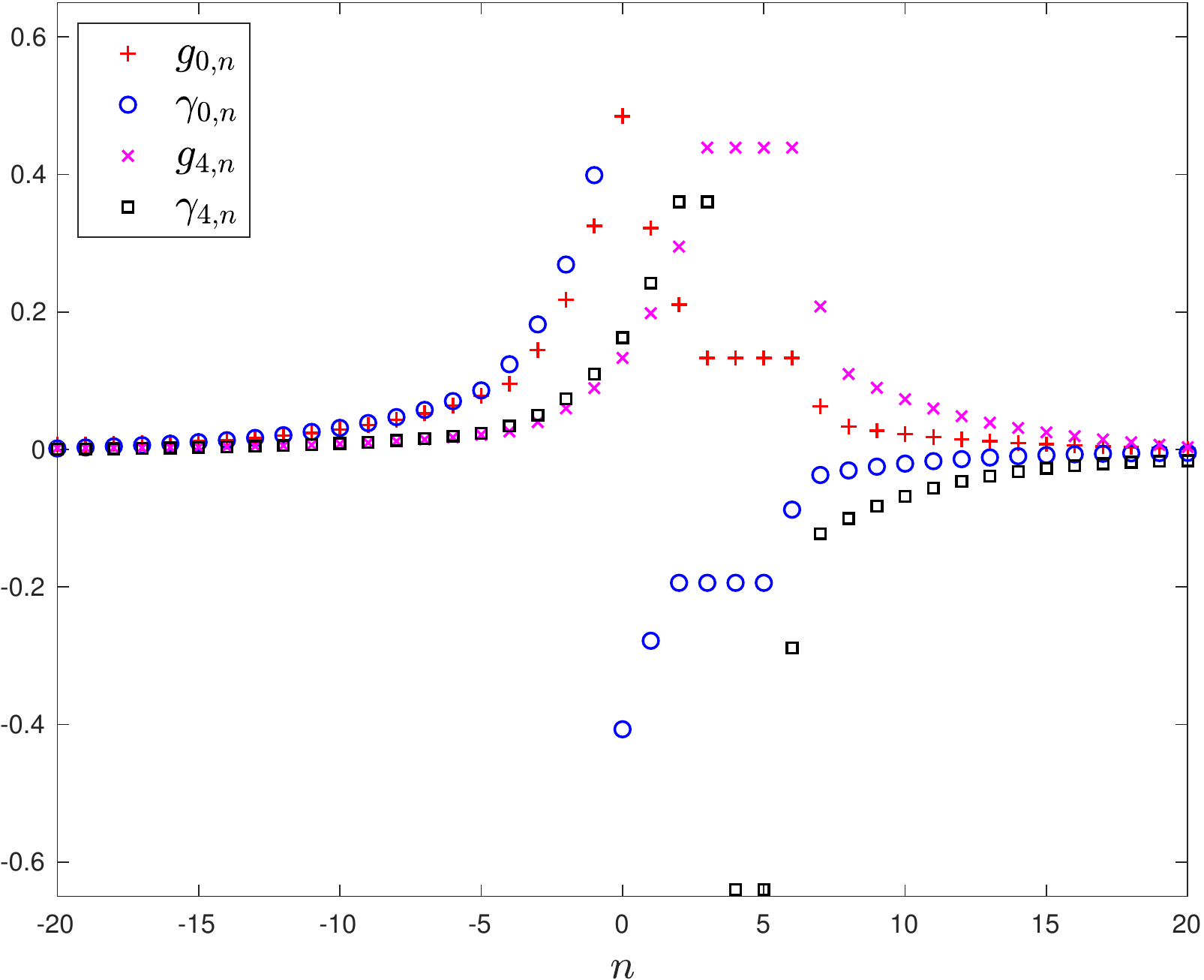}
  \end{subfigure}
  \begin{subfigure}{0.45\textwidth}
    \includegraphics[width=\textwidth]{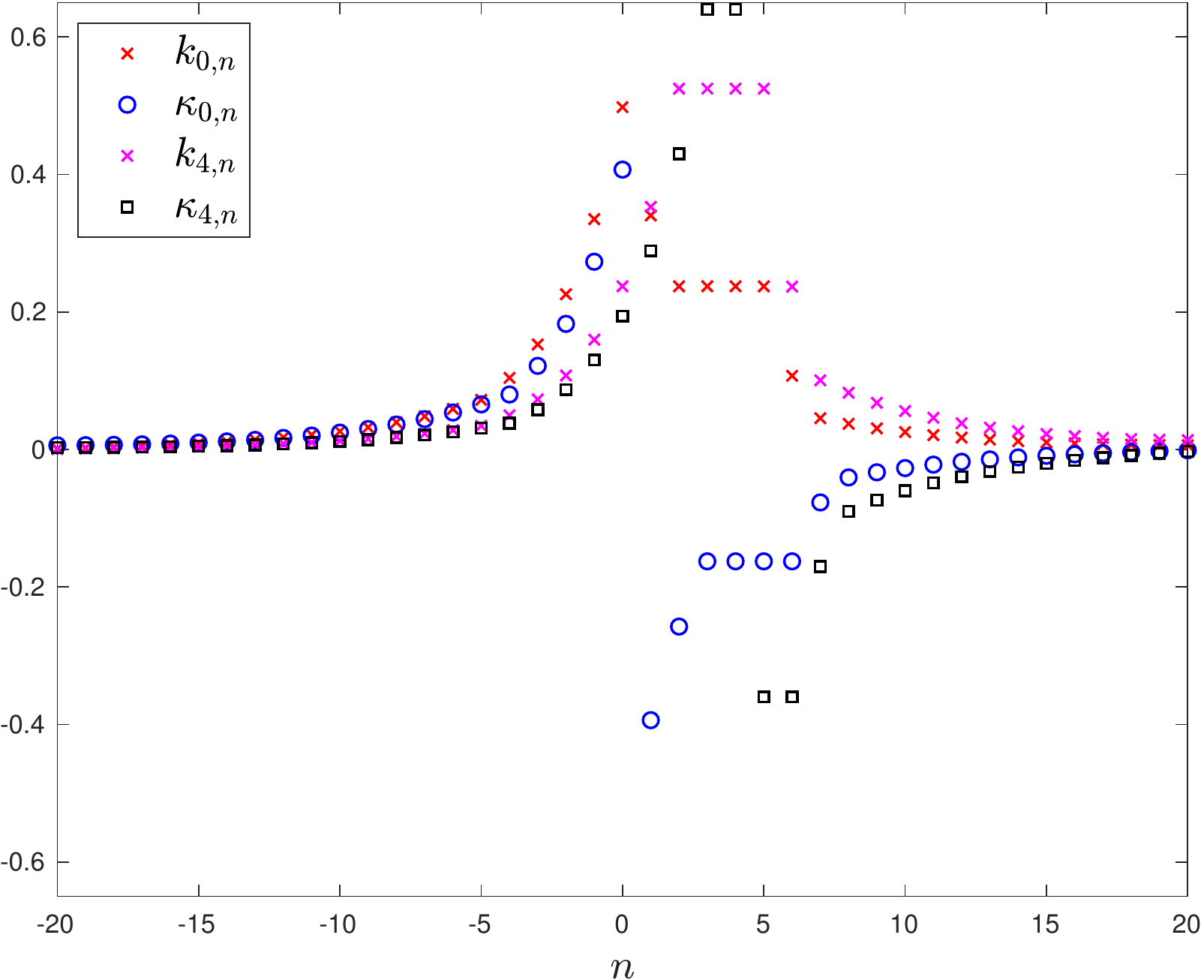}
  \end{subfigure}
  \caption{Sketch of $g_{i,n}$, $\gamma_{i,n}$, $k_{i,n}$, and $\kappa_{i,n}$ for $\dxi = 0.2$, $i = 0,4$ and $a_n = a(n\dxi)$ for $a(\xi)$ defined in \eqref{eq:a_ex}. Note the jump of size $-1 + \Ocal(\dxi)$ at $n = i$ for both $\gamma$ and $\kappa$.}
  \label{fig:gk}
\end{figure}

\begin{proof}[Proof of Lemma \ref{lem:gk_sgn}]
  We prove this only for $g$ and $\gamma$ as the proof for $k$ and $\kappa$ is similar.
  The proof relies on the reasoning in Remark \ref{rem:vec_sgn}.
  
  As a first step we want to show that the properties \eqref{lem:gk_sgn:pos} and \eqref{lem:gk_sgn:sgn} hold for $g_{i,i}$, $g_{i,i+1}$, $\gamma_{i,i-1}$, and $\gamma_{i,i}$.
  To this end, we recall from the proof of Theorem \ref{thm:green_g} that since $g_{i,j}$ and $\gamma_{i,j}$ satisfy \eqref{eq:fundrelggkk}, they must also satisfy
  \begin{equation*}
    \begin{bmatrix}
      g_{i,j} \\ -\gamma_{i,j-1}
    \end{bmatrix} =
     B_j \cdots B_{i-1} \begin{bmatrix}
      g_{i,i} \\ -\gamma_{i,i-1}
    \end{bmatrix}, \qquad j \le i-1 
  \end{equation*}
  and
  \begin{equation*}
    \begin{bmatrix}
      g_{i,j} \\ \gamma_{i,j-1}
    \end{bmatrix} =
    A_{j-1} \cdots A_{i+1} \begin{bmatrix}
      g_{i,i+1} \\ \gamma_{i,i}
    \end{bmatrix}, \qquad j \ge i+2,
  \end{equation*}
  with $A_k$ and $B_k$ as defined in \eqref{eq:forw_g} and \eqref{eq:backw_g}.
  By our assumptions, the Green's functions must tend to zero asymptotically, and we recall from Remark \ref{rem:vec_sgn} that a necessary condition for this is for the vectors $[g_{i,i}, -\gamma_{i,i-1}]^\tp$ and $[g_{i,i+1},\gamma_{i,i}]^\tp$ to have entries of opposite sign.
  Hence, $g_{i,i} \gamma_{i,i-1} > 0$ and $g_{i,i+1} \gamma_{i,i} < 0$, where we stress the importance of $a_j \ge 0$ for this argument to hold.
  Using only \eqref{eq:fundrelggkk} we calculate
  \begin{align*}
    0 &> g_{i,i+1} \gamma_{i,i} - g_{i,i} \gamma_{i,i-1} \\
      &= \dxi \frac{(g_{i,i+1}-g_{i,i}) \gamma_{i,i} + g_{i,i} (\gamma_{i,i}-\gamma_{i,i-1})}{\dxi} \\
      &= \dxi \left[ a_i \gamma_{i,i} \gamma_{i,i} + g_{i,i} \left[ a_i g_{i,i} - \frac{1}{\dxi} \right]  \right] \\
      &= \dxi a_i \left[ (g_{i,i})^2 + (\gamma_{i,i})^2 \right] - g_{i,i}.
  \end{align*}
  Since $a_j \ge 0$, it follows that $g_{i,i} \ge 0$.
  Recalling that $g_{i,i}$ must be nonzero according to the sign requirements, we necessarily have $g_{i,i} > 0$, and then $\gamma_{i,i-1} > 0$ follows.
  Moreover, multiplying the identity $g_{i,i+1} - \dxi a_i \gamma_{i,i} = g_{i,i}$ by $g_{i,i+1}$ and using $a_i \ge 0$, $g_{i,i} > 0$, and $g_{i,i+1} \gamma_{i,i} < 0$, we must have $g_{i,i+1} > 0$, which then implies $\gamma_{i,i} < 0$.
  
  Next we must prove that \eqref{lem:gk_sgn:pos} and \eqref{lem:gk_sgn:sgn} hold for the remaining values of $j$, and this will be achieved with a contradiction argument.
  We define the vectors
  \begin{equation*}
    v^+_{j} \coloneqq \begin{bmatrix}
      g_{i,j} \\ \gamma_{i,j-1} \end{bmatrix}, \qquad v^-_{j} \coloneqq \begin{bmatrix}
      g_{i,j+1} \\ -\gamma_{i,j}
    \end{bmatrix}
  \end{equation*}
  such that $v^+_{i+1}$ and $v^-_{i-1}$ both have positive first component and negative second component, and satisfy
  \begin{equation*}
    v^+_{j+1} \coloneqq A_{j} v^+_{j} \text{ for } j \ge i+1, \qquad v^-_{j-1} \coloneqq B_{j} v^-_{j} \text{ for } j \le i-1.
  \end{equation*}
  If we can prove that they retain the sign property under the above propagation, then we are done.
  Let us consider
  \begin{equation*}
    v^+_{j+1} \coloneqq A_{j} v^+_{j}, \qquad j \ge i+1.
  \end{equation*}
  Assume that $v^+_{j}$ does not retain the sign property, then there is some $k \ge i+1$ which is the first index such that $v^+_{k+1}$ does not have a positive first component and negative second component.
  We consider the two possible cases.
  
  The first case is $v^+_{k+1} \ge 0$ ($v^+_{k+1} \le 0$) considered entrywise.
  First of all, $v^+_{k+1}$ cannot be the zero vector as $A_k$ has full rank, since then $v^+_{k}$ would also have to be zero which contradicts $k+1$ being the first problematic index.
  Otherwise, the entrywise inequality $A_{k+1} \ge I$ leads to $v^+_{k+2} = A_{k+1} v^+_{k+1} \ge v^+_{k+1}$ ($v^+_{k+2} \le v^+_{k+1}$), and thus $\lim_{n \to+\infty}v^+_{n} \ge v^+_{k+1}$ ($\lim_{n \to+\infty}v^+_{n} \le v^+_{k+1}$).
  This is however impossible, as it contradicts the assumed decay of the Green's functions.
  
  The remaining case is that the entries interchange sign from $v^+_k$ to $v^+_{k+1}$.
  However, then we would have
  \begin{equation*}
    v^+_{k} = (A_{k})^{-1} v^+_{k+1} = \begin{bmatrix} 1 & -a_k \dxi \\ -a_k \dxi &	1 + (a_k \dxi)^2 \end{bmatrix} v^+_{k+1}.
  \end{equation*}
  Since $a_k \ge 0$, $v^+_{k}$ would also have negative first component and positive second component, which contradicts $k+1$ being the first problematic index.
  Hence, $v^+_{j}$ always has positive first component and negative second component for $j > i$, thus for $j \ge i$ it follows that $g_{i,j}$ is always positive, while $\gamma_{i,j}$ is always negative which shows that $g_{i,j}$ is decreasing in this direction.
  
  A similar argument holds in the other direction when considering $v^-_j$ and $B_j$.
  Then $-\gamma_{i,j}$ is always negative for $j<i$, which means that $g_{i,j}$ is increasing with $j$ for these indices.
  Thus, \eqref{lem:gk_sgn:pos} and \eqref{lem:gk_sgn:sgn} hold for $\{g_{i,j}\}_{j\in\Z}$ and $\{\gamma_{i,j}\}_{j\in\Z}$.
\end{proof}

\section{An equivalent semi-discrete system for global solutions in time}
\label{sec:analysis}
We now return to the initial value problem \eqref{eq:2CH}. We use the Lagrangian formulation introduced in earlier works, see
\cite{Grunert2012}, but reformulate the governing equations by propagating the fundamental solutions of the momentum operator.

\subsection{Reformulation of the continuous problem using operator propagation}
The 2CH system can be written as
\begin{equation*}
  u_t + uu_x + P_x = 0,\quad   \rho_t + (u\rho)_x = 0
\end{equation*}
for $P$ implicitly defined by
\begin{equation}
  \label{eq:defPsys}
  P - P_{xx} = u^2 + \frac12u_x^2 + \frac12\rho^2.
\end{equation}
Let us introduce $\bar{\rho} \coloneqq \rho - \rho_{\infty}$ to ease notation. Note that most expressions simplify when
we consider $\rho_\infty=0$. We have chosen to cover the case of arbitrary $\rho_\infty$, to allow for the initial condition
$\rho(0,x) = \epsi$, for any $\epsi>0$.  Such initial data lead to solutions without blow-up, see \cite{Grunert2015}. In the case
of the 2CH system, the conservation law for the energy is given by
\begin{equation}
  \label{eq:consenergsys}
  (\tfrac12(u^2 + u_x^2 + \bar{\rho}^2))_t + (u\tfrac12(u^2 + u_x^2 + \bar{\rho}^2))_x + (uR)_x = 0,
\end{equation}
where we have used $P$ from \eqref{eq:defPsys} to define
\begin{equation}
  \label{eq:defRsys}
  R = P - \tfrac12 u^2 - \tfrac12\rho_\infty^2.
\end{equation}
We can check that the first order system
\begin{equation}
  \label{eq:defRQsys}
  \begin{bmatrix}
    -\partial_x & 1 \\
    1 & -\partial_x
  \end{bmatrix}
  \circ
  \begin{bmatrix}
    R\\Q
  \end{bmatrix}
  =
  \begin{bmatrix}
    uu_x\\\tfrac12(u^2 + u_x^2 + \bar{\rho}^2) + \rho_\infty \bar{\rho}
  \end{bmatrix}
\end{equation}
is equivalent to \eqref{eq:defPsys}. Hence,
\begin{subequations}
  \label{eq:newequivchsys}
  \begin{align}
    u_t + uu_x + Q &= 0,\\
    \rho_t + (u\rho)_x &= 0
  \end{align}
\end{subequations}
and \eqref{eq:defRQsys} is yet another form of the 2CH system.

We introduce as before the Lagrangian position $y(t,\xi)$ and velocity $U(t,\xi)$. Moreover, we define the Lagrangian density
$r(t,\xi) \coloneqq \rho(t,y(t,\xi))y_{\xi}(t,\xi)$, and the cumulative energy $H$ given by
\begin{equation}
  H(t,\xi) = \frac12 \int_{-\infty}^{y(t,\xi)} ( u^2 + u_x^2 + \bar{\rho}^2 )(t,x)\,dx = \frac12 \int_{-\infty}^{\xi} ( ( u^2 +
  u_x^2 + \bar{\rho}^2 ) \circ y ) y_\xi(t,\eta) \,d\eta,
\end{equation}
as well as the Lagrangian variables $\bar Q = Q\circ y$ and $\bar R = R\circ y$. From \eqref{eq:newequivchsys}, we get $U_t =
-\bar Q$ and $r_t = 0$, while the conservation of energy \eqref{eq:consenergsys} yields $H_t = -U\bar R$. Finally, we rewrite
the system \eqref{eq:defRQsys} in terms of the Lagrangian variables. To simplify the notation, we replace $\bar Q$ by $Q$, and
similarly for $\bar R$. The equivalent system in Lagrangian variables is then given by
\begin{subequations}
  \label{eq:old2CH}
  \begin{align}
    y_t &= U, \label{eq:old2CH_y}\\
    U_t &= -Q, \label{eq:old2CH_U}\\
    H_t &= -U R, \label{eq:old2CH_H} \\
    r_t &= 0, \label{eq:old2CH_r}
  \end{align}
\end{subequations}
with
\begin{equation}
  \label{eq:RQsysdeflag}
  \begin{bmatrix}
    -\partial_\xi & y_\xi\\
    y_\xi & -\partial_\xi
  \end{bmatrix}
  \circ
  \begin{bmatrix}
    R\\Q
  \end{bmatrix}
  =
  \begin{bmatrix}
    UU_\xi\\ H_\xi + \rho_\infty (r-\rho_{\infty}y_\xi)
  \end{bmatrix}.
\end{equation}
In \eqref{eq:RQsysdeflag}, we use the same notation for the variable $y_\xi$ and
the operator for pointwise multiplication by $y_\xi$. We will use this convention
for the rest of the paper. The equivalence between \eqref{eq:defRQsys} and
\eqref{eq:RQsysdeflag} holds only assuming the that $y_\xi\geq0$ and all the
functions are smooth enough to do the manipulation.

Note that we need to decompose the variables $y$ and $r$ in \eqref{eq:old2CH} to give them a decay
which enables us to define them in a proper functional setting. We define
$\zeta$ and $\bar r$ as
\begin{equation*}
y(t,\xi) = \zeta(t,\xi) +
\xi\quad\text{ and }\quad r(t,\xi) = \bar{r}(t,\xi) + \rho_{\infty}y_{\xi}(t,\xi).
\end{equation*}
The Banach space which contains $\zeta$ and $H$ is the subspace of bounded and
continuous functions with derivative in $\Ltwo$,
\begin{equation}\label{eq:spaceV}
\V \coloneqq \{f \in \mathbf{C}_{\text{b}}(\R) \:|\: f_{\xi} \in \Ltwo(\R) \},
\end{equation}
endowed with the norm $\|f\|_{\V} \coloneqq \|f\|_{\Linfty} + \|f_{\xi}\|_{\Ltwo}$.
Then we let
\begin{equation}\label{eq:spaceE}
\E \coloneqq \V \times \Hone \times \V \times \Ltwo
\end{equation}
be a Banach space tailored for the tuple $X = (\zeta, U, H, \bar{r})$ with norm
\begin{equation}
\|X\|_{\E} \coloneqq \|\zeta\|_\V + \|U\|_{\Hone} + \|H\|_{\V} + \|\bar{r}\|_{\Ltwo}.
\end{equation}
The unique solution of \eqref{eq:old2CH}, as studied in \cite{Grunert2012},
is then completely described by this tuple.

An alternative viewpoint of the equivalent Lagrangian system is the following.
Let us define the operators $\Gcal$ and $\Kcal$ as the fundamental solutions to
the operator in \eqref{eq:RQsysdeflag}, meaning that they satisfy
\begin{equation}
  \label{eq:GKid}
  \begin{bmatrix}
    -\partial_\xi & y_\xi\\
    y_\xi & -\partial_\xi
  \end{bmatrix}
  \circ
  \begin{bmatrix}
    \Kcal&\Gcal\\
    \Gcal&\Kcal
  \end{bmatrix}
  =
  \begin{bmatrix}
    \delta&0\\
    0&\delta
  \end{bmatrix}.
\end{equation}
As we mentioned in the introduction, the operators $\Kcal$ and $\Gcal$ can be
computed explicitly, using the fundamental solution of the Helmholtz operators in
Eulerian coordinates. If we define
\begin{subequations}
  \label{eq:contkernels}
  \begin{equation}
    \label{eq:kernelg}
    g(\eta,\xi) = \frac12e^{-\abs{y(\xi) - y(\eta)}}
  \end{equation}
  and
  \begin{equation}
    \label{eq:kernelk}
    \kappa(\eta,\xi) = -\frac12\sgn(\xi - \eta)e^{-\abs{y(\xi) - y(\eta)}},
  \end{equation}
\end{subequations}
then we can check that the operators defined as $\Gcal(f)=\int_\Real
g(\eta,\xi)f(\eta)\,d\eta$ and $\Kcal(f)=\int_\Real
\kappa(\eta,\xi)f(\eta)\,d\eta$ are solutions to \eqref{eq:GKid}, again assuming
$y$ is monotone increasing in $\xi$. This means that we can obtain explicit expressions for $R$ and $Q$
given by
\begin{subequations}
  \label{eq:explicitRQintform}
  \begin{align}
    \label{eq:explicitRQintform-R}  R &= \int_\Real \kappa(\eta,\xi) U(\eta)U_\xi(\eta)\,d\eta + \int_\Real g(\eta,\xi)( H_\xi(\eta) + \rho_\infty (r(\eta) -\rho_{\infty}y_\xi(\eta)))\,d\eta, \\
    \label{eq:explicitRQintform-Q}
    Q &= \int_\Real g(\eta,\xi) U(\eta)U_\xi(\eta)\,d\eta + \int_\Real \kappa(\eta,\xi)( H_\xi(\eta) + \rho_\infty (r(\eta) -\rho_{\infty}y_\xi(\eta)))\,d\eta.
  \end{align}
\end{subequations}
In \cite{holden2007global, Grunert2012}, the authors prove that the right-hand side
of their respective versions of \eqref{eq:old2CH} is locally Lipschitz, and
consecutive contraction arguments yield the existence of a unique short-time
solution. In the same manner, we would like to prove that there exists a unique
short-time solution for our semi-discrete system, but the explicit forms for $g$
and $\kappa$ in \eqref{eq:contkernels} are not available in the discrete setting. As
a remedy, we propagate the kernel operators corresponding to $\Kcal$ and $\Gcal$
by incorporating them in the governing equations. Given the evolution of $y$,
that is, $y_t = U$, we can derive evolution equations for
$\Gcal$ and $\Kcal$. Let us see how this can be done in the continuous case
before dealing with the discrete case. Formally we have
\begin{align*}
  \fracpart\Gcal(f) &= \frac12\int_\Real \fracpart e^{-\abs{y(t,\xi) - y(t,\eta)}}f(\eta)\,d\eta\\
                    &= -\frac12\int_\Real \sgn(y(t,\xi) - y(t,\eta))(y_t(t,\xi) - y_t(t,\eta)) e^{-\abs{y(t,\xi) - y(t,\eta)}}f(\eta)\,d\eta\\
                    &= -\frac12\int_\Real \sgn(y(t,\xi) - y(t,\eta))(U(t,\xi) - U(t,\eta)) e^{-\abs{y(t,\xi) - y(t,\eta)}}f(\eta)\,d\eta.
\end{align*}
Here we assume again that we know \textit{a priori} that $y$ remains a monotone
function with respect to $\xi$. Then, we can rewrite the last equality as
\begin{equation}
  \label{eq:dertGcal}
  \fracpart\Gcal(f) = -\frac12\int_\Real \sgn(\xi - \eta)(U(t,\xi) - U(t,\eta)) e^{-\abs{y(t,\xi) - y(t,\eta)}}f(\eta)\,d\eta.
\end{equation}
For a function $U$, we can associate a pointwise multiplication operator which we denote by $\Ucal$. That is, we may write
$\Ucal(f)(\xi) = U(\xi)f(\xi)$ for any function $f$ and any point $\xi$. The integral kernel of $\Ucal$ would be singular and equal to
$U(\xi)\delta(\xi - \eta)$. Using this notation, we can rewrite \eqref{eq:dertGcal} as
\begin{equation*}
  \fracpart\Gcal(f) = (\Ucal \circ \Kcal)(f) - (\Kcal \circ \Ucal)(f).
\end{equation*}
This can equivalently be stated as
\begin{equation}
  \label{eq:GKcal_evol}
  \fracpart\Gcal = [\Ucal, \Kcal], \qquad \fracpart\Kcal = [\Ucal, \Gcal],
\end{equation}
where the evolution equation for $\Kcal$ is derived analogously. An equivalent system of equations for the 2CH system is then
given by
\begin{subequations}
  \label{eq:newcontequivsys}
  \begin{equation}
    y_{t} = U,\quad U_{t} = -Q,\quad H_{t} = -UR, \quad r_t = 0,
  \end{equation}
  \label{eq:newcontequivsys2}
  \begin{equation}
    \fracpart\Gcal = [\Ucal, \Kcal],\quad \fracpart\Kcal = [\Ucal, \Gcal],
  \end{equation}
\end{subequations}
with $R$ and $Q$ given as
\begin{equation}
  \label{eq:newdefQR}
  \begin{bmatrix}
    R\\Q
  \end{bmatrix}
  =
  \begin{bmatrix}
    \Kcal&\Gcal\\
    \Gcal&\Kcal
  \end{bmatrix}
  \circ
  \begin{bmatrix}
    UU_\xi\\
    H_\xi + \rho_\infty (r - \rho_{\infty}y_\xi)
  \end{bmatrix}
\end{equation}
For all initial conditions we will consider, the new system \eqref{eq:newcontequivsys} and \eqref{eq:newdefQR} gives rise to the
same solutions as the one given by \eqref{eq:old2CH}, \eqref{eq:contkernels} and \eqref{eq:explicitRQintform}. It can be shown
that the evolution equations \eqref{eq:newcontequivsys2} for $\Gcal$ and $\Kcal$ can be obtained directly from the product
identity \eqref{eq:GKid} by differentiating it and using integration by parts.

\subsection{Reformulation of the semi-discrete system}
Turning back to the formal expression \eqref{eq:semidisc_sys_formal}, we use the
the Green's functions from Theorem \ref{thm:green_g} and Corollary
\ref{cor:green_k} to write out the right-hand side explicitly. Considering
\eqref{eq:fundrelggkk} where we now have $a_j = \D_+y_j$, we observe that they
correspond to the discrete versions of \eqref{eq:GKid}. Indeed, we have the
following identity
\begin{equation}
  \label{eq:discggkkid}
  \begin{bmatrix}
    -\D_{j-} & (\D_{+}y_j) \\
    (\D_{+}y_j) & -\D_{j+}
  \end{bmatrix}
  \circ
  \begin{bmatrix}
    \gamma_{i,j}&  k_{i,j}\\
    g_{i,j}  & \kappa_{i,j}
  \end{bmatrix}
  = \frac{1}{\dxi}
  \begin{bmatrix}
    \delta_{i,j}&0\\
    0&\delta_{i,j}
  \end{bmatrix}
\end{equation}
which has to be compared with \eqref{eq:GKid} in the continuous case. Thus,
the second equation in \eqref{eq:semidisc_sys_formal} can be rewritten as
\begin{equation}
  \label{eq:defdotuj}
  \dot{U}_j = -\sumZxi{i}{g_{i,j} \left(U_i (\D_+U_i) + \D_-\left(\frac{h_i}{\D_+y_i} + \rho_{\infty}\frac{\bar{r}_i}{\D_+y_i}\right)\right)},
\end{equation}
where we have defined
\begin{equation}
  \bar{r}_i \coloneqq \rho_{0,i} - \rho_{\infty} (\D_+y_i)
  \label{eq:rbar}
\end{equation}
and
\begin{equation}
  \label{eq:h}
  h_i \coloneqq \frac12 (U_i)^2(\D_+y_i) + \frac12 \frac{(\D_+U_i)^2}{\D_+y_i} + \frac{1}{2} \frac{\bar{r}_i^2}{\D_+y_i}.
\end{equation}
From the expressions in \eqref{eq:defdotuj} and \eqref{eq:h}, it seems that, if $\D_+y_i$ goes to zero for some index $i$ and time
$t$, then $\dot U_j$ and $h_i$ blow up. However, it turns out that these quantities remain bounded, which allows us to extend the solution
globally in time. To obtain a well-defined system, we are going to remove the explicit dependence on $1/\D_+y_i$ by adding
$h$ to the set of variables of the system.

With the discrete kernels $g$, $k$, $\gamma$, and $\kappa$ from Section \ref{sec:aux} we are able to express $\A{\D_+y}^{-1}$ in \eqref{eq:semidisc_sys_formal} to
obtain \eqref{eq:defdotuj}. However, since we do not know their explicit form as functions of $\D_+y_j$,
we derive a system
analogous to \eqref{eq:newcontequivsys} by introducing $g$, $k$, $\gamma$, and $\kappa$ as variables. To compute the evolution of
$g$, $k$, $\gamma$, and $\kappa$, we repeat the procedure from the continuous case. By differentiating \eqref{eq:discggkkid} and
using the fact that $\dot y_i=U_i$, we get
\begin{equation*}
  \begin{bmatrix}
    \dot\gamma & \dot k\\
    \dot g & \dot\kappa
  \end{bmatrix}
  =
  -\begin{bmatrix}
    \gamma &  k\\
    g & \kappa
  \end{bmatrix}
  \ast
  \begin{bmatrix}
    0 & \D_{+}U \\
    \D_{+}U & 0
  \end{bmatrix}
  \begin{bmatrix}
    \gamma &  k\\
    g & \kappa
  \end{bmatrix}
\end{equation*}
which in explicit form yields
\begin{align}
  \begin{aligned}
    \dot g_{i,j} &= -\kappa_{m,j}\ast((\D_+U_m) \gamma_{i,m}) - g_{m,j}\ast((\D_+U_m) g_{i,m}),\\
    \dot \gamma_{i,j} &= -k_{m,j}\ast((\D_+U_m)\gamma_{i,m}) - \gamma_{m,j}\ast((\D_+U_m) g_{i,m}),
  \end{aligned}
                        \label{eq:ggamma_dot}
\end{align}
and
\begin{align}
  \label{eq:kkappa_dot}
  \begin{aligned}
    \dot k_{i,j} &= -k_{m,j}\ast((\D_+U_m) k_{i,m}) - \gamma_{m,j}\ast((\D_+U_m) \kappa_{i,m}),\\
    \dot \kappa_{i,j} &= -\kappa_{m,j}\ast((\D_+U_m) k_{i,m}) - g_{m,j}\ast((\D_+U_m) \kappa_{i,m}).
  \end{aligned}
\end{align}
Here we denote by $(g \ast f)_j$ the action of the operator $g_{i,j}$ as a summation kernel on a sequence $f_i$, defined as
\begin{equation*}
  (g \ast f)_j = \sumZxi{i}{ g_{i,j}f_i }.
\end{equation*}
Moreover, we introduce the following norms for the operators,
\begin{align}
\label{eq:fund_norms}
\begin{aligned}
\norml{p}{g} &= \sup_{i}\norml{p}{g_i} = \sup_{i}\left(\sumZxi{j}{|g_{i,j}|^p}\right)^{\frac1p}, \\
\normlinfty{g} &= \sup_{i}\left(\sup_{j}|g_{i,j}|\right).
\end{aligned}
\end{align}
We establish in the next lemma some important properties for the fundamental solutions.

\begin{lem}[Preservation of identities]\label{lem:fundsolprops}
  Let $T > 0$, and assume that, for $t \in [0,T]$, $(\D_+y_j(t))_t = \D_+U_j(t)$
  for $j \in \Z$, and that $g, k, \gamma, \kappa$ and $\D_+U$ are bounded in
  $\ltwo$-norm in $[0,T]$.  Then, for $t \in [0,T]$ the sequences
  $g_{i,j}(t)$, $k_{i,j}(t)$, $\gamma_{i,j}(t)$, $\kappa_{i,j}(t)$ satisfy the
  following identities:
  \begin{enumerate}[(i)]
  \item The Green's function identities \eqref{eq:discggkkid},
  \item The symmetry identities
    \begin{equation}
      \label{eq:symopgk}
      g_{j,i} = g_{i,j}\quad\text{ and }\quad k_{j,i} = k_{i,j},
    \end{equation}
    and the antisymmetry identity
    \begin{equation}
      \label{eq:relgk}
      \gamma_{j,i} = -\kappa_{i,j}.
    \end{equation}
  \end{enumerate}
\end{lem}

\begin{proof}[Proof of Lemma \ref{lem:fundsolprops}]
  Recall from Remark \ref{rem:greens} that these identities are satisfied for
  $t=0$ by construction.  The rest of the proof then relies on Gr\"{o}nwall's
  inequality.  \textit{(i)}: We introduce the four operators $z_{l}$ for
  $l=1,2,3,4$ defined as
  \begin{align*}
    z_{1,i,j}  &=  (\D_+y_i) g_{i,j} - \D_{j-}\gamma_{i,j} - \frac{\delta_{i,j}}{\dxi},&
    z_{2,i,j}  &=  (\D_+y_j) k_{i,j} - \D_{j+}\kappa_{i,j} - \frac{\delta_{i,j}}{\dxi},\\
    z_{3,i,j}  &=  (\D_+y_j) \gamma_{i,j} - \D_{j+}g_{i,j},&
    z_{4,i,j}  &=  (\D_+y_j) \kappa_{i,j} - \D_{j-}k_{i,j}.
  \end{align*}
  Using $(\D_+y_j(t))_t = \D_+U_j(t)$ and \eqref{eq:ggamma_dot} we find that
  \begin{align*}
    (z_{1,i,j})_t &= (\D_+y_j)_t g_{i,j} + (\D_+y_j) \dot{g}_{i,j} - \D_{j-}\dot{\gamma}_{i,j} \\
                  &= (\D_+U_j) g_{i,j} - (\D_+y_j)  \sumZxi{m}{(\D_+U_m) \left( g_{i,m}g_{m,j} + \gamma_{i,m}\kappa_{m,j} \right)} \\
                  &\quad+ \D_{j-}\sumZxi{m}{(\D_+U_m) \left( g_{i,m}\gamma_{m,j} + \gamma_{i,m} k_{m,j} \right)} \\
                  &= (\D_+U_j) g_{i,j} - \sumZxi{m}{(\D_+U_m) g_{i,m}\left( (\D_+y_j) g_{m,j} - \D_{j-}\gamma_{m,j} \right)} \\
                  &\quad- \sumZxi{m}{(\D_+U_m) \gamma_{i,m}\left( (\D_+y_j) \kappa_{m,j} - \D_{j-}k_{m,j} \right)} \\
                  &= -\sumZxi{m}{(\D_+U_m) (g_{i,m}z_{1,m,j} +  \gamma_{i,m}z_{4,m,j})}.
  \end{align*}
  Similarly, one shows that
  \begin{align*}
    (z_{2,i,j})_t &= -\sumZxi{m}{(\D_+U_m) (k_{i,m}z_{2,m,j} + \kappa_{i,m}z_{3,m,j})}, \\
    (z_{3,i,j})_t &= -\sumZxi{m}{(\D_+U_m) (g_{i,m}z_{3,m,j} + \gamma_{i,m}z_{2,m,j})}
  \end{align*}
  and
  \begin{equation*}
    (z_{4,i,j})_t = -\sumZxi{m}{(\D_+U_m) (k_{i,m}z_{4,m,j} + \kappa_{i,m}z_{1,m.j})}.
  \end{equation*}
  Integrating the first of these, taking absolute values, applying H\"{o}lder's inequality and taking supremum over $i$ we obtain
  \begin{multline*}
    \sup_{i}\left| z_{1,i,j}(t) \right| \le \sup_{i}\left| z_{1,i,j}(0) \right| \\
    + \int_{0}^{t}\left(\normltwo{\D_+U(s)} \left(\normltwo{g(s)}\sup_{m}\left| z_{1,m,j}(s)\right| + \normltwo{\gamma(s)}\sup_{m}\left| z_{4,m,j}(s)\right|\right)\right)ds 
  \end{multline*}
  Treating the three other relations similarly and defining
  \begin{align*}
    Z(t) = \sum_{l=1}^4 \normlinfty{z_l(t)},
  \end{align*}
  we may add the four inequalities to obtain an inequality of the form
  \begin{equation*}
    Z(t) \le Z(0) + \int_{0}^{t} C(s) Z(s)\,ds,
  \end{equation*}
  where
  \begin{equation*}
    C(s) = 2\normltwo{\D_+U}\left(\normltwo{g} + \normltwo{k} + \normltwo{\gamma} + \normltwo{\kappa}\right)(s)
  \end{equation*}
  is bounded by assumption. Since $Z(0) = 0$, Gr\"{o}nwall's inequality yields
  $Z(t) = 0$ for $t \in [0,T]$, which proves the result.
  
  \textit{(ii):} We prove the symmetry of $g$. From \eqref{eq:discggkkid} we
  have $(\D_+y_m) g_{i,m} - \D_{m-}\gamma_{i,m}=\frac1{\Delta\xi}\delta_{i,m}$,
  such that summation by parts shows
  \begin{align*}
    g_{j,i} &= \sumZxi{m}{ \left[ (\D_+y_m)g_{i,m} - \D_{m-}\gamma_{i,m} \right] g_{j,m} } \\
            &= \sumZxi{m}{ \left[ (\D_+y_m) g_{i,m} g_{j,m} + \gamma_{i,m} \D_{m+}g_{j,m} \right] }.
  \end{align*}
  Then, we use the identity $\D_{m+}g_{j,m} = (\D_{+}y_m) \gamma_{j,m}$ from \eqref{eq:discggkkid} twice,
  first for $j$ and then for $i$, to obtain
  \begin{align*}
    g_{j,i} &= \sumZxi{m}{ \left[ (\D_+y_m) g_{i,m} g_{j,m} + \gamma_{i,m} (\D_+y_m) \gamma_{j,m} \right] } \\
            &= \sumZxi{m}{ \left[ (\D_+y_m) g_{i,m} g_{j,m} + (\D_{m+}g_{i,m}) \gamma_{j,m} \right] }.
  \end{align*}
  After summation by parts and using \eqref{eq:discggkkid} once more, we end up with
  \begin{equation*}            
    g_{j,i} = \sumZxi{m}{ g_{i,m} \left[ (\D_+y_m) g_{j,m} + \D_{m-}\gamma_{j,m} \right] } = g_{i,j},
  \end{equation*}
  and the symmetry of $g$ is proved. A similar procedure shows the symmetry of
  $k_{i,j}$. For the antisymmetry we also use \eqref{eq:discggkkid} and summation by parts to compute
  \begin{align*}
    \gamma_{j,i} &= \sumZxi{m}{ \left[ (\D_+y_m)k_{i,m} - \D_{m+}\kappa_{i,m} \right] \gamma_{j,m} } \\
                 &= \sumZxi{m}{ \left[ k_{i,m} \D_{m+}g_{j,m} + \kappa_{i,m} \D_{m-}\gamma_{j,m} \right] } \\
                 &= -\sumZxi{m}{ \left[ (\D_{m-}k_{i,m}) g_{j,m} - \kappa_{i,m} \D_{m-}\gamma_{j,m} \right] } \\
                 &= -\sumZxi{m}{ \kappa_{i,m} \left[ (\D_+y_m) g_{j,m} - \D_{m-}\gamma_{j,m} \right] } \\
                 &= -\kappa_{i,j}.
  \end{align*}
\end{proof}
Returning to \eqref{eq:defdotuj}, the second term in the right-hand side of the governing equations can be simplified as follows,
\begin{align*}
  \begin{aligned}
    -\sumZxi{i}{g_{i,j} \D_-\left(\frac{h_i}{\D_+y_i} + \rho_{\infty}\frac{\bar{r}_i}{\D_+y_i}\right)}& = \sumZxi{i}{\frac{\D_{i+}g_{j,i}}{\D_+y_i} \left(h_i + \rho_{\infty}\bar{r}_i\right)} \\
    &= \sumZxi{i}{\gamma_{j,i} \left(h_i + \rho_{\infty}\bar{r}_i\right)} \\
    &= -\sumZxi{i}{ \kappa_{i,j} \left(h_i + \rho_{\infty}\bar{r}_i\right)},
  \end{aligned}
\end{align*}
where we have used \eqref{eq:discggkkid} and \eqref{eq:relgk}. We define
\begin{equation}
  \label{eq:Q}
  Q_j \coloneqq \sumZxi{i}{g_{i,j} U_i (\D_+U_i)} + \sumZxi{i}{\kappa_{i,j} \left( h_i + \rho_\infty\bar{r}_i \right)}.
\end{equation}
Then, the evolution of $U$ is given by
\begin{equation}\label{eq:evol_U}
  \dot{U}_j = -Q_j
\end{equation}
The form of $Q$ in \eqref{eq:Q} motivates the definition
\begin{equation}
  \label{eq:R}
  R_j \coloneqq \sumZxi{i}{\gamma_{i,j} U_i (\D_+U_i)} + \sumZxi{i}{k_{i,j} \left(h_i + \rho_\infty\bar{r}_i\right)}.
\end{equation}
Indeed, with these definitions we have
\begin{equation*}
  \begin{bmatrix}
    R\\Q
  \end{bmatrix}
  =
  \begin{bmatrix}
    \gamma & k \\
    g&\kappa
  \end{bmatrix} \ast
  \begin{bmatrix}
    U (\D_+U)\\
    h +  \rho_{\infty} \bar{r} 
  \end{bmatrix},
\end{equation*}
meaning $R$ and $Q$ satisfies
\begin{equation}
  \label{eq:discRQrel}
  \begin{bmatrix}
    -\D_{-} & (\D_{+}y_j) \\
    (\D_{+}y_j) & -\D_{+}
  \end{bmatrix}
  \circ
  \begin{bmatrix}
    R_j \\ Q_j
  \end{bmatrix}
  = \begin{bmatrix}
    U_j (\D_+U_j)\\
    h_j +  \rho_{\infty} \bar{r}_j 
  \end{bmatrix}.
\end{equation}
We recognize this as the discrete version of \eqref{eq:RQsysdeflag}.

The relation $\dot{U}_j = -Q_j$ shows that we have a differential equation for
$U$ in the variables $y$, $U$, $H$, $\bar{r}$, $g$, and $\kappa$.  From
\eqref{eq:rbar} we obtain
\begin{equation}\label{eq:evol_rb}
  \dot{\bar{r}}_j = \dot{r}_j - \rho_\infty \D_+\dot{y}_j = - \rho_\infty\D_+ U_j.
\end{equation}
Next, we introduce the cumulative energy $H_j$ as
\begin{equation}
  \label{eq:H}
  H_j = \dxi \sum_{i = -\infty}^{j-1} h_i,
\end{equation}
so that $h_j = \D_+H_j$. To obtain the evolution equation of $H$, we first multiply \eqref{eq:h} by $\D_+y_i$ and differentiate
the result with respect to time to obtain
\begin{align*}
  \diff{t}{} \left((\D_+y_i) h_i \right) 
  = -U_i Q_i (\D_+y_i)^2 + U_i^2 (\D_+y_i) (\D_+U_i) - (\D_+U_i)(\D_+Q_i) - \rho_{\infty}\bar{r}_i \D_+U_i,
\end{align*}
after using \eqref{eq:evol_U} and \eqref{eq:evol_rb}. Then, we use the relation between $Q$ and $R$ given in \eqref{eq:discRQrel}
to obtain
\begin{align*}  
  \diff{t}{} \left((\D_+y_i) h_i \right) = (\D_+U_i)h_i -(\D_+y_i)[U_i (\D_-R_i) + R_i(\D_+U_i)]
\end{align*}
Simplifying further, we obtain
\begin{equation*}
  \dot{h}_i  = -\left[ U_i (\D_-R_i) + R_i (\D_+U_i)\right].
\end{equation*}
This leads to
\begin{equation}\label{eq:evol_H}
  \dot{H}_j = -\dxi \sum_{i = -\infty}^{j-1} \left[ U_i (\D_-R_i) + R_i (\D_+U_i)\right] = -U_j R_{j-1}, 
\end{equation}
where in the last equality we have used the decay at infinity together with \eqref{eq:sum_by_parts}.

Collecting all the equations and applying the relations \eqref{eq:symopgk} and
\eqref{eq:relgk} we obtain the closed system
\begin{subequations}
  \label{eq:ODE_disc}	
  \begin{align}
    \dot{\zeta}_j &= U_j, \label{eq:ODE_disc_z}	 \\
    \dot{U}_j &= -Q_j \label{eq:ODE_disc_U} \\
    \dot{H}_j &= -U_j R_{j-1}, \label{eq:ODE_disc_H} \\
    \dot{\bar{r}}_j &= - \rho_\infty \D_+U_j, \label{eq:ODE_disc_r} \\
    \dot{g}_{i,j} &= -\sumZxi{m}{(\D_+U_m) \left( g_{i,m}g_{m,j} + \gamma_{i,m} \kappa_{m,j} \right)}, \label{eq:ODE_disc_g} \\
    \dot{k}_{i,j} &= -\sumZxi{m}{(\D_+U_m) \left( k_{i,m}k_{m,j} + \kappa_{i,m} \gamma_{m,j} \right)}, \label{eq:ODE_disc_k} \\
    \dot{\gamma}_{i,j} &= -\sumZxi{m}{(\D_+U_m) \left( \gamma_{i,m} k_{m,j} + g_{i,m} \gamma_{m,j} \right)}, \label{eq:ODE_disc_gm} \\
    \dot{\kappa}_{i,j} &= -\sumZxi{m}{(\D_+U_m) \left( \kappa_{i,m} g_{m,j} + k_{i,m} \kappa_{m,j} \right)}, \label{eq:ODE_disc_kp}
  \end{align}
\end{subequations}
where $y_j = j\dxi + \zeta_j$, and we recall
\begin{align*}
  R_j &= \sumZxi{i}{\gamma_{i,j} U_i (\D_+U_i)} + \sumZxi{i}{k_{i,j} \left(h_i + \rho_\infty\bar{r}_i\right)}, \\
  Q_j &= \sumZxi{i}{g_{i,j} U_i (\D_+U_i)} + \sumZxi{i}{\kappa_{i,j} \left(h_i + \rho_\infty\bar{r}_i\right)}.
\end{align*}

\section{Existence and uniqueness of the solution to the semi-discrete 2CH system}
\label{sec:existence}

In this section, we show that the semi-discrete system \eqref{eq:ODE_disc} has a unique, globally defined solution. Let us first
introduce the functional setting for the analysis. We define the discrete analogue of the $\Hone(\R)$-inner product,
\begin{equation}
  \ip{a}{b}_{\hone} \coloneqq \sumZxi{j}{\left[ a_{j} b_{j} + (\D_+a_{j}) (\D_+b_{j}) \right]}, \label{eq:iph1}
\end{equation}
which induces a norm in the usual manner. The discrete Sobolev-type inequality
\begin{equation}\label{eq:sobolev_l}
  \|a\|_{\linfty} \le \frac{1}{\sqrt{2}} \|a\|_{\hone}
\end{equation}
can be proven in a very similar way as its continuous version, see, e.g., \cite{Brezis}. We introduce the subspace $\Vd$ of
$\linfty$ defined as
\begin{equation}
  \Vd \coloneqq \{a \in \linfty \:|\: \D_+a \in \ltwo \}, \qquad  \norm{a}_{\Vd} \coloneqq \normlinfty{a} + \normltwo{\D_+a}.
  \label{eq:Vd}
\end{equation}
We define the discrete version of the space used in the continuous setting, namely
\begin{equation}\label{eq:Ed}
  \Ed \coloneqq \Vd \times \hone \times \Vd \times \ltwo,
\end{equation}
with norm
\begin{equation*}
  \norm{(\zeta,U,H,\bar{r})}_{\Ed} \coloneqq \normVd{\zeta} + \normhone{U} +\normVd{H} + \normltwo{\bar{r}}.
\end{equation*}
Since we have included the operator kernels as solution variables in \eqref{eq:ODE_disc},
we have to introduce a space for them as well. 
To account for that the kernels are well-behaved, we choose their space to be $\lstar \coloneqq \lone \cap \linfty$
with norm $\normlstar{\cdot} = \normlone{\cdot} + \normlinfty{\cdot}$,
with $\lbf^p$-norms defined in \eqref{eq:fund_norms}.
We note that $\lstar \subset \ltwo$, since we have the inequality
\begin{equation}\label{eq:l2l1linf}
\normltwo{g} \le \normlinfty{g}^{1/2} \normlone{g}^{1/2} \le \frac12 (\normlinfty{g} + \normlone{g}).
\end{equation}
Thus, we will consider solution tuples of the form
\begin{equation*}
X = (\zeta, U, H, \bar{r}, g, k, \gamma, \kappa) \in \Ed \times (\lstar)^4 \eqqcolon \Ek,
\end{equation*}
where $\Ek$ denotes the space $\Ed$ augmented with the space for the kernel operators $\lstar$.
Moreover, for the kernel operator $g$ we have that the transpose $g^\tp$ of $g$ is given by
$(g^\tp)_{i,j} = g_{j,i}$. Then, the following result, reminiscent of Young's convolution inequality, will prove useful.
\begin{prop}[Young's inequality for general operators]
  \label{prp:young}
  \begin{equation}
    \label{eq:altYoung}
    \norml{r}{g\ast f} \le  \norml{q}{g}^{\frac{q}{r}} \norml{q}{g^\tp}^{1-\frac{q}{r}}\norml{p}{f},
  \end{equation}
  for
  \begin{equation*}
    1 + \frac{1}{r} = \frac{1}{p} + \frac{1}{q}, \qquad p, q, r \in [1,\infty].
  \end{equation*}
\end{prop}
Above, we use the convention $q/\infty = 0$ for $q < \infty$, and $\infty/\infty = 1$. Note that the standard Young's
inequality is usually given for a translation invariant kernel where $g$ takes the form $g_{i,j} = \hat g_{i - j}$ for some
sequence $\hat g$. For an operator of this form, we can check that $g^\tp=\tau\circ g\circ\tau$, where the operator $\tau$ inverts
the indexing, that is $\tau(f)_j = f_{-j}$. Since the operator $\tau$ is an isometry in all $\lsp{q}$-spaces, the expression
\eqref{eq:altYoung} simplifies to
\begin{equation*}
  \norml{r}{g\ast f} \le  \norml{q}{g}\norml{p}{f}.
\end{equation*}

\begin{proof}[Proof of Young's inequality]
  Below we will use the following discrete version of the generalized H\"{o}lder inequality,
  \begin{equation}\label{eq:genHolder}
    \norm{\prod_{k=1}^{n}a_k}_{\lbf^q} \le \prod_{k=1}^{n} \norm{a_k}_{\lbf^{p_k}} \text{ for } \sum_{k=1}^{n}\frac{1}{p_k} = \frac{1}{q}, \quad q, p_k \in [1,\infty],
  \end{equation}
  where the $j$-th component of a product of sequences is interpreted as $\left( \prod_{k=1}^{n}a_k \right)_j =
  \prod_{k=1}^{n}(a_k)_j$.  We note that the proof of \eqref{eq:genHolder} follows that of the continuous case, see, e.g.,
  \cite[Ex.~4.4]{Brezis}.  Let us denote $h = g\ast f$. Note that $r < \infty \implies p,q < \infty$, which shows that
  some configurations are impossible and can be excluded.  We deal with the three remaining cases:
  
  \textit{(i) $r < \infty$}: From the generalized H\"{o}lder inequality we obtain
  \begin{align*}
    |h_j| &\le \sumZxi{i}{ \left( |f_i|^{\frac{p}{r}} |g_{i,j}|^{\frac{q}{r}} \right) |f_i|^{1-\frac{p}{r}} |g_{i,j}|^{1-\frac{q}{r}} } \\
          &\le \left[ \sumZxi{i}{ \left( |f_i|^{\frac{p}{r}} |g_{i,j}|^{\frac{q}{r}} \right)^r } \right]^{\frac{1}{r}} \left[ \sumZxi{i}{\left(|f_i|^{1-\frac{p}{r}}\right)^{\frac{rp}{r-p}}} \right]^{\frac{r-p}{rp}} \\
          &\quad\times \left[ \sumZxi{i}{\left(|g_{i,j}|^{1-\frac{q}{r}}\right)^{\frac{rq}{r-q}}} \right]^{\frac{r-q}{rq}} \\
          &\le \left[ \sumZxi{i}{ |f_i|^{p} |g_{i,j}|^{q} } \right]^{\frac{1}{r}} \left[ \sumZxi{i}{|f_i|^{p}} \right]^{\frac{r-p}{rp}} \left[ \sup_{j \in \Z} \left( \sumZxi{i}{|g_{i,j}|^{q} } \right)^{\frac{1}{q}} \right]^{\frac{r-q}{r}}
  \end{align*}
  which implies
  \begin{align*}
    \sumZxi{j}{|h_j|^r} &\le \norm{f}_{\lbf^p}^{r-p} \left[ \sup_{j \in \Z} \left( \sumZxi{i}{|g_{i,j}|^{q} } \right)^{\frac{1}{q}} \right]^{r-q} \sumZxi{j}{ \sumZxi{i}{ |f_i|^{p} |g_{i,j}|^{q} } } \\
                        &\le \norm{f}_{\lbf^p}^{r-p} \left[ \sup_{j \in \Z} \left( \sumZxi{i}{|g_{i,j}|^{q} } \right)^{\frac{1}{q}} \right]^{r-q} \sumZxi{i}{|f_i|^p \sumZxi{j}{|g_{i,j}|^q} } \\
                        &\le \norm{f}_{\lbf^p}^{r} \left[ \sup_{j \in \Z} \left( \sumZxi{i}{|g_{i,j}|^{q} } \right)^{\frac{1}{q}} \right]^{r-q} \left[ \sup_{i \in \Z} \left( \sumZxi{j}{|g_{i,j}|^q} \right)^{\frac{1}{q}} \right]^q,
  \end{align*}
  where we have used Fubini's theorem in the second inequality.
  Taking $r$-th roots we obtain the result.
  
  \textit{(ii) $r = \infty$, $q < \infty$}: We find
  \begin{equation*}
    |h_j| \le \sumZxi{i}{|g_{i,j}| |f_i|} \le \norm{f}_{\lbf^p} \left(\sumZxi{i}{|g_{i,j}|^q}\right)^{\frac{1}{q}},
  \end{equation*}
  and taking supremum over $j$ this corresponds to \eqref{eq:altYoung} where $q/\infty = 0$.
  
  \textit{(iii) $r = q = \infty$}: We find
  \begin{equation*}
    |h_j| \le \sumZxi{i}{|g_{i,j}| |f_i|} \le \sumZxi{i}{ |f_i| \left(\sup_{j \in \Z} |g_{i,j}| \right) } \le \sup_{i \in \Z} \left(\sup_{j \in \Z} |g_{i,j}| \right) \sumZxi{i}{|f_i|},
  \end{equation*}
  and taking supremum over $j$ this corresponds to \eqref{eq:altYoung} where $\infty/\infty = 1$.
\end{proof}

To prove the short-time existence of \eqref{eq:ODE_disc}, we consider an auxiliary system which corresponds to \eqref{eq:ODE_disc},
except that we have decoupled $\zeta$, $U$ and $H$ from their discrete derivatives $\D_+\zeta$, $\D_+U$ and $\D_+H$ by introducing
the sequences $\alpha$, $\beta$ and $h$. The reason for this is that we cannot take for granted that the kernels satisfy
\eqref{eq:discggkkid} for $t > 0$, and then we cannot use \eqref{eq:discRQrel} when estimating the right-hand side of \eqref{eq:ODE_disc_U} in
$\hone$-norm.  Once the short-time existence of solutions to the auxiliary system is established, we will prove that
the coupling between $y$, $U$, $H$ and their discrete derivatives is indeed preserved if it holds initially. The auxiliary system
reads
\begin{subequations}
  \label{eq:aux}
  \begin{align}
     \label{eq:aux_z}
    \dot{\zeta}_j &= U_j, & \dot{U}_j &= -Q_j, &\dot{H}_j &= -U_j R_{j-1},\\ 
    \label{eq:aux_a}
    \dot{r}_j &= -\rho_\infty \beta_j,& 
                                        \dot{\alpha}_j &= \beta_j
  \end{align}
  \begin{align}
    \label{eq:aux_b}
    \dot{\beta}_j &= -R_j (1+\alpha_j) + h_j +  \rho_\infty r_j,\\ 
    \label{eq:aux_h}
    \dot{h}_j &= \left((U_j)^2 - R_j\right) \beta_j - U_j Q_j (1 + \alpha_j),
  \end{align}
  and
  \begin{align}
    \dot{g}_{i,j} &= -\sumZxi{m}{\beta_m \left( g_{i,m}g_{j,m} - \gamma_{i,m} \gamma_{j,m} \right)}, \label{eq:aux_g} \\
    \dot{k}_{i,j} &= -\sumZxi{m}{\beta_m \left( k_{i,m}k_{j,m} - \kappa_{i,m} \kappa_{j,m} \right)}, \label{eq:aux_k} \\
    \dot{\gamma}_{i,j} &= -\sumZxi{m}{\beta_m \left( \gamma_{i,m} k_{j,m} - g_{i,m} \kappa_{j,m} \right)}, \label{eq:aux_gm} \\
    \dot{\kappa}_{i,j} &= -\sumZxi{m}{\beta_m \left( \kappa_{i,m} g_{j,m} - k_{i,m} \gamma_{j,m} \right)}, \label{eq:aux_kp}
  \end{align}
\end{subequations}
where we have momentarily redefined $R$ and $Q$ as
\begin{equation*}
  \begin{bmatrix}
    R\\Q
  \end{bmatrix}
  =
  \begin{bmatrix}
    \gamma & k \\
    g&\kappa
  \end{bmatrix} \ast
  \begin{bmatrix}
    U\beta\\
    h +  \rho_{\infty} \bar{r} 
  \end{bmatrix}.
\end{equation*}
The evolution equations\eqref{eq:aux_b}, \eqref{eq:aux_h}, and the second equation of \eqref{eq:aux_a}
have been obtained formally by applying $\D_+$ to \eqref{eq:ODE_disc_z},
\eqref{eq:ODE_disc_U} and \eqref{eq:ODE_disc_H}, in combination with \eqref{eq:discRQrel}. We collect all
the variables in a tuple
\begin{equation*}
  Y = \left( \zeta, U, H, r, \alpha, \beta, h, g, k, \gamma, \kappa \right) \in \linfty \times \left( \ltwo \cap \linfty \right) \times \linfty \times (\ltwo)^4 \times (\lstar)^4 \eqqcolon \Ea
\end{equation*}
and introduce the corresponding norm
\begin{align*}
  \normEa{Y} &\coloneqq \normlinfty{\zeta} + \normltwo{U} + \normlinfty{U} + \normlinfty{H} + \normltwo{r} + \normltwo{\alpha} + \normltwo{\beta} + \normltwo{h} \\
             &\quad+ \normlstar{g} + \normlstar{k} + \normlstar{\gamma} + \normlstar{\kappa}.
\end{align*}

Note how we require $U \in \linfty$ to account for the fact that the decoupling of $U$ and $\D_+U$ deprives us of the continuous
inclusion $\hone \subset \linfty$.

\begin{lem}[Short-time solution for \eqref{eq:aux}]
  Let $Y_0 \in \Ea$ be such that $1 + \alpha_j \ge 0$ for all $j$, and with initial auxiliary variables $g_0, k_0, \gamma_0,
  \kappa_0$ constructed according to Theorem \ref{thm:green_g} and Corollary \ref{cor:green_k} with $a_j = 1 + \alpha_j$. Then,
  there exists a time $T>0$ depending only on $\normEa{Y_0}$ such that \eqref{eq:aux} has a unique solution $Y \in
  \mathbf{C}^1([0,T],\Ea)$ with initial data $Y_0$.
  \label{lem:auxLocal}
\end{lem}
\begin{proof}[Proof of Lemma \ref{lem:auxLocal}]
  We are going to use the symmetry and anti-symmetry identities
  \eqref{eq:symopgk} and \eqref{eq:relgk} in our estimates and we explain now
  why it can be done. First, we note that these identities hold initially by the
  construction of \eqref{eq:g} and \eqref{eq:k}. Then, from the evolution
  equations \eqref{eq:aux_g}--\eqref{eq:aux_kp} one can check that the symmetry
  identities are preserved by the Picard fixed-point operator which we will use here to
  prove the short-time existence of \eqref{eq:aux}. Then, by establishing local Lipschitz
  regularity of the right-hand side, we can prove
  the existence of a short-time solution in the closed subset of $\Ea$ where
  \eqref{eq:symopgk} and \eqref{eq:relgk} hold.
  
  Let us consider two functions in $\Ea$,
  \begin{equation*}
    Y = \left( \zeta, U, H, r, \alpha, \beta, h, g, k, \gamma, \kappa \right)\quad\text{ and }\quad
    \tilde{Y} = \left( \tilde{\zeta}, \tilde{U}, \tilde{H}, \tilde{r}, \tilde{\alpha}, \tilde{\beta}, \tilde{h}, \tilde{g}, \tilde{k}, \tilde{\gamma}, \tilde{\kappa} \right).
  \end{equation*}
  For the Lipschitz estimates, we first treat the right-hand sides of
  \eqref{eq:aux_g}--\eqref{eq:aux_kp}. We only provide details for
  \eqref{eq:aux_g} as  \eqref{eq:aux_k}--\eqref{eq:aux_kp} can be treated similarly.
  
  We start by considering the $\linfty$-norm using the following splitting,
  \begin{multline*}
     \left| -\sumZxi{m}{\beta_m \left( g_{j,m} g_{i,m} - \gamma_{i,m} \gamma_{j,m} \right) } \right. 
    + \left. \sumZxi{m}{\tilde{\beta}_m \left( \tilde{g}_{j,m} \tilde{g}_{i,m} - \tilde{\gamma}_{i,m} \tilde{\gamma}_{j,m} \right)} \right| \\
    \le \left| \sumZxi{m}{\beta_m g_{j,m} g_{i,m} } - \sumZxi{m}{\tilde{\beta}_m \tilde{g}_{j,m} \tilde{g}_{i,m} } \right| \\
    + \left| \sumZxi{m}{\beta_m \gamma_{i,m} \gamma_{j,m} } - \sumZxi{m}{\tilde{\beta}_m \tilde{\gamma}_{i,m} \tilde{\gamma}_{j,m} } \right|
  \end{multline*}
  We estimate the first term as follows
  \begin{multline*}
    \left| \sumZxi{m}{\beta_m g_{j,m} g_{i,m} } - \sumZxi{m}{\tilde{\beta}_m \tilde{g}_{j,m} \tilde{g}_{i,m} } \right| \\
    \leq \normlinfty{g} \normltwo{g} \normsltwo{\beta - \tilde{\beta}} + \normlinfty{g} \normsltwo{\tilde{\beta}} \normltwo{g -
      \tilde{g}} + \normsltwo{\tilde{\beta}} \normltwo{\tilde{g}} \normlinfty{g-\tilde{g}}
  \end{multline*}
  and the second term has a similar estimate.  For the $\lone$-norm, use the same splitting and consider again only the first
  term. We make use of the symmetry properties of the kernel operators, as given in Lemma \eqref{lem:fundsolprops}, to switch
  between indices and obtain
  \begin{multline*}
    \sumZxi{i}{ \left| \sumZxi{m}{\beta_m g_{j,m} g_{i,m} } - \sumZxi{m}{\tilde{\beta}_m \tilde{g}_{j,m} \tilde{g}_{i,m} } \right| } \\
    \le \normlone{g} \normltwo{g} \normsltwo{\beta - \tilde{\beta}} + \normlone{g} \normsltwo{\tilde{\beta}}
    \normltwo{g-\tilde{g}} + \normsltwo{\tilde{\beta}} \normltwo{\tilde{g}} \normlone{g-\tilde{g}},
  \end{multline*}
  From \eqref{eq:l2l1linf} we can then conclude
  that the right-hand side in \eqref{eq:aux_g} is locally Lipschitz-continuous with respect to the $\Ea$-norm.
  
  Let us consider Lipschitz properties of $R$ and $Q$. We decompose $Q$ in $Q_1 + Q_2 $ where
  \begin{subequations}
    \begin{align}
      \label{eq:Q1} 
      (Q_1)_j &\coloneqq \sumZxi{i}{g_{i,j} U_i (\D_+U_i)},\\
      \label{eq:Q2}
      (Q_2)_j &\coloneqq \sumZxi{i}{\kappa_{i,j} \left( h_i + \rho_\infty\bar{r}_i \right)}.
    \end{align}
  \end{subequations}
  Similarly, we decompose $R$ in $R_1 + R_2 $ where
  \begin{subequations}
    \begin{align}
      \label{eq:R1} 
      (R_1)_j &\coloneqq \sumZxi{i}{\gamma_{i,j} U_i (\D_+U_i)},\\
      \label{eq:R2}
      (R_2)_j &\coloneqq \sumZxi{i}{k_{i,j} \left(h_i + \rho_\infty\bar{r}_i\right)}.
    \end{align}
  \end{subequations}
  We have $Q_2 = \kappa\ast f$ for $f= h + \rho_{\infty} r$ so that
  \begin{equation*}
    \normsltwo{f} = \normsltwo{h +  \rho_{\infty} r} \le \normsltwo{h} + \rho_{\infty} \normsltwo{r}.
  \end{equation*}
  Starting with $Q_2$, we have
  \begin{equation*}
    \normsltwo{Q_2-\tilde{Q}_2} = \normsltwo{\kappa\ast f - \tilde\kappa\ast \tilde f} 
    \le \normsltwo{(\kappa - \tilde \kappa)\ast f} +  \normsltwo{\tilde\kappa\ast (f - \tilde f)}
  \end{equation*}
  For the first term above, applying the Young's inequality \eqref{eq:altYoung} with $r=p=2$ and $q=1$, we get
  \begin{equation*}
    \normsltwo{(\kappa - \tilde \kappa)\ast f} \leq \normslone{\kappa - \tilde \kappa}^{\tfrac12}\normslone{(\kappa - \tilde \kappa)^\tp}^{\tfrac12}\normsltwo{f}
  \end{equation*}
  Using the antisymmetry property \eqref{eq:relgk} of $\kappa$ and $\tilde\kappa$, namely $\kappa^\tp
  = -\gamma$ and $\tilde\kappa^\tp = -\tilde\gamma$, we get
  \begin{equation*}
    \normsltwo{(\kappa - \tilde \kappa)\ast f} \leq \normslone{\kappa - \tilde \kappa}^{\tfrac12}\normslone{\gamma - \tilde \gamma}^{\tfrac12}\normsltwo{f}
  \end{equation*}
  Hence, we obtain the following estimate in $\ltwo$-norm,
  \begin{equation*}
    \normsltwo{Q_2-\tilde{Q}_2} \leq \frac{\normslone{\gamma-\tilde{\gamma}} + \normslone{\kappa-\tilde{\kappa}}}{2} \normsltwo{f} + \frac{\normslone{\tilde{\gamma}} + \normslone{\tilde{\kappa}}}{2} \normsltwo{f - \tilde f}.
  \end{equation*}
  For the $\linfty$-norm, we use the same splitting
  \begin{equation*}
    \normslinfty{Q_2-\tilde{Q}_2} \leq \normslinfty{(\kappa - \tilde \kappa)\ast f} +  \normslinfty{\kappa\ast (f - \tilde f)}.
  \end{equation*}
  Applying \eqref{eq:altYoung} for $r=\infty$ and $p=q=2$,
  and the symmetry property of $\kappa$, we obtain in a similar way as before
  that
  \begin{equation*}
    \normslinfty{Q_2-\tilde{Q}_2} \leq \normsltwo{\gamma - \tilde{\gamma}}\normsltwo{f} +
    \normsltwo{\tilde{\gamma}}\normsltwo{f - \tilde f}.
  \end{equation*}
  In a similar fashion as for $Q_2$ we find
  \begin{align*}
    &\normsltwo{R_2-\tilde{R}_2} \le \normslone{ k - \tilde{k} } \normltwo{f} + \normslone{ \tilde{k} } \normsltwo{f - \tilde{f}}, \\
    &\normslinfty{R_2-\tilde{R}_2} \le \normsltwo{k-\tilde{k}} \normltwo{f} + \normsltwo{\tilde{k}}\normsltwo{f-\tilde{f}}.
  \end{align*}
  Furthermore, analogous applications of \eqref{eq:altYoung} and
  \eqref{eq:relgk} produce
  \begin{align*}
    &\normsltwo{Q_1 - \tilde{Q}_1} \le \normsltwo{g - \tilde{g}} \normlone{U \beta} + \normsltwo{\tilde{g}} \normslone{U\beta - \tilde{U}\tilde{\beta}}, \\
    &\normslinfty{Q_1 - \tilde{Q}_1} \le \normslinfty{g - \tilde{g}} \normlone{U \beta} + \normslinfty{\tilde{g}} \normslone{U\beta - \tilde{U}\tilde{\beta}}, \\
    &\normsltwo{R_1 - \tilde{R}_1} \le \normsltwo{\gamma - \tilde{\gamma}} \normlone{U \beta} + \normsltwo{\tilde{\gamma}} \normslone{U\beta - \tilde{U}\tilde{\beta}}, \\
    &\normslinfty{R_1 - \tilde{R}_1} \le \normslinfty{\gamma - \tilde{\gamma}} \normlone{U \beta} + \normslinfty{\tilde{\gamma}} \normslone{U\beta - \tilde{U}\tilde{\beta}}.
  \end{align*}
  For the $\lone$-norms above we then apply the Cauchy--Schwarz inequality to obtain
  \begin{equation*}
    \normlone{U \beta} \le \normltwo{U} \normltwo{\beta}, \quad \normslone{U\beta - \tilde{U}\tilde{\beta}} \le \normltwo{U} \normsltwo{\beta-\tilde{\beta}} + \normsltwo{\tilde{U}} \normsltwo{\beta - \tilde{\beta}},
  \end{equation*}
  which contain the relevant norms.
  
  From the preceding estimates on $Q_1$ and $Q_2$ the local Lipschitz property of the right-hand side of the second equation in
  \eqref{eq:aux_z} in the $\ltwo \cap \linfty$-norm is clear. Furthermore, since $U \in \linfty$, the previous $\linfty$-estimates
  on $R$ and $Q$ also show that the right-hand sides of \eqref{eq:aux_b} and \eqref{eq:aux_h} are locally Lipschitz in the
  $\ltwo$-norm.  For the last equation in \eqref{eq:aux_z}, we introduce the right-shift operator $(\uptau R)_j = R_{j-1}$ and we
  have
  \begin{align*}
    \normslinfty{U (\uptau R) - \tilde{U}(\uptau \tilde{R}) } &\le \normslinfty{U - \tilde{U}}\normlinfty{\uptau R} + \normslinfty{\tilde{U}}\normslinfty{\uptau( R-\tilde{R})} \\
                                                              &\le \normslinfty{U - \tilde{U}}\normlinfty{R} + \normslinfty{\tilde{U}}\normslinfty{R-\tilde{R}},
  \end{align*}
  The remaining right-hand sides of \eqref{eq:aux_z} and \eqref{eq:aux_a} are linear in the solution variables, and thus Lipschitz
  in their respective norms.  Hence, for \eqref{eq:aux} written as $\dot{Y} = \hat{F}(Y)$ we have
  \begin{equation*}
    \|\hat{F}(Y) - \hat{F}(\tilde{Y})\|_{\Ea} \le C(\|Y\|_{\Ea}, \|\tilde{Y}\|_{\Ea}) \|Y-\tilde{Y}\|_{\Ea},
  \end{equation*}
  which is what we set out to prove.
\end{proof}

The final step in obtaining short-time existence for \eqref{eq:ODE_disc} from the auxiliary system,
is to show that if the initial data for \eqref{eq:aux} satisfy
\begin{subequations}
  \begin{equation}
    \label{eq:presidggkk}
    \begin{bmatrix}
      -\D_{j-} & (1 + \alpha_j) \\
      (1 + \alpha_j) & -\D_{j+}
    \end{bmatrix}
    \circ
    \begin{bmatrix}
      \gamma_{i,j}&k_{i,j}\\
      g_{i,j}&\kappa_{i,j}\\
    \end{bmatrix}
    = \frac{1}{\dxi}
    \begin{bmatrix}
      \delta_{i,j}& 0\\
      0& \delta_{i,j}
    \end{bmatrix}
  \end{equation}
  \begin{equation}
    \label{eq:presdiscder}
    \alpha = \D_+\zeta, \quad \beta = \D_+U, \text{ and } \quad h = \D_+H,
  \end{equation}
\end{subequations}
then these identities are preserved in time by the solution. The result for \eqref{eq:presidggkk} has been proved in Lemma
\ref{lem:fundsolprops}, as it only depends on the identity $(\D_+y)_t = \D_+U$, which is replaced here by $\dot{\alpha} =
\beta$. Using \eqref{eq:presidggkk}, we infer from \eqref{eq:discRQrel} that
\begin{equation}
  \label{eq:dqalpha}
  \begin{bmatrix}
    -\D_{-} & (1 + \alpha_j) \\
    (1 + \alpha_j) & -\D_{+}
  \end{bmatrix} \circ
  \begin{bmatrix}
    R_j \\
    Q_j
  \end{bmatrix}
  =
  \begin{bmatrix}
    U_j \beta_j \\
    h_j +  \rho_{\infty} \bar{r}_j
  \end{bmatrix}.
\end{equation}
From the definition of \eqref{eq:aux} we get
\begin{subequations}
  \label{eq:smallauxode}
  \begin{equation}
    \frac{d}{dt}(\alpha_j - \D_+\zeta_j) = \beta_j - \D_+U_j,
  \end{equation}
  while the expression for $\D_+Q_j$ from \eqref{eq:dqalpha} yields
  \begin{equation}
    \frac{d}{dt}(\beta_j - \D_+U_j) = 0,
  \end{equation}
  and from the expression for $\D_-R_j$ we obtain
  \begin{equation}
    \diff{t}{}(h_j - \D_+H_j) = -R_j (\beta_j - \D_+U_j).
  \end{equation}
\end{subequations}
Hence, the equations \eqref{eq:smallauxode} give us that \eqref{eq:presdiscder}
holds for all time if it holds initially.
Then we have proved the following theorem.
\begin{thm}[Short-time solution for \eqref{eq:ODE_disc}]
  \label{thm:solLocal}
  Given $X_0 \in \Ek$ such that $1 + \D_+\zeta_j \ge 0$ and $g_0$, $k_0$, $\gamma_0$, and $\kappa_0$ are constructed according to
  Theorem \ref{thm:green_g} and Corollary \ref{cor:green_k} with $a_j = 1 + \D_+\zeta_j$. Then, there exists a time $T$ depending
  only on $\normEk{X_0}$ such that \eqref{eq:ODE_disc} has a unique solution $X \in C^1([0,T],\Ek)$ with initial datum $X_0$.
\end{thm}
The next step is to prove that there exists a subset, denoted by $\Bcal$, of $\Ek$ which is preserved by the evolution equation. For
this subset, the solution exists globally in time. The subset $\Bcal$ is defined as follows.
\begin{dfn}\label{def:setB}
  The set $\Bcal$ is composed of all $\left(\zeta, U, H, \bar{r}, g, k, \gamma, \kappa \right) \in \Ek$ such that
  \begin{enumerate}[(a)]
  \item $g, k, \gamma, \kappa$ satisfy the properties listed in Lemma \ref{lem:fundsolprops} for $a = \D_+y$, \label{setB:kern}
  \item $(\D_+y, \D_+U, \D_+H, \bar{r}) \in (\linfty)^4$, \label{setB:linf}
  \item $2(\D_+y_j)(\D_+H_j) = (U_j)^2(\D_+y_j)^2 + (\D_+U_j)^2 +  \bar{r}_j^2$ for all $j$, \label{setB:id}
  \item $\D_+y_{j} \ge 0$, $\D_+H_{j} \ge 0$, $\D_+y_{j} + \D_+H_{j} > 0$ for all $j$. \label{setB:pos}
  \end{enumerate}
\end{dfn}

\begin{lem}[Properties preserved by the flow]\label{lem:solInSet}
  Given initial datum $X_0 \in \Bcal$, let $X(t) \in C^1([0,T],\Ek)$ be the corresponding short-time solution given by Theorem
  \ref{thm:solLocal}. Then $X(t) \in \Bcal$ for all $t \in [0,T]$.
\end{lem}
\begin{proof}[Proof of Lemma \ref{lem:solInSet}]
  Property \eqref{setB:kern} follows from Lemma \ref{lem:fundsolprops}, since the solution variables in $X(t)$ satisfy
  $\D_+\dot{y}_j = \D_+U_j$ and $\D_+U \in \ltwo$, where we as usual have $\D_+y_j = 1 + \D_+\zeta_j$.  The proof of property
  \eqref{setB:linf} essentially follows \cite[Lem.~3.3]{Grunert2012}, and so we omit it.  The proof of \eqref{setB:id} is
  similar to the proof of \cite[Lem.~3.5]{Grunert2012}, while the proof of \eqref{setB:pos} is analogous to that of
  \cite[Lem.~2.7]{holden2007global}, and they are also omitted here.
\end{proof}
For the rest of the paper we will only consider $X \in \Bcal \cap \Ek$, as solutions in this set contains all the relevant
solutions to the original 2CH system \eqref{eq:2CH}. Lemma \ref{lem:solInSet}, and in particular the preservation of the identity
\begin{equation}
  \label{eq:Bid}
  2(\D_+y_j) h_j = U_j^2(\D_+y_j)^2 + (\D_{+}U_j)^2 + \bar{r}_j^2,
\end{equation}
allows us to prove useful estimates for the solutions in $\Bcal$.  We have
\begin{equation}
  \label{eq:UDUleH} \sumZxi{j}{\abs{U_j}\abs{\D_+U_j}} \le \Htotal(t),
\end{equation}
where $\Htotal(t) = \lim_{n\to+\infty}H_n$ is the total energy of the discrete system. This quantity corresponds to $\Hcaldis$ in
\eqref{eq:ham_dis}. Indeed, the Hamiltonian \eqref{eq:HamiltonianDis} is conserved for $t \in [0,T]$, that is
$\Htotal(t) = \Htotal(0) < \infty$ for $t \in [0,T]$.  We denote the preserved total energy $\Htotal(t)$ by $\Htotal$. Turning
back to the inequality \eqref{eq:UDUleH}, it can be proved as follows,
\begin{align*}
  \sumZxi{j}{\abs{U_j}\abs{\D_+U_j}} &\le \sumZxi{j}{\abs{U_j} \sqrt{(\D_+y_j)[2h_j - U_j^2 (\D_+y_j)]}} \\
                                     &\le \frac12 \sumZxi{j}{U_j^2 (\D_+y_j)} + \frac12 \sumZxi{j}{ [2h_j - U_j^2 (\D_+y_j)] } \\
                                     &= \Htotal,
\end{align*}
where in the first inequality we have used \eqref{eq:Bid}, and in the second inequality we have used $\D_+y_j \ge 0$ together with the Cauchy--Schwarz inequality.
An immediate consequence of \eqref{eq:UDUleH} is that $\normlinfty{U}$ can be uniformly bounded by a constant depending only on $\Htotal$.
To show this, we add and subtract in \eqref{eq:FD_prod} to find the identity
\begin{equation*}
  \D_{\pm}(U_i)^2 = 2 U_i (\D_{\pm}U_i) \pm \dxi (\D_\pm U_i)^2.
\end{equation*}
Taking advantage of the decay of $U$ at infinity, we may then write
\begin{equation*}
  (U_j)^2 = -2\dxi \sum_{i=j}^{\infty}U_i (\D_+U_i) - (\dxi)^2 \sum_{i=j}^{\infty}(\D_+U_i)^2 \le 2 \sumZxi{i}{\abs{U_i} \abs{\D_+U_i}} \le 2 \Htotal,
\end{equation*}
from which the bound
\begin{equation}
  \sup_{0 \le t \le T}\normlinfty{U(t)} \le \sqrt{2\Htotal}
  \label{eq:bnd_U}
\end{equation}
follows.  From \eqref{eq:bnd_U} and \eqref{eq:ODE_disc_z}, we then obtain the estimate
\begin{equation}
  \normlinfty{\zeta(t)} \le \normlinfty{\zeta(0)} + \sqrt{2\Htotal}t.
  \label{eq:bnd_zeta}
\end{equation}
Another useful estimate coming from \eqref{eq:Bid} is
\begin{equation}
  \label{eq:rb_hDy}
  |\bar{r}_j| \le \sqrt{2(\D_+y_j)h_j}.
\end{equation}
Now that Lemma \ref{lem:solInSet} has established $\D_+y_j(t) \ge 0$ in
the short-time solution for $t \in [0,T]$, we can apply Lemma \ref{lem:gk_sgn}
with $a_j = \D_+y_j$. Indeed, the sequences $g$, $\gamma$, $k$, and $\kappa$
solve \eqref{eq:discggkkid} and belong to $\lstar$ for $t \in [0,T]$, and so they
correspond to the unique decaying solution.  These properties contained in
Lemmas \ref{lem:gk_sgn} and \ref{lem:fundsolprops} are essential to establish
the \textit{a priori} estimates contained in the next lemma.

\begin{lem}[A priori relations and inequalities for the kernels]\label{lem:kern_ineq}
  As a consequence of establishing the preservation of the summation kernels and their sign properties over time, we have the identities
  \begin{subequations}
    \label{eq:sum_id1}
    \begin{align}
      \label{eq:sum_id1a}
      \sumZxi{j}{(\D_+y_j) |\gamma_{i,j}|} &= \sumZxi{j}{\abs{\D_{j+}g_{i,j}}} = 2 \normlinfty{g}, \\
      \label{eq:sum_id1b}
      \sumZxi{j}{(\D_+y_j) |\kappa_{i,j}|} &= \sumZxi{i}{\abs{\D_{j-}k_{i,j}}} = 2 \normlinfty{k},
    \end{align}
  \end{subequations}
  as well as
  \begin{subequations}\label{eq:sum_id2}
    \begin{align}
      \label{eq:sum_id2a}
      \sumZxi{j}{(\D_+y_j) g_{i,j}} &= \sumZxi{j}{(\A{\D_+y}g_i)_j} = 1, \\
      \label{eq:sum_id2b}
      \sumZxi{i}{(\D_+y_j) k_{i,j}} &= \sumZxi{j}{(\B{\D_+y}k_i)_j} = 1,
    \end{align}
  \end{subequations}
  and the bounds
  \begin{equation}\label{eq:fml_bnds}
    \normlinfty{g}, \normlinfty{k}, \normlinfty{\gamma}, \normlinfty{\kappa} \le 1,
  \end{equation}
  \begin{align}\label{eq:fml_l1_1}
    \begin{aligned}
      \normlone{g} &\le 1 + 2\normlinfty{\zeta}, & \normlone{k} &\le 1 + 2\normlinfty{\zeta}, \\
      \normlone{\gamma} &\le 2 \left[ 1 + \normlinfty{\zeta} \right], & \normlone{\kappa} &\le 2 \left[ 1 + \normlinfty{\zeta} \right].
    \end{aligned}
  \end{align}
\end{lem}
\begin{proof}[Proof of Lemma \ref{lem:kern_ineq}]
  To prove \eqref{eq:sum_id1a} we use $\D_+y_j \ge 0$ and \eqref{eq:discggkkid} for the leftmost equalities, while for the
  rightmost equalities we use the monotonicity properties of \eqref{eq:gk_mono} to write
  \begin{equation*}
    \sumZxi{j}{\abs{\D_{j+}g_{i,j}}} = \dxi \sum_{j=-\infty}^{i-1}\D_{j+}g_{i,j} - \dxi \sum_{j=i}^{\infty} \D_{j+}g_{i,j} = 2 g_{i,i} = 2 \normlinfty{g}.
  \end{equation*}
  We obtain \eqref{eq:sum_id1a} in the same way. To obtain \eqref{eq:sum_id2}, we use the definitions of the operators
  $\operatorname{A}$ in \eqref{eq:defAop} and $\operatorname{B}$ in \eqref{eq:defBop}, and apply telescopic cancellation to the
  differences $\D_{j+}\gamma_{i,j}$ and $\D_{j-}\kappa_{i,j}$ in the identities \eqref{eq:discggkkid}. In the same manner,
  telescopic cancellation applied to \eqref{eq:discggkkid} yields
  \begin{equation*}
    \gamma_{i,j} = \begin{cases}
      \dxi \sum_{m=-\infty}^{j} (\D_+y_m) g_{i,m}, & j \le i-1, \\
      -\dxi \sum_{m=j+1}^{\infty} (\D_+y_m) g_{i,m}, & j \ge i,
    \end{cases}
  \end{equation*}
  Using the fact that $\D_+y_j\geq0$ and $g_{i,j}\geq 0$, the triangle inequality and \eqref{eq:sum_id2} yield \eqref{eq:fml_bnds}
  for $\gamma$. We proceed similarly for $\kappa$. For $g$, observe that, using \eqref{eq:discggkkid}, we can rewrite them as
  \begin{equation}\label{eq:g_sum}
    \begin{aligned}
      g_{i,j} = \sumZ{m}{g_{i,m}\delta_{j,m}} &= \sumZxi{m}{g_{i,m} \left[ (\D_+y_m)g_{j,m} - \D_{m-}\gamma_{j,m} \right]} \\
      &= \sumZxi{m}{(\D_+y_m) \left[ g_{i,m} g_{j,m} + \gamma_{i,m} \gamma_{j,m} \right]}.
    \end{aligned}
  \end{equation}
  Using the decay at infinity we can then write
  \begin{align*}
    (g_{i,i})^2 &= \sum_{m=i}^{+\infty} \left[ (g_{i,m+1})^2 - (g_{i,m})^2 \right] = \dxi \sum_{m=i}^{+\infty} \left[ g_{i,m+1} + g_{i,m} \right] \D_{m+}g_{i,m} \\
                &= \dxi \sum_{m=i}^{+\infty} \left[ g_{i,m+1} + g_{i,m} \right] (\D_+y_m) |\gamma_{i,m}| \le 2 \dxi \sum_{m=i}^{+\infty} g_{i,m} (\D_+y_m) |\gamma_{i,m}| \\
                &\le \dxi \sum_{m=i}^{+\infty} (\D_+y_m) \left[ (g_{i,m})^2 + (\gamma_{i,m})^2 \right] \le \sumZxi{m}{ (\D_+y_m) \left[ (g_{i,m})^2 + (\gamma_{i,m})^2 \right] } \\
                &= g_{i,i},
  \end{align*}
  where we have used \eqref{eq:gk_mono} for the first inequality, and \eqref{eq:g_sum} for the final identity.  The bound
  $g_{i,i} \le 1$ follows, and we use
  $0 \le g_{i,j} \le g_{i,i}$ from \eqref{eq:gk_mono} to conclude. A similar procedure can be applied to prove that
  $k_{i,j}\leq 1$. Furthermore, we have
  \begin{align*}
    \sumZxi{j}{g_{i,j}} = \sumZxi{j}{\left[ \D_+y_j - \D_+\zeta_j \right] g_{i,j}} &= 1 + \sumZxi{j}{\zeta_{j+1} (\D_{j+}g_{i,j})}, \text{ from \eqref{eq:sum_id2},} \\
                                                                                   &= 1 + \sumZxi{j}{\zeta_{j+1} (\D_+y_j)\gamma_{i,j}}, \text{ from \eqref{eq:discggkkid},}  \\
                                                                                   &\le 1 + \normlinfty{\zeta} \sumZxi{j}{(\D_+y_j)|\gamma_{i,j}|},
  \end{align*}
  and the result on the $\lone$ bound of $g$ follows from
  \eqref{eq:sum_id1} and \eqref{eq:fml_bnds}. A similar procedure proves
  the bound on $\normlone{k}$.  For the bound on $\normlone{\gamma}$ we
  find
  \begin{align*}
    \sumZxi{j}{|\gamma_{i,j}|} &= \sumZxi{j}{\left[ \D_+y_j - \D_+\zeta_j \right] |\gamma_{i,j}|} \\
                               &= 2 g_{i,i} - \dxi \sum_{j=-\infty}^{i-1} (\D_+\zeta_j) \gamma_{i,j} + \dxi \sum_{j=i}^{+\infty} (\D_+\zeta_j) \gamma_{i,j} \\
                               &= 2 \normlinfty{g} - 2 \zeta_i \gamma_{i,i-1} + \dxi \sum_{j=-\infty}^{i-1} \zeta_j (\D_{j-}\gamma_{i,j}) - \dxi \sum_{j=i}^{+\infty} \zeta_j (\D_{j-}\gamma_{i,j}) \\
                               &= 2 \normlinfty{g} + (1- 2 \gamma_{i,i-1}) \zeta_i + \sumZxi{j}{\sgn\left(i-j-\frac12 \right) \zeta_i (\D_+y_j) g_{i,j}} \\
                               &\le 2 \normlinfty{g} + \normlinfty{\zeta} \left[ |1-2\gamma_{i,i-1}| + \sumZxi{j}{(\D_+y_j)g_{i,j}} \right],
  \end{align*}
  where in the second equality we use Lemma \ref{lem:gk_sgn}, the third
  equality uses summation by parts \eqref{eq:sum_by_parts}, and the fourth
  is due to the kernel definition property \eqref{eq:discggkkid}. Then the result
  follows from \eqref{eq:sum_id2}, \eqref{eq:fml_bnds}, and $0 \le
  \gamma_{i,i-1} \le 1$.  A similar procedure proves the bound on
  $\normlone{\kappa}$.
\end{proof}

A direct consequence of \eqref{eq:fml_bnds} is that the $\linfty$-norms of the
kernels remain bounded by 1 for all time.  Moreover, combining
\eqref{eq:fml_l1_1} with \eqref{eq:bnd_zeta} we find that the $\lone$-norms
remain bounded for any finite $t$, namely
\begin{align}\label{eq:fml_l1_2}
  \begin{aligned}
    \normlone{g(t)}, \normlone{k(t)} &\le 1 + 2\left[ \normlinfty{\zeta(0)} + \sqrt{2\Htotal}t \right], \\
    \normlone{\gamma(t)}, \normlone{\kappa(t)} &\le 2\left[ 1 + \normlinfty{\zeta(0)} + \sqrt{2\Htotal}t \right].
  \end{aligned}
\end{align}

Furthermore, Lemma \ref{lem:kern_ineq} allows us to find a bound similar to \eqref{eq:bnd_U} for $\normlinfty{R}$ and $\normlinfty{Q}$.
Indeed, for $Q$ we find
\begin{equation}
  \begin{aligned}
    \label{eq:linfestQ}
    \normslinfty{Q} &\leq \normslinfty{g\ast(U (\D_+U))} + \normslinfty{\kappa\ast(h + \rho_\infty\bar r)}\\
    &\leq \normslinfty{g}\normslone{U (\D_+U)}+ \normslinfty{\kappa}\normslone{h} + \rho_\infty\normslinfty{\kappa\ast \abs{\bar r}}.
  \end{aligned}
\end{equation}
Using \eqref{eq:rb_hDy} and the Cauchy--Schwarz inequality, we have
\begin{equation*}
  \rho_{\infty} \normslinfty{\kappa\ast \abs{\bar r}} \leq \frac12\rho_{\infty}^2 \normslinfty{\abs{\kappa}\ast (\D_+y)} + \frac12 \normslinfty{\abs{\kappa}\ast (2h)}
\end{equation*}
which by \eqref{eq:sum_id1} and \eqref{eq:fml_bnds} simplifies to
\begin{equation*}
  \rho_{\infty} \normslinfty{\kappa\ast \abs{\bar r}}\leq \frac12\rho_{\infty}^2 (2\normslinfty{k}) + \normslinfty{\kappa}\normslone{h} \leq \rho_{\infty}^2 + \Htotal.
\end{equation*}
Using \eqref{eq:UDUleH}, we get $\normslone{UD_+U}\leq \Htotal$. Hence, from \eqref{eq:linfestQ}, we get
\begin{equation*}
  \normslinfty{Q} \leq 3 \Htotal + \rho_\infty^2.
\end{equation*}
An analogous estimate for $R$ can be obtained so that we can conclude with the
bounds
\begin{equation}
  \label{eq:bnd_RQ}
  \sup_{0 \le t \le T}\normlinfty{R(t)} \le 3 \Htotal + \frac12\rho_{\infty}^2, \qquad \sup_{0 \le t \le T}\normlinfty{Q(t)} \le 3 \Htotal + \rho_{\infty}^2.
\end{equation}
Now we are set to prove global existence for solutions of \eqref{eq:ODE_disc}.
\begin{thm}[Global existence]
  Given initial datum $X_0$ in the set $\Bcal$ from Definition
  \ref{def:setB}, the system \eqref{eq:ODE_disc} admits a unique global solution
  $X \in \mathbf{C}^1([0,\infty), \Ed)$, such that $X \in \Bcal$ for all
  times.  In particular, for $t > 0$, the norm $\normEd{X(t)}$ is bounded by $ C
  \normEd{X(0)}$ for a constant $C$ depending only on $t$, the total energy  $\Htotal$, the asymptotic density
  $\rho_{\infty}$, and $\normlinfty{\zeta(0)}$.
  \label{thm:solGlobal}
\end{thm}

\begin{proof}
  The solution has a finite time of existence $T$ only if
  \begin{equation*}
    \normEd{X} = \normVd{\zeta} + \normhone{U} + \normVd{H} + \normltwo{\bar{r}}
  \end{equation*}
  blows up as $t$ approaches $T$. Otherwise the solution can be prolonged by a small time interval by Theorem \ref{thm:solLocal}.
  Let $X$ be the short-time solution given by \ref{thm:solLocal} for initial datum $X_0$.  We will prove that $\sup_{0 \le t \le T}
  \normEd{X} < \infty$.

  From the definition of the $\hone$-norm and \eqref{eq:sobolev_l} we find that the right-hand side of \eqref{eq:ODE_disc_z} is
  bounded in the $\Vd$-norm by $\frac{2+\sqrt{2}}{2} \normhone{U}$, while the right-hand side of \eqref{eq:ODE_disc_r} is bounded
  in $\ltwo$-norm by $\rho_{\infty}\normhone{U}$. Next, we estimate the right hand side of \eqref{eq:ODE_disc_U},
  \begin{align*}
    \normhone{Q} &\le \normltwo{Q} + \normltwo{\D_+Q} \le \normltwo{Q} + \normltwo{R (1+\D_+\zeta) - h - \rho_{\infty}\bar{r}} \\
                 &\le \normltwo{Q} + \normltwo{R} + \normlinfty{R} \normltwo{\D_+\zeta} + \normltwo{h + \rho_{\infty}\bar{r} },
  \end{align*}
  where we have used the definition of the $\hone$-norm, \eqref{eq:discRQrel} and the decomposition $\D_+y_j = 1 + \D_+\zeta_j$.
  Then, recalling the definitions \eqref{eq:Q1} and \eqref{eq:Q2} and applying the Young inequality \eqref{eq:altYoung} to the
  final expression above we see that it is bounded by
  \begin{align*}
    &\normlone{g} \normltwo{U (\D_+U)} + \normlone{\gamma}^{\frac12} \normlone{\kappa}^{\frac12} \normltwo{h + \rho_{\infty}\bar{r} } + \normlone{\gamma}^{\frac12} \normlone{\kappa}^{\frac12} \normltwo{U (\D_+U)} \\
    &\qquad+ \normlinfty{R} \normltwo{\D_+\zeta} + \normlone{k} \normltwo{h + \rho_{\infty}\bar{r} } + \normltwo{h + \rho_{\infty}\bar{r} } \\
    &\le \left[\normlone{g} + \normlone{\gamma}^{\frac12} \normlone{\kappa}^{\frac12} \right] \normlinfty{U} \normltwo{\D_+U} + \normlinfty{R} \normltwo{\D_+\zeta} \\
    &\qquad+ \left[ \normlone{k} + \normlone{\gamma}^{\frac12} \normlone{\kappa}^{\frac12} \right] \left[ \normltwo{h} + \rho_{\infty}\normltwo{\bar{r}} \right].
  \end{align*}
  Then, applying \eqref{eq:UDUleH}, \eqref{eq:bnd_U}, \eqref{eq:fml_l1_2}, \eqref{eq:bnd_RQ} and the definitions of the $\Vd$- and $\hone$-norms we obtain that $\normhone{Q}$ is bounded by
  \begin{equation*}
    \left( 3 + 4 [\normlinfty{\zeta(0)} + \sqrt{2\Htotal}t] \right)[\normhone{U} + \normVd{H} + \rho_{\infty} \normltwo{\bar{r}}] + \left(3\Htotal + \frac12 \rho_{\infty}^2\right)\normVd{\zeta}.
  \end{equation*}
  Finally, the $\Vd$-norm of the right-hand side of \eqref{eq:ODE_disc_H} can be estimated as
  \begin{align*}
    \normVd{U (\uptau R)} &= \normlinfty{U (\uptau R)} + \normltwo{[U^2-R](\D_+U) - U Q [1 + \D_+\zeta]} \\
                          &\le \normlinfty{R} \normlinfty{U} + [\normlinfty{U}^2 + \normlinfty{R}]\normltwo{\D_+U} \\
                          &\quad+ \normlinfty{Q}\normltwo{U} + \normlinfty{Q}\normlinfty{U}\normltwo{\D_+\zeta} \\
                          &\le \left(\frac{2+\sqrt{2}}{2}(3\Htotal + \rho_\infty^2) + 2\Htotal \right) \normhone{U} \\
                          &\quad+ \sqrt{2\Htotal}\left(3\Htotal + \frac12\rho_{\infty}^2 \right)\normVd{\zeta} ,
  \end{align*}
  where we again use the notation $(\uptau R)_j = R_{j-1}$.
  In the first identity above we have employed \eqref{eq:discRQrel}, while in the final line we have used the definitions of the $\Vd$- and $\hone$-norms together with \eqref{eq:sobolev_l}, \eqref{eq:bnd_U} and \eqref{eq:bnd_RQ}.
  
  Gathering all the above estimates of the right-hand sides, writing \eqref{eq:ODE_disc} in integral form, and taking norms we
  obtain the following inequality for $X(t) = (\zeta,U,H,\bar{r})(t)$,
  \begin{equation*}
    \normEd{X(t)} \le \normEd{X(0)} + C(\Htotal,\normlinfty{\zeta(0)},\rho_{\infty}) \int_{0}^{t}(1+s)\normEd{X(s)}\,ds, 
  \end{equation*}
  for $t \in [0,T]$ and some constant $C(\Htotal,\normlinfty{\zeta(0)},\rho_{\infty})$ depending only on $\Htotal$,
  $\normlinfty{\zeta(0)}$ and $\rho_{\infty}$.  Gr\"{o}nwall's inequality then yields
  \begin{equation*}
    \normEd{X(t)} \le \normEd{X(0)} \exp\left\{ C(\Htotal,\normlinfty{\zeta(0)},\rho_{\infty})\left[t + \frac{1}{2}t^2 \right] \right\}, \qquad t \in [0,T],
  \end{equation*}
  which shows that $\normEd{X(T)}$ is bounded, and we may according to Theorem \ref{thm:solLocal} extend our solution indefinitely.
  
  In retrospect, with the estimates \eqref{eq:bnd_U} and \eqref{eq:bnd_RQ} in hand, we can apply a Gr\"{o}nwall estimate to the evolution equations
  for $\alpha$, $\beta$, $h$, and $\bar{r}$ in \eqref{eq:aux}. From this we find that the $\linfty$-norms of $\D_{+}y$,
  $\D_+U$, $h$, and $\bar{r}$ at time $t \in [0,T]$ are bounded by their $\linfty$-norm at time $t = 0$ times a factor
  $\exp\{C(\Htotal,\rho_{\infty})t\}$, where the constant $C(\Htotal,\rho_{\infty})$ depends only on $\Htotal$ and
  $\rho_{\infty}$.
\end{proof}

As mentioned in the introduction, if $\rho>0$ initially for the 2CH system \eqref{eq:2CH}, then the smoothness of the initial data is preserved, see
\cite{Grunert2012}. This is because the characteristics do not collide in this case, and $y_\xi$ remains
positive for all time. In the discrete case, this property takes the form of a lower bound on $\D_+y$. For any given time $T$,
there exists a constant $C>0$ depending on $\max_{t\in[0, T]}\norm{X(t)}_{\Ed}$, $\rho_\infty$, and $T$ such that
\begin{equation}
  \label{eq:lowboundDy}
  (\D_+ y)_j(t) \geq \frac{\rho_{0,j}^2}{C},
\end{equation}
for all $j$ and $t\in[0, T]$. This follows from property \eqref{setB:id} in Definition \ref{def:setB}. Thus, if $\rho_{0,j}>0$, we will have
$y_j(t) < y_{j + 1}(t)$ for all time.

\section{Convergence of the scheme}
\label{sec:convergence}

In this section we interpolate the solutions of the semi-discrete scheme analyzed in Section \ref{sec:existence}, with initial data constructed in Section \ref{sec:init_data}.
We shall then show that these interpolated functions converge to the solution of
the 2CH system as written in \eqref{eq:old2CH} and \eqref{eq:RQsysdeflag}. Let us in this section use $Y_{\dxi}$ to denote the
tuple of grid functions obtained in Theorem \ref{thm:solGlobal} for $t \in [0,T]$,
\begin{equation}\label{eq:grid_tuple}
  Y_{\dxi}(t) = (\zeta, U, H, \bar{r})(t) \in \Ed.
\end{equation}
Since these functions also belong to the set $\Bcal$ in Definition  \ref{def:setB}, we will augment the $\Ed$-norm \eqref{eq:Ed} as follows:
\begin{equation*}
  \norm{Y_{\dxi}}_\Bcal = \normEd{Y_{\dxi}} + \normlinfty{\D_+\zeta} + \normlinfty{\D_+U} + \normlinfty{\D_+H} + \normlinfty{\bar{r}}.
\end{equation*}
In order to ease notation below, we will write $\norm{Y_{\dxi}}$ for $\sup_{0 \le t \le T}\norm{Y_{\dxi}(t)}_\Bcal$. We define the
interpolated functions as follows
\begin{align}\label{eq:interpol}
\begin{aligned}
    V_{\itD}(t,\xi) &= \sumZ{j}{\left[ V_j(t) + (\xi-\xi_j)(\D_+V_j(t)) \right] \chi_j(\xi) }, \qquad \bar{r}_{\itD}(t,\xi) = \sumZ{j}{\bar{r}_j(t) \chi_j(\xi) }, \\
    R_{\itD}(t,\xi) &= \sumZ{j}{\left[ R_{j}(t) + (\xi-\xi_{j+1})(\D_-R_j(t)) \right] \chi_j(\xi)},
\end{aligned}
\end{align}
where $V$ is a placeholder for $\zeta$, $U$, $H$, and $Q$, while $\chi_j(\xi)$ denotes the indicator function for the interval $[\xi_j, \xi_{j+1})$.
We also introduce the functions
\begin{equation}\label{eq:ipyr}
  y_{\itD}(t,\xi) \coloneqq \xi + \zeta_{\itD}(t,\xi), \qquad r_{\itD}(t,\xi) \coloneqq \bar{r}_{\itD}(t,\xi) + \rho_{\infty} \pdiff{\xi}{y_{\itD}(t,\xi)}.
\end{equation}
Observe that the interpolated functions above are piecewise linear and continuous, except for $r_\itD, \bar{r}_\itD$ which are
piecewise constant.  In particular we note the identity
\begin{equation*}
  R_j + (\xi-\xi_{j+1}) (\D_-R_j) = R_{j-1} + (\xi-\xi_j) (\D_-R_j), \quad \xi \in [\xi_j, \xi_{j+1}],
\end{equation*}
which shows $R_\itD(t,\xi_j) = R_{j-1}$. Let us also recall the definition of the space $\E$ in \eqref{eq:spaceE}.
A consequence of Theorem \ref{thm:solGlobal} is that tuple of interpolated functions
\begin{equation}\label{eq:Xdisc}
  X_{\itD}(t) \coloneqq \left( \zeta_{\itD}(t,\cdot), U_{\itD}(t,\cdot), H_{\itD}(t,\cdot), \bar{r}_{\itD}(t,\cdot) \right)
\end{equation}
satisfies $X_{\itD}(t) \in \mathbf{C}^1([0,T],\E )$ for any fixed $T > 0$ and $\dxi > 0$. Let us now consider a given initial datum $X_0 =
(\zeta_0, U_0, H_0, \bar r_0)\in\E$ for the equivalent 2CH system \eqref{eq:old2CH}. We assume there exists a sequence of discrete initial
data $Y_{\dxi, 0} \in \Ed$ such that the interpolation of
$Y_{\dxi, 0}$, denoted $X_{\itD, 0}$, converges to $X_0$, i.e.,
\begin{equation}\label{eq:init_conv}
  \lim_{\dxi\to 0}\norm{X_{\itD, 0} - X_0}_{\E} = 0.
\end{equation}
We will explain how to construct such sequence in the next section. For $T>0$ and each $Y_{\dxi,0}$, let $Y_{\dxi}$ be the corresponding solution
given by Theorem \ref{thm:solGlobal}. Furthermore, we denote by $X \in \mathbf{C}([0, T], \E)$ the
solution to \eqref{eq:ODE_disc} with initial data $X_{0}$, while $X_\itD\in \mathbf{C}([0, T], \E)$ is the function interpolated from
$Y_\dxi$ using \eqref{eq:interpol}. Then, we have the following convergence result.
\begin{thm}[Convergence]
  \label{thm:conv}
  The approximation $X_\itD$ converges to the solution $X$ to the 2CH system \eqref{eq:old2CH} in $\mathbf{C}([0, T], \E)$.
\end{thm}

\begin{proof}[Proof of Theorem \ref{thm:conv}]
  The strategy of the proof is to show that our interpolated functions $(y_{\itD}, U_{\itD}, H_{\itD}, r_{\itD})$ satisfy
  \eqref{eq:old2CH} and \eqref{eq:RQsysdeflag}, where we allow for a small error of order $\Ocal(\dxi)$.  For \eqref{eq:old2CH_y},
  \eqref{eq:old2CH_U}, and \eqref{eq:old2CH_r}, we observe that, by construction, we have
\begin{equation*}
  \pdiff{t}{y_{\itD}} = U_{\itD},\quad \pdiff{t}{U_{\itD}} = -Q_{\itD},\quad  \pdiff{t}{r_{\itD}} = 0 
\end{equation*}
due to \eqref{eq:discRQrel}, \eqref{eq:ODE_disc_z}, \eqref{eq:ODE_disc_U}, and \eqref{eq:ODE_disc_r}.  Thus, the three linear
equations in \eqref{eq:old2CH} are satisfied exactly by our interpolants.  The next step is to control the evolution of the error
for the variable $H_{\itD}$. We find
\begin{equation*}
  \pdiff{t}{H_{\itD}} = -U_{\itD} R_{\itD} + \sumZ{j}{(\xi-\xi_j)(\xi-\xi_{j+1})(\D_+U_j)(\D_-R_j) \chi_j}.
\end{equation*}
This identity then implies
\begin{equation*}
  \left( \pdiff{t}{H_{\itD}} + U_{\itD} R_{\itD} \right)_\xi = \sumZ{j}{(2\xi-\xi_j-\xi_{j+1}) (\D_+U_j)(\D_-R_j) \chi_j},
\end{equation*}
almost everywhere. Combining the above identities we can estimate the error in
the $\V$-norm as follows,
\begin{align}\label{eq:est_H}
  \begin{aligned}
    &\normV{\pdiff{t}{H_{\itD}} + U_{\itD} R_{\itD} } \\
    &\quad\le \dxi^2 \sumZ{j}{\abs{\D_+U_j}\abs{\D_-R_j}} + \left( \sumZxi{j}{\dxi^2 \abs{\D_+U_j}^2\abs{\D_-R_j}^2} \right)^{\frac12} \\
    &\quad\le \dxi \normltwo{\D_+U}\normltwo{\D_-R} + \dxi \normlinfty{\D_+U}\normltwo{\D_-R} \\
    &\quad\le \dxi \left(\normltwo{\D_+U} + \normlinfty{\D_+U}\right)\normltwo{(\D_+y)Q - U(\D_+U)} \\
    &\quad\le \dxi \left(\normltwo{\D_+U} + \normlinfty{\D_+U}\right) \left(\normlinfty{\D_+y}\normltwo{Q} + \normlinfty{U} \normltwo{\D_+U}\right).
  \end{aligned}
\end{align}

Now, for the relations \eqref{eq:RQsysdeflag}, we measure the error in $\Ltwo$-norm. From \eqref{eq:discRQrel}, we obtain the
relation
\begin{equation*}
  \pdiff{\xi}{y_{\itD}} Q_{\itD} - \pdiff{\xi}{R_{\itD}} - U_{\itD} \pdiff{\xi}{U_{\itD}} = \sumZ{j}{(\xi-\xi_j) \left[ (\D_+y_j) (\D_+Q_j) - (\D_+U_j)^2\right] \chi_j},
\end{equation*}
and find
\begin{multline}
  \label{eq:est_dR}
  \normLtwo{ \pdiff{\xi}{y_{\itD}} Q_{\itD} - \pdiff{\xi}{R_{\itD}} - U_{\itD} \pdiff{\xi}{U_{\itD}} }\\
  \le \dxi \left( \normlinfty{\D_+y}\normltwo{\D_+Q} + \normlinfty{\D_+U}\normltwo{\D_+U} \right).
\end{multline}
Finally, using \eqref{eq:discRQrel} once more, we have
\begin{align*}
  &\pdiff{\xi}{y_{\itD}} R_{\itD} - \pdiff{\xi}{S_{\itD}} - \pdiff{\xi}{H_{\itD}} -  \rho_{\infty} \bar{r}_{\itD} = \sumZ{j}{(\xi-\xi_{j+1})(\D_-R_j)(\D_+y_j) \chi_j }
\end{align*}
which can be estimated as
\begin{equation}\label{eq:est_dQ}
  \normLtwo{ \pdiff{\xi}{y_{\itD}} R_{\itD} - \pdiff{\xi}{S_{\itD}} - \pdiff{\xi}{H_{\itD}} -  \rho_{\infty} \bar{r}_{\itD} } \le \dxi \normlinfty{\D_+y} \normltwo{\D_-R}.
\end{equation}
The estimate \eqref{eq:est_H} is exactly as we want it, \eqref{eq:old2CH_H} is satisfied in the appropriate norm up to some small remainder.
However, the estimates \eqref{eq:est_dR} and \eqref{eq:est_dQ} require some more work, as we shall see next.

Let us estimate the $\E$-norm of the difference between $X_{\itD}(T)$ and the exact solution $X(T) \coloneqq (\zeta,U,H,\bar{r})(T)$.
From the above estimates and \eqref{eq:old2CH} we find
\begin{align}
  \label{eq:normdiff}
  \begin{aligned}
    \normV{(\zeta_{\itD}-\zeta)(T,\cdot)} &\le \normV{(\zeta_{\itD}-\zeta)(0,\cdot)} + \int_{0}^{T} \normV{ (U_{\itD} - U)(t,\cdot) }\,dt \\
    \normHone{(U_{\itD} - U)(T,\cdot)} &\le \normHone{(U_{\itD} - U)(0,\cdot)} + \int_{0}^{T} \normHone{(Q_{\itD} - Q)(t,\cdot)}\,dt \\
    \normV{ (H_{\itD} - H)(T,\cdot)} &\le \normV{ (H_{\itD} - H)(0,\cdot)} + \int_{0}^{T} \normV{ (U_{\itD}R_{\itD} - U R)(t,\cdot) }\,dt \\
    &\quad+ \dxi C_H(\norm{Y_{\dxi}}) T \\
    \normLtwo{(\bar{r}_{\itD} - \bar{r})(T,\cdot)} &\le \normLtwo{(\bar{r}_{\itD} - \bar{r})(0,\cdot)} + \rho_{\infty} \int_{0}^{T} \normLtwo{ \pdiff{\xi}{(U_{\itD} - U)(t,\cdot)} }\,dt,
  \end{aligned}
\end{align}
where we have used that the final expression in \eqref{eq:est_H} can be bounded by $\dxi C_H(\norm{Y_{\dxi}})$ for some constant $C_H$ depending only on $\norm{Y_{\dxi}}$.

From \eqref{eq:normdiff}, it is clear that we need estimates of $\normHone{Q_{\itD}-Q}$, $\normLinfty{R_{\itD}-R}$, and $\normLtwo{(R_{\itD} -R)_\xi}$ in terms of
\begin{equation*}
  \normE{X_{\itD}-X} = \normV{\zeta_{\itD} - \zeta} + \normHone{U_{\itD} - U} + \normV{H_{\itD}-H} + \normLtwo{\bar{r}_{\itD}-\bar{r}},
\end{equation*}
and by definition of the $\Hone$-norm and the Sobolev inequality $\normLinfty{f} \le \tfrac{1}{\sqrt{2}}\normHone{f}$, it will be
sufficient to bound $\normHone{Q_{\itD}-Q}$ and $\normHone{R_{\itD}-R}$.  To this end, we note that by the estimates
\eqref{eq:est_dR} and \eqref{eq:est_dQ}, it follows that
\begin{equation} \label{eq:iplRQid}
  \begin{bmatrix}
    -\partial_\xi & (y_{\itD})_\xi \\
    (y_{\itD})_\xi & -\partial_\xi
  \end{bmatrix} \circ \begin{bmatrix}
    R_{\itD} \\ Q_{\itD}
  \end{bmatrix} = \begin{bmatrix}
    U_{\itD} (U_{\itD})_\xi \\
    (H_{\itD})_\xi + \rho_{\infty} \bar{r}_{\itD}
  \end{bmatrix} + \dxi \begin{bmatrix}
    v_{\itD} \\ w_{\itD}
  \end{bmatrix}
\end{equation}
for some functions $v_{\itD}, w_{\itD} \in \Ltwo$ which are bounded by a constant depending only on the norm $\norm{Y_{\dxi}}$ of \eqref{eq:grid_tuple}.
Recalling \eqref{eq:explicitRQintform} and the operators defined in \eqref{eq:GKid} we know  that $R(t,\xi)$ and $Q(t,\xi)$ can be written as
\begin{align*}
  R(t,\xi) &= \intR{\eta}{ \kappa[t](\eta,\xi) UU_\xi(t,\eta) } + \intR{\eta}{ g[t](\eta,\xi) [ H_\xi +  \rho_{\infty} \bar{r} ](t,\eta) } \\
           &= \Kcal\left(UU_\xi\right) + \Gcal\left( H_\xi +  \rho_{\infty} \bar{r} \right), \\
  Q(t,\xi) &= \intR{\eta}{ g[t](\eta,\xi) UU_\xi(t,\eta) } + \intR{\eta}{ \kappa[t](\eta,\xi) [ H_\xi +  \rho_{\infty} \bar{r} ](t,\eta) } \\
           &= \Gcal\left(UU_\xi\right) + \Kcal\left( H_\xi +  \rho_{\infty} \bar{r} \right)
\end{align*}
with kernels
\begin{equation*}
  g[t](\eta,\xi) \coloneqq \tfrac12 e^{ -| y(t,\xi) - y(t,\eta) |}, \qquad \kappa[t](\eta,\xi) \coloneqq -\sgn(\xi - \eta) g[t](\eta,\xi).
\end{equation*}
Due to the obvious similarities between \eqref{eq:iplRQid} and \eqref{eq:newdefQR} we would like to generalize the operator identity \eqref{eq:GKid} by replacing $y(t,\xi)$ with any function $b(t,\xi)$ such that $b(t,\cdot)-\id \in \V$ and $b_\xi(t,\xi) \ge 0$, in particular this holds for our $y_{\itD}(t,\xi)$ in \eqref{eq:ipyr} by virtue of Lemma \ref{lem:solInSet}.
This is can be done, and the unique $\Hone$-solution of
\begin{equation*}
  \begin{bmatrix}
    -\partial_\xi & b_{\xi}(t,\xi) \\ b_{\xi}(t,\xi) & -\partial_\xi
  \end{bmatrix} \begin{bmatrix}
    \phi(t,\xi) \\ \psi(t,\xi)
  \end{bmatrix} = \begin{bmatrix}
    v(t,\xi) \\ w(t,\xi)
  \end{bmatrix}
\end{equation*}
for $v(t,\cdot), w(t,\cdot) \in \Ltwo$ is then
\begin{align*}
  \phi(t,\xi) &= \intR{\eta}{ \frac{1}{2}e^{-|b(t,\xi)-b(t,\eta)|} \left[ w(t,\eta) -\sgn(\xi-\eta) v(t,\eta) \right] }, \\
  \psi(t,\xi) &= \intR{\eta}{ \frac{1}{2}e^{-|b(t,\xi)-b(t,\eta)|}  \left[ v(t,\eta) - \sgn(\xi-\eta) w(t,\eta) \right] }.
\end{align*}
Consequently, we can generalize $\Gcal$ and $\Kcal$ from \eqref{eq:GKid} to be operators from $\V \times \Ltwo$ to $\Hone$ as follows,
\begin{align}
  \Gcal[t,\xi](b-\id,f) &\coloneqq \intR{\eta}{ \frac{1}{2}e^{-|b(t,\xi)-b(t,\eta)|} f(\eta) }, \label{eq:G2} \\
  \Kcal[t,\xi](b-\id,f) &\coloneqq -\intR{\eta}{ \sgn(\xi-\eta) \frac{1}{2}e^{-|b(t,\xi)-b(t,\eta)|} f(\eta) }. \label{eq:K2}
\end{align}
Using these operators, we may write the general solutions $\phi(t,\xi)$, $\psi(t,\xi)$ as
\begin{align*}
  \phi(t,\xi) &= \Kcal[t,\xi](b-\id,v) + \Gcal[t,\xi](b-\id,w), \\
  \psi(t,\xi) &= \Gcal[t,\xi](b-\id,v) + \Kcal[t,\xi](b-\id,w).
\end{align*}
An argument analogous to \cite[Lem.\ 3.1]{Grunert2012} then proves that the operators
\begin{equation*}
  \Rcal_1[t,\cdot] \colon (\zeta, U, H, \bar{r}) \mapsto \Kcal[t,\cdot](\zeta,U U_\xi) + \Gcal[t,\cdot](\zeta,H_\xi + \rho_{\infty} \bar{r})
\end{equation*}
and
\begin{equation*}
  \Rcal_2[t,\cdot] \colon (\zeta,v,w) \mapsto \Kcal[t,\cdot](\zeta, v) + \Gcal[t,\cdot](\zeta,w)
\end{equation*}
are locally Lipschitz as operators from $\E \to \Hone$ and $\V \times (\Ltwo)^2 \to \Hone$ respectively,
and the same is true for
\begin{equation*}
  \Qcal_1[t,\cdot] \colon (\zeta, U, H, \bar{r}) \mapsto \Gcal[t,\cdot](\zeta,U U_\xi) + \Kcal[t,\cdot](\zeta,H_\xi + \rho_{\infty} \bar{r})
\end{equation*}
and
\begin{equation*}
  \Qcal_2[t,\cdot] \colon (\zeta, v, w) \mapsto \Gcal[t,\cdot](\zeta,v) + \Kcal[t,\cdot](\zeta,w).
\end{equation*}

Finally turning back to the functions we are interested in, we note that, since
our interpolants $R_{\itD}$ and $Q_{\itD}$ are solutions of \eqref{eq:iplRQid},
they can be written as
\begin{align*}
  R_{\itD}(t,\xi) &= \Rcal_1[t,\xi]\left(\zeta_{\itD}, U_{\itD}, H_{\itD},\bar{r}_{\itD} \right) + \dxi \, \Rcal_2[t,\xi](\zeta_{\itD},v_{\itD},w_{\itD}), \\
  Q_{\itD}(t,\xi) &= \Qcal_1[t,\xi]\left(\zeta_{\itD}, U_{\itD}, H_{\itD},\bar{r}_{\itD} \right) + \dxi \, \Qcal_2[t,\xi](\zeta_{\itD},v_{\itD},w_{\itD}).
\end{align*}
These should then be compared to $R$ and $Q$ for the exact solution, which now can be written as
\begin{align*}
  R(t,\xi) &= \Rcal_1[t,\xi](\zeta, U, H, \bar{r}), \\
  Q(t,\xi) &= \Qcal_1[t,\xi](\zeta, U, H, \bar{r}).
\end{align*}
Then, we write
\begin{align*}
  Q_{\itD}(t,\xi) - Q(t,\xi) &= \Qcal_1\left(\zeta_{\itD}, U_{\itD}, H_{\itD}, \bar{r}_{\itD}\right) - \Qcal_1\left(\zeta, U, H, \bar{r} \right) \\
                             &\quad+ \dxi \, \Qcal_2[t,\xi](\zeta_{\itD},v_{\itD},w_{\itD})
\end{align*}
and it follows from the Lipschitz property that
\begin{align*}
  \normHone{ Q_{\itD}(t,\cdot) - Q(t,\cdot) } &\le C_{Q,1}(\normE{X_{\itD}(t)}, \normE{X(t)}) \normE{X_{\itD}(t)-X(t)} \\
                                              &\quad+ \dxi C_{Q,2}(\norm{Y_{\dxi}})
\end{align*}
for constants $C_{Q,1}, C_{Q,2}$, and an analogous estimate holds for $\normHone{R_{\itD}(t,\cdot)-R(t,\cdot)}$.

From the above estimates, the obvious inequality $\normLtwo{f_\xi} \le \normHone{f}$, and $\normV{f} \le
\tfrac{2+\sqrt{2}}{2}\normHone{f}$ coming from $\normLinfty{f} \le \frac{1}{\sqrt{2}}\normHone{f}$, we may add the equations in
\eqref{eq:normdiff} to obtain
\begin{align*}
  \normE{X_{\itD}(T) - X(T)} &\le \normE{X_{\itD}(0) - X(0)} + \dxi C_1(\norm{Y_{\dxi}}) T \\
                             &\quad+ C_2(\norm{Y_{\dxi}}, \norm{X}) \int_{0}^{T} \normE{X_{\itD}(t)-X(t)}\,dt,
\end{align*}
where we have used $\norm{X} \coloneqq \sup_{0 \le t \le T}\normE{X(t)}$ and $\sup_{0 \le t \le T}\normE{X_{\itD}} \le C(\norm{Y_{\dxi}})$ are bounded by constants depending on $T$ and the $\E$-norm of their initial data.
In particular, by Theorem \ref{thm:solGlobal} we know $\norm{Y_{\dxi}}$ is bounded by a constant depending only on $T$, $\Htotal$, $\normlinfty{\zeta(0)}$, and $\rho_{\infty}$.
Gr\"{o}nwall's inequality then yields the estimate
\begin{equation*}
  \normE{X_{\itD}(T) - X(T)} \le C_3\left(\norm{Y_{\dxi}}, \norm{X}\right) \left[ \normE{X_{\itD}(0) - X(0)} + \dxi C_1(\norm{Y_{\dxi}})T \right].
\end{equation*}
Combining this estimate with \eqref{eq:init_conv}, we obtain
the desired result.
\end{proof}

Since convergence in Lagrangian coordinates implies convergence in the corresponding Eulerian coordinates, see \cite{Grasmair2018}
for details, this shows that interpolated solutions of the discrete two-component Camassa--Holm system can be used to obtain
conservative solutions of the 2CH system \eqref{eq:2CH}.  In particular, as conservative solutions of \eqref{eq:CH} are unique
according to \cite{Bressan2015}, our discretization of the CH equation corresponds to the unique conservative solution of the CH
equation.

\section{Construction of the  initial data}
\label{sec:init_data}

In this section we consider initial data for the continuous system given by $u_0\in\Hone$, $\bar{\rho}_0 = \rho_0 -
\rho_\infty \in\Ltwo$, and a measure $\mu_0$ which corresponds to the energy distribution, see \cite{Grunert2012}.  To ease
notation we omit the subscript 0 and the dependence on $t$ for the rest of this section, as we are always considering $t =
0$.  The absolutely continuous part of the measure $\mu$ satisfies
\begin{equation*}
  \mu_{\text{\rm ac}} = \tfrac12(u^2 + u_x^2 + \bar{\rho}^2 )\, dx,
\end{equation*}
and may in general contain singular parts.  Here we will restrict ourselves to the case where the singular part is purely
atomic, and construct corresponding initial data for the discrete scheme. The ability to handle singular initial data was one
of the motivations for the effort put into Section \ref{sec:aux} to allow for $\D_+y_i = 0$.  From \cite[Thm.\
4.9]{Grunert2012}, we know the functions $(y,U,H,r)$ defined as
\begin{subequations}
  \label{eq:continitdata}
\begin{equation}\label{eq:y_sup}
  y(\xi) \coloneqq \sup\{x \enskip:\enskip \mu((-\infty,x)) + x < \xi \},
\end{equation}
\begin{equation}
  U(\xi) = u \circ y(\xi),\quad H(\xi) = \xi - y(\xi)\quad \text{and}\quad \bar{r}(\xi) = (\bar{\rho} \circ y(\xi)) y_\xi(\xi).
\end{equation}
\end{subequations}
give us the initial data for the equivalent system \eqref{eq:old2CH} which provides us the global conservative solutions of
\eqref{eq:2CH} with initial data $(u,\mu)$.

We define the discrete initial data $y_j = y(\xi_j)$ and $U_j = U(y_j)$.  For the $\Ltwo$-function $\bar{r}$ we define
\begin{equation*}
  \bar{r}_j = \frac{1}{\dxi}\int_{\xi_j}^{\xi_{j+1}} \bar{r}(\eta)\,d\eta.
\end{equation*}
The discrete identity \eqref{eq:Bid} is essential to obtain global existence of solution to the semi-discrete system. The
identity reflects the strong connection between the energy variable $H_j$ and the other variables. To fulfill \eqref{eq:Bid},
we set $h_j$ as follows: If $\D_+y_j > 0$, we define $h_j \ge 0$ such that it satisfies \eqref{eq:Bid}, that is
\begin{equation}
  \label{eq:discinith}
  2 h_j = U_j^2 \D_+y_j + \frac{(\D_+U_j)^2}{\D_+y_j} + \frac{\bar{r}_j^2}{\D_+y_j} 
\end{equation}
and if $\D_+y_j = 0$ we set $h_j = \tfrac12$. Then we define $H_j = \dxi\sum_{m=-\infty}^{j-1}h_m$ to ensure $\D_+H_j = h_j$.
Note that in the norms below we will use $U$ to denote both the continuous-case function $U(\xi)$ and the discrete function
$\{U_j\}_{j \in \Z}$. However, the norm used will indicate which of them we are considering: $\lsp{p}$ and $\Vd$ are used for
discrete functions, and $\mathbf{L}^p$ and $\V$ are used for continuous-case functions.

\begin{thm}
  \label{thm:data}
  We consider the initial data of the two-component Camassa--Holm system \eqref{eq:2CH} given by $u_0 \in \Hone$, $\rho_0$
  such that $\rho_0 - \rho_{\infty} \eqqcolon \bar{\rho}_0 \in \Ltwo$ for some $\rho_{\infty} \ge 0$, and a positive finite
  Radon measure $\mu_0$ whose absolutely continuous part satisfies $\mu_{0,\text{ac} } = (u_0^2 + u_{0,x}^2 +
  \bar{\rho}_0^2)\,dx$, while its singular part may be an atomic measure (the singular continuous part is zero).  By
  definition, the global conservative solution of 2CH is obtained by solving \eqref{eq:old2CH} for the initial datum
  $X_0\in\E$ constructed from $(u_0,\rho_0,\mu_0)$, where $X_0$ is given in \eqref{eq:continitdata}. For this $X_0$, we can
  construct sequences of initial data for the semi-discrete scheme, $X_{0,n} = (\zeta_{0,n},U_{0,n},H_{0,n},\bar{r}_{0,n})
  \in \Ed$ such that each element of the sequence belongs to the set $\Bcal$ defined in Definition \ref{def:setB} and the
  interpolation sequence defined in \eqref{eq:interpol} converges to $X_0$ in $\E$.
\end{thm}

\begin{proof}
  Let us start by verifying that these initial data satisfy properties \eqref{setB:kern}--\eqref{setB:pos} in Definition
  \ref{def:setB}.  Clearly, from \eqref{eq:y_sup} it follows that $\D_+y_j \ge 0$, and so the construction in Section
  \ref{sec:aux} gives fundamental solutions satisfying property \eqref{setB:kern}.  Properties \eqref{setB:id} and
  \eqref{setB:pos} have already been satisfied through our definition of $h_j$. To verify \eqref{setB:linf}, we need to show
  that the discrete initial data are uniformly bounded.  Following \cite{holden2007global} and \cite{Grunert2012} we have
  $\abs{y(\xi)-\xi} \le \mu(\R)$, and since the total energy $\mu(\R)$ is bounded, we have $\normLinfty{y-\id} \le
  \mu(\R)$. Since $y_j = y(\xi_j)$, this carries directly over to our setting, $\abs{y_j - \xi_j} \le \mu(\R)$, meaning
  $\normlinfty{\zeta} \le \mu(\R)$.  Moreover, in the aforementioned works, the authors prove that $\xi \mapsto y(\xi)$ is
  Lipschitz with Lipschitz constant 1, which yields
  \begin{equation*}
    \abs{y(\xi_{j+1})-y(\xi_j)} \le \abs{\xi_{j+1}-\xi_j} = \dxi \implies \abs{\D_+y_j} \le 1.
  \end{equation*}
  Hence, $\D_+y \in \linfty$. They also prove $\xi \mapsto \int_{-\infty}^{y(\xi)}u_x^2(x)dx$ to be Lipschitz with Lipschitz
  constant 1.  Then, we have the estimate
  \begin{equation}\label{eq:Udiff_est}
    \abs{U(\xi_{j+1}) - U(\xi_j)} = \abs{\int_{y(\xi_j)}^{y(\xi_{j+1})}u(x)\,dx} \\
    \le \sqrt{y(\xi_{j+1})-y(\xi_j)} \sqrt{\int_{y(\xi_j)}^{y(\xi_{j+1})}u_{x}^2(x)dx},
  \end{equation}
  and from the aforementioned Lipschitz properties we obtain $\abs{U_{j+1}-U_j} \le \dxi$, implying $\abs{\D_+U_j} \le 1$ and
  $\D_+U \in \linfty$.  In addition, since $u \in \Linfty$ it is clear from $U_j = u(y_j)$ that $\normlinfty{U} \le
  \normLinfty{u}$.  From our definition of $\bar{r}_j$ we have the estimate $\abs{\bar{r}_j} \le \sup_\xi \abs{\bar{r}(\xi)}
  \le 1$, where the final inequality comes from \cite[Eq.\ (4.7)]{Grunert2012}, and thus $\bar{r} \in \linfty$. For $h_j$,
  when $\D_+y_j > 0$, we estimate $h_j$ as follows. From \eqref{eq:Udiff_est}, we have $\abs{U_{j+1}-U_j} \le \sqrt{\dxi}
  \sqrt{y_{j+1}-y_j}$, or equivalently $\abs{\D_+U_j} \le \sqrt{\D_+y_j}$. For $\bar{r}_j$ we have
  \begin{equation*}
    \bar{r}_j = \frac{1}{\dxi}\int_{\xi_j}^{\xi_{j+1}} \bar{\rho}(y(\eta)) y_\xi(\eta)\,d\eta \le \frac{1}{\dxi} \sqrt{\int_{\xi_j}^{\xi_{j+1}} \bar{\rho}^2(y(\eta)) y_\xi(\eta) \,d\eta } \sqrt{y(\xi_{j+1}) - y(\xi_j)}.
  \end{equation*}
  Applying once more the continuous-case inequality $(\bar{\rho}\circ y)y_\xi \le 1$, the above estimate yields
  $\bar{r}_j \le \sqrt{\D_+y_j}$. 
  From the preceding estimates, $\D_+y_j \le 1$, and \eqref{eq:discinith} we find
  \begin{equation*}
    2 h_j = U_j^2 \D_+y_j + \frac{(\D_+U_j)^2}{\D_+y_j} + \frac{\bar{r}_j^2}{\D_+y_j} \le U_j^2 + 1 + 1 \le \normLinfty{u}^2 + 2.
  \end{equation*}
  Hence, $h \in \linfty$. We have
  \begin{equation*}
    \abs{\zeta(\xi_{j+1})-\zeta(\xi_j)}^2 = \abs{ \int_{\xi_j}^{\xi_{j+1}} \zeta_\xi(\xi)\,d\xi }^2 \le \dxi \int_{\xi_j}^{\xi_{j+1}} \abs{\zeta_\xi(\xi)}^2\,d\xi,
  \end{equation*}
  or equivalently
  \begin{equation*}
    \dxi \abs{\D_+\zeta_j}^2 \le \int_{\xi_j}^{\xi_{j+1}} \abs{\zeta_\xi(\xi)}^2\,d\xi.
  \end{equation*}
  Summing over $j$ in the above equation we obtain $\normltwo{\D_+\zeta}^2 \le \normLtwo{\zeta_\xi}^2$, and so $\zeta \in \Vd$.
  A completely analogous procedure shows $\normltwo{\D_+U} \le \normLtwo{U_\xi}$. For the $\Ltwo$-norm of $U$ we estimate
  \begin{align*}
    \sumZxi{j}{\abs{U_j}^2} &= \sumZ{j}{\int_{\xi_j}^{\xi_{j+1}}\abs{U(\xi) - \int_{\xi_j}^{\xi}U_\xi(s)\,ds}^2 } \\
                            &\le 2 \sumZ{j}{\int_{\xi_j}^{\xi_{j+1}}\abs{U(\xi)}^2\,d\xi } + 2 \sumZ{j}{\int_{\xi_j}^{\xi_{j+1}}\left(\int_{\xi_j}^{\xi_{j+1}}\abs{U_\xi(s)}\,ds \right)^2\,d\xi } \\
                            &\le 2 \normLtwo{U}^2 + 2 \sumZ{j}{\dxi^2 \int_{\xi_j}^{\xi_{j+1}}\abs{U_\xi(s)}^2\,ds },
  \end{align*}
  which translates into $\normltwo{U}^2 \le 2 \normLtwo{U}^2 + 2 \dxi^2 \normLtwo{U_\xi}$, and so $U \in \hone$.
  For $\bar{r}_j$ we use Jensen's inequality to estimate
  \begin{equation*}
    \bar{r}_j^2 \le \frac{1}{\dxi} \int_{\xi_j}^{\xi_{j+1}} \bar{r}^2(\eta)\,d\eta,
  \end{equation*}
  and multiplying with $\dxi$ and summing over $j$ we obtain $\normltwo{\bar{r}}^2 \le \normLtwo{\bar{r}}^2$.
  Then it remains to check that $H(0) \in \Vd$, and from \eqref{eq:Bid} we estimate
  \begin{align}\label{eq:h_bnd}
    \begin{aligned}
      2 h_j &= U_j^2 \D_+y_j + (\D_+U_j)^2 + \bar{r}_j^2 - 2 h_j \D_+\zeta_j \\
      &\le U_j^2 + (\D_+U_j)^2 + \bar{r}_j^2 + h_j + h_j \abs{\D_+\zeta_j}^2.
    \end{aligned}
  \end{align}
  Now, summing over $j$ we find $\normlone{h} \le \normhone{U}^2 + \normltwo{\bar{r}} +
  \normlinfty{h}\normltwo{\D_+\zeta}^2$, where the right-hand side is bounded by our previous estimates.  Since $h_j > 0$, it
  follows from our definition of $H_j$ that $H_j < H_{j+1}$ and $H_j < \normlone{h}$, which yields $\normlinfty{H} =
  \normlone{h}$. Finally, we have $\normltwo{h} \le \normlinfty{h}\normlone{h}$, so $H \in \Vd$. Thus, we have proved that
  $X_j$ belongs to $\Bcal$.

  Let us now prove that the interpolants for these initial data defined by \eqref{eq:interpol} converge to the continuous
  initial data in $\E$-norm. We start with $\zeta$ in $\Linfty$-norm,
  \begin{align*}
    \normLinfty{\zeta-\zeta_\itD} &= \sup_\xi  \sumZ{i}{ \left| \frac{\xi_{i+1}-\xi}{\dxi} \int_{\xi_i}^{\xi}\zeta_\xi(\eta)\,d\eta - \frac{\xi-\xi_i}{\dxi} \int_{\xi}^{\xi_{i+1}}\zeta_\xi(\eta)\,d\eta \right| \chi_i(\xi) }\\
                                  &\le \sup_{i \in \Z} \int_{\xi_i}^{\xi_{i+1}} \abs{\zeta_\xi(\eta)}\,d\eta \le \dxi \normLtwo{\zeta _\xi} \xrightarrow{\dxi \to 0} 0.
  \end{align*}
  Next, we consider the $\Ltwo$-norm of $U$,
  \begin{align*}
    \normLtwo{U - U_\itD}^2 &= \sumZ{i}{ \int_{\xi_i}^{\xi_{i+1}} \left( \frac{\xi_{i+1}-\xi}{\dxi} \int_{\xi_i}^{\xi}U_\xi(\eta)\,d\eta - \frac{\xi-\xi_i}{\dxi} \int_{\xi}^{\xi_{i+1}}U_\xi(\eta)\,d\eta \right)^2 d\xi } \\
                            &\le \sumZ{i}{ \int_{\xi_i}^{\xi_{i+1}} \left(\int_{\xi_i}^{\xi_{i+1}} \abs{U_\xi(\eta)} \, d\eta \right)^2 d\xi } \le \dxi^2 \normLtwo{U_\xi}^2 \xrightarrow{\dxi \to 0} 0.
  \end{align*}
  Then, for the $\Ltwo$-norm of $U_\xi$, we have
  \begin{align*}
    \normLtwo{(U_{\itD})_\xi - U_\xi}^2 &= \sumZ{i}{ \int_{\xi_i}^{\xi_{i+1}} \left| \D_+U_i - U_\xi(\xi) \right|^2 d\xi } \\
                                        &=  \sumZ{i}{\int_{\xi_i}^{\xi_{i+1}}  \left( \frac{1}{\dxi} \int_{\xi_i}^{\xi_{i+1}} ( U_\xi(\eta) - U_\xi(\xi) ) \, d\eta  \right)^2 d\xi } \\
                                        &\overset{\mathclap{\text{Jensen}}}{\le} \quad \sumZ{i}{ \int_{\xi_i}^{\xi_{i+1}} \frac{1}{\dxi} \int_{\xi_i}^{\xi_{i+1}} (U_\xi(\eta) - U_\xi(\xi))^2 \, d\eta \, d\xi } \\
                                        &\le \sumZ{i}{ \int_{\xi_i}^{\xi_{i+1}} \frac{1}{\dxi} \int_{-\dxi}^{\dxi} (U_\xi(\xi + z) - U_\xi(\xi))^2 \, dz \, d\xi } \\
                                        &\overset{\mathclap{\text{Tonelli}}}{=} \quad \frac{1}{\dxi} \int_{-\dxi}^{\dxi} \sumZ{i}{ \int_{\xi_i}^{\xi_{i+1}}  (U_\xi(\xi + z) - U_\xi(\xi))^2 \, d\xi } \, dz \\
                                        &= \frac{1}{\dxi} \int_{-\dxi}^{\dxi} \normLtwo{U_\xi(\cdot + z) - U_\xi(\cdot)}^2 d\xi \\
                                        &\le 2 \max_{\abs{z} \le \dxi} \normLtwo{U_\xi(\cdot + z) - U_\xi(\cdot)}^2 \xrightarrow{\dxi \to 0} 0,
  \end{align*}
  where in the final limit we use \cite[Lem.~4.3]{Brezis}. A completely analogous estimate holds for the convergence of  $\zeta_{\Delta, \xi}$ in $\Ltwo$. Considering $\bar{r}_\itD$ we find
  \begin{align*}
    \normLtwo{\bar{r}_\itD - \bar{r}}^2 &= \sumZ{i}{ \int_{\xi_i}^{\xi_{i+1}} \left( \frac{1}{\dxi} \int_{\xi_i}^{\xi_{i+1}}(\bar{r}(\eta) - \bar{r}(\xi))\,d\eta \right)^2 d\xi } \\
                                        &\overset{\mathclap{\text{Jensen}}}{\le} \quad\sumZ{i}{ \int_{\xi_i}^{\xi_{i+1}} \frac{1}{\dxi} \int_{\xi_i}^{\xi_{i+1}} (\bar{r}(\eta) - \bar{r}(\xi))^2 \,d\eta \, d\xi },
  \end{align*}
  and following the proof for $U_\xi$ we find that this also converges. It remains to prove $H_\itD \to H$ in $\V$.  We shall
  first prove that $h_\itD$ converges to $h$ in $\Lone$, and we do it as follows.  For a given $n\in\{1,2,\ldots\}$, we
  consider the partition of $\R$ defined by the points $\xi_{i,n} = i2^{-n}$, which corresponds to $\dxi = 2^{-n}$. In this
  way, each partition is a subdivision of a coarser partition. We denote $h_\itD$ by $h_{\itD_n}$ and similarly for all the
  other variables. We consider the sets $B = \{ \xi\in\R \text{ s.t. } y_\xi(\xi) = 0 \}$,
  \begin{equation*}
    B_n = \{ \xi\in\R \: : \: \text{there exists } i\in\Z, \text{ s.t. }
    \xi \in (\xi_{i, n}, \xi_{i+1, n}) \text{ and } y(\xi_{i+1, n}) = y(\xi_{i, n})\}.
  \end{equation*}
  Let us also define $B^o$ as the union $B^o = \cup_{n\geq 0}B_n$. Since $y$ is increasing, we have $B_n\subset B$. Moreover, as partitions for larger
  $n$ are obtained by further subdivision, we have $B_n\subset B_{n + 1}$. Let $L$ be the set of Lebesgue points for
  $y_\xi$. We know that the set of Lebesgue points have full measure, that is $m(L^c)=0$. We have, for some $i$,
  \begin{equation*}
    \abs{y_{\itD_n,\xi}(\xi) - y_\xi(\xi)} \leq \frac1{\dxi}\int_{\xi_{i,n}}^{\xi_{i + 1, n}}\abs{y_\xi(\xi) - y_\xi(\eta)}\,d\eta\leq\frac1{\dxi}\int_{\xi - \dxi}^{\xi + \dxi}\abs{y_\xi(\xi) - y_\xi(\eta)}\,d\eta,
  \end{equation*}
  which tends to zero for any Lebesgue point $\xi\in L$. We consider a measure $\mu$ such that the singular part does not
  contain any singular continuous part, that is of the form
  \begin{equation}
    \label{eq:musform}
    \mu_s = \sum_{i = 1}^\infty a_i\delta_{x_i},
  \end{equation}
  for a sequence $a_j\geq 0$ such that $\norm{a_j}_{\lone}<\infty$. When $\mu_s$ takes this form, the set $B$ can be written
  as
  \begin{equation*}
    B = \bigcup_{i = 1}^\infty (\gamma_i, \gamma_i + a_i),
  \end{equation*}
  for some values $\gamma_i\in\Real$, for which we do not need explicit expressions in this proof. In this case, we have
  \begin{equation}
    \label{eq:nullset}
    m(B\cap (B^o)^c) = 0.
  \end{equation}
  Indeed, this is a consequence of $m((\gamma_i, \gamma_i + a_i)\cap (B^o)^c) = 0$, which can be proved as follows.  We have
  \begin{equation*}
    (\gamma_i, \gamma_i + a_i)\cap B_n^c \subset (\gamma_i, \gamma_i + \dxi_n)\cup(\gamma_i + a_i - \dxi_n, \gamma_i + a_i),
  \end{equation*}
  and therefore $m((\gamma_i, \gamma_i + a_i)\cap B_n^c )\leq 2\dxi_n$, which yields
  \begin{equation*}
    m((\gamma_i, \gamma_i + a_i)\cap (B^o)^c) = \lim_{n\to\infty} m((\gamma_i, \gamma_i + a_i)\cap B_n^c ) = 0.
  \end{equation*}
  Then, by countable additivity of the measure, we conclude that \eqref{eq:nullset} holds.  This implies $\lim_{n\to\infty}
  \chi_{B_n} = \chi_B\text{ in }\Lone$. The value of $h_n$ is given by
  \begin{equation*}
    h_{\itD_n}(\xi)  = \frac{1}{2}\chi_{B_n}(\xi) + \left( (U_{\itD_n})^2y_{{\itD_n},\xi} + \frac{(U_{\itD_n,\xi})^2}{y_{\itD_n,\xi}} \right)\!\!(\xi) \, \chi_{B_n^c}(\xi),
  \end{equation*}
  while an analogous expression defines $h$. We have
  \begin{equation*}
    \lim_{n\to\infty}\left( (U_{\itD_n})^2y_{{\itD_n},\xi} + \frac{(U_{\itD_n,\xi})^2}{y_{\itD_n,\xi}} \right)\!\!(\xi) = \left( U^2y_{\xi} + \frac{U_{\xi}^2}{y_{\xi}} \right)\!\!(\xi)
  \end{equation*}
  for every $\xi\in (\cup_{n=1}^\infty B_n)^c \cap L$, that is almost everywhere in $B^c$ because of \eqref{eq:nullset}. Now,
  for any $\epsi>0$, there exists a compact $K_1$ such that
  \begin{equation*}
    \norm{h}_{\Lone(K^c_1)}\leq \epsi.
  \end{equation*}
  On the other hand, From \eqref{eq:h_bnd}, we get that
  \begin{equation*}
    h_{\itD_n} \leq U_j^2 y_{\itD_n,\xi} + (U_{\itD,\xi})^2 + \bar{r}_{\itD_n}^2 + C (\zeta_{\itD_n,\xi})^2,
  \end{equation*}
  for $C$ such that $h_{\itD_n} \leq C$ for all $n$. This means that
  \begin{equation*}
    h_{\itD_n} \leq f_n
  \end{equation*}
  for some positive $f_n$. We have already proved that the sequence $f_n$ is convergent in $\Lone$, and we denote by $f$ its
  limit. For any $\epsi>0$, there exists $K_2$ such that $\int_{K^c_2}f\,dx\leq \frac{\epsi}2$. Then
  \begin{equation*}
    \int_{K^c_2} f_n\,dx \leq \norm{f_n - f}_{\Lone} + \frac{\epsi}2.
  \end{equation*}
  so that for $n$ large enough we have
  \begin{equation*}
    \norm{h_{\itD_n}}_{\Lone(K^c_2)}\leq \epsi.
  \end{equation*}
  We take $K = K_1 \cup K_2$ and we have
  \begin{equation*}
    \norm{h_{\itD_n} - h}_{\Lone} = \norm{h_{\itD_n} - h}_{\Lone(K)}  + \norm{h_{\itD_n} - h}_{\Lone(K^c)}\leq \norm{h_{\itD_n} - h}_{\Lone(K)}  + 2\epsi.
  \end{equation*}
  Since $h_n$ is uniformly bounded in $\Linfty$, by the dominated convergence theorem we have $\lim_{n\to\infty}
  \norm{h_{\itD_n} - h}_{\Lone(K)} = 0$ for any given compact $K$. Hence, $\lim_{n\to\infty}h_{\itD_n} = h$ in $\Lone$. Since
  $H_{\itD,\xi} = h_{\itD_n}$ and $H_\xi = h$, the above convergence implies $H_{\itD_n} \to H$ in $\Linfty$.  Moreover, the
  uniform boundedness of $h_{\itD_n}$ together with the estimate $\normLtwo{h_{\itD_n}-h} \le (\normLinfty{h_{\itD_n}} +
  \normlinfty{h}) \normLone{h_{\itD_n} - h}$ proves that $h_{\itD_n} \to h$ in $\Ltwo$ as well.
\end{proof}

In the special case where the initial data of \eqref{eq:2CH} is smooth, that is, $u \in \Hone$ and $u_{x}, \rho -
\rho_{\infty} \in \Ltwo \cap \Linfty$, we can choose $y_j = \xi_j$. Then, $\zeta_j = 0$ and the initial
conditions for \eqref{eq:ODE_disc} can be chosen as $U_j = U(\xi_j)$ and $\rho_j = \rho(\xi_j)$. Then we define initial
values for the auxiliary variables through
\begin{equation*}
  \bar{r}_j = \rho_j - \rho_\infty, \qquad H_j = \dxi\sum_{m=-\infty}^{j-1}\left[U_m^2 + (\D_+U_m)^2 + (\bar{r}_m)^2 \right].
\end{equation*}
Moreover, since in this case $g_{i,j}$, $k_{i,j}$ are Green's functions for $\A{\mathbf{1}} = \B{\mathbf{1}} = \id-\D_-\D_+$,
they can be computed explicitly. Indeed, for $\D_+y_j = 1$ we have
\begin{equation*}
  g_{i,j} = k_{i,j} = \frac{(\lambda^+)^{-|j-i|}}{\sqrt{4+\dxi^2}},
\end{equation*}
with $\lambda^+$ defined in \eqref{eq:greens_eul}. Thus, initially we have the Eulerian Green's sequences as computed in
\cite{Holden2006}.

\subsection*{Acknowledgments:}
S.~T.~G.~ is thoroughly grateful to Katrin Grunert and Helge Holden for their careful reading of the manuscript, and the many
helpful discussions in the course of this work.

\end{document}